\documentclass{amsart}
\chardef\bslash=`\\

\usepackage{amssymb,amsmath,amsfonts,amsthm,amscd,color,stmaryrd,mathrsfs}
\usepackage{hyperref}
\usepackage[all,cmtip,poly]{xy}

\usepackage{silence}
\newtheorem{theorem}{Theorem}[section]
\newtheorem{lemma}[theorem]{Lemma}
\newtheorem{proposition}[theorem]{Proposition}
\newtheorem{corollary}[theorem]{Corollary}

\newtheorem*{conjecture}{Conjecture}

\numberwithin{equation}{section} 
\newtheorem{introthm}{Theorem}

\newif\iffinalrun
\iffinalrun
\else
\fi

\iffinalrun
\newcommand{\need}[1]{}
\newcommand{\mar}[1]{}
\else
\newcommand{\need}[1]{{\tiny *** #1}}
\newcommand{\mar}[1]{\marginpar{\color{red}\raggedright\tiny  #1}}\fi

\renewcommand\mathbb{\mathbf}

\newcommand{\Lie}{{\operatorname{Lie}\,}}

\newcommand{\rec}{{\operatorname{rec}}}

\newcommand{\wv}{{\widetilde{v}}}

\renewcommand{\ell}{l}

\def\PSL{\mathrm{PSL}}
\def\PGL{\mathrm{PGL}}

\def\Iw{\mathrm{Iw}}

\newcommand{\St}{\operatorname{St}}
\newcommand{\ad}{\operatorname{ad}}
\newcommand{\diag}{\operatorname{diag}}
\newcommand{\tr}{\operatorname{tr}}

\newcommand{\bA}{\ensuremath{\mathbf{A}}}

\newcommand{\bC}{\ensuremath{\mathbf{C}}}

\newcommand{\bF}{\ensuremath{\mathbf{F}}}
\newcommand{\bG}{\ensuremath{\mathbf{G}}}

\newcommand{\bP}{\ensuremath{\mathbf{P}}}
\newcommand{\bL}{\ensuremath{\mathbf{L}}}

\newcommand{\bQ}{\ensuremath{\mathbf{Q}}}

\newcommand{\bR}{\ensuremath{\mathbf{R}}}

\newcommand{\bT}{\ensuremath{\mathbf{T}}}

\newcommand{\T}{{\mathbb T}}

\newcommand{\bZ}{\ensuremath{\mathbf{Z}}}

\newcommand{\bbZ}{\ensuremath{\mathbf{Z}}}
\newcommand{\bbQ}{\ensuremath{\mathbf{Q}}}

\newcommand{\cA}{{\mathcal A}}

\newcommand{\cC}{{\mathcal C}}

\newcommand{\cE}{{\mathcal E}}

\newcommand{\cH}{{\mathcal H}}

\newcommand{\cJ}{{\mathcal J}}

\newcommand{\cO}{{\mathcal O}}

\newcommand{\cS}{{\mathcal S}}

\newcommand{\cV}{{\mathcal V}}

\newcommand{\cU}{{\mathcal{U}}}

\newcommand{\m}{\frakm}
\newcommand{\ffrm}{{\mathfrak m}}

\newcommand{\frakm}{\mathfrak{m}}

\newcommand{\frakp}{\mathfrak{p}}
\newcommand{\p}{\frakp}
\newcommand{\frakq}{\mathfrak{q}}
\newcommand{\q}{\frakq}

\DeclareMathOperator{\Fitt}{Fitt}
\DeclareMathOperator{\Aut}{Aut}
\DeclareMathOperator{\val}{val}
\DeclareMathOperator{\Ad}{Ad}
\DeclareMathOperator{\Ann}{Ann}

\DeclareMathOperator{\End}{End}
\DeclareMathOperator{\Ext}{Ext}
\DeclareMathOperator{\Fil}{Fil}

\DeclareMathOperator{\Gal}{Gal}
\newcommand{\GL}{\mathrm{GL}}

\DeclareMathOperator{\Hom}{Hom}

\DeclareMathOperator{\cInd}{c-Ind}
\DeclareMathOperator{\Ind}{Ind}

\DeclareMathOperator{\Mod}{Mod}

\DeclareMathOperator{\SL}{SL}
\DeclareMathOperator{\Sp}{Sp}

\DeclareMathOperator{\Sym}{Sym}

\newcommand{\Frob}{\mathrm{Frob}}

\newcommand{\Art}{{\operatorname{Art}}}
\newcommand{\Res}{\operatorname{Res}}

\newcommand{\doubleslash}{/\kern-0.2em{/}}

\title{A $p$-adic approach to the existence of level-raising congruences}
\author{Jack A. Thorne}

\begin{document}
	\begin{abstract}
    We construct level-raising congruences between $p$-ordinary automorphic representations, and apply this to the problem of symmetric power functoriality for Hilbert modular forms. In particular, we prove the existence of the $n^\text{th}$ symmetric power lift of a Hilbert modular eigenform of regular weight for each odd integer $n = 1, 3, \dots, 25$. 
    
    In a future work with James Newton, we will use these results to establish the existence of the $n^\text{th}$ symmetric power lift for all $n \geq 1$.
	\end{abstract}
	\maketitle

	\tableofcontents
	
	\section{Introduction}
	
	\textbf{Context.} Let $p$ be a prime number, and let $\iota : \overline{\bQ}_p \to \bC$ be an isomorphism. Let $F$ be a CM number field, and let $\pi$ be a regular algebraic automorphic representation of $\GL_n(\bA_F)$. Then (see \cite{Har16, Sch15}) there exists a continuous semisimple representation
	\[ r_\iota(\pi) : G_F \to \GL_n(\overline{\bQ}_p), \]
	determined up to isomorphism by the requirement that for almost every finite place $v \nmid p$ of $F$, we have the local-global compatibility relation
	\[ \mathrm{WD}(r_\iota(\pi)|_{G_{F_v}}) \cong \rec_{F_v}^T(\iota^{-1} \pi_v) \]
	(in other words, compatibility with the Tate-normalised local Langlands correspondence for $\GL_n(F_v)$). By a level-raising congruence, we mean the data of another regular algebraic automorphic representation $\pi'$ of $\GL_n(\bA_F)$ such that there is an isomorphism of residual representations 
	\begin{equation}\label{eqn_intro_iso_of_residual_representations} \overline{r_\iota(\pi)} \cong \overline{r_\iota(\pi')} 
	\end{equation}
	from $G_F$ to $\GL_n(\overline{\bF}_p)$ and a finite place $v$ of $F$ such that $\pi'_v$ is ``more ramified'' than $\pi_v$ (for example, such that $\pi_v$ is unramified but $\pi'_v$ is not, or such that the conductor ideal of $\pi'_v$ is strictly contained in the conductor ideal of $\pi_v$). The existence of an isomorphism (\ref{eqn_intro_iso_of_residual_representations}) is equivalent to the eigenvalues of the unramified Hecke operators on $\pi$ and $\pi'$, viewed as elements of $\overline{\bZ}_p$, being congruent modulo the maximal ideal of $\overline{\bZ}_p$, while for classical modular forms the conductor ideal corresponds to the level of the corresponding congruence subgroup of $\mathrm{SL}_2(\mathbb{Z})$. This explains the terminology `level-raising congruence', which goes back to Ribet's fundamental paper \cite{Rib84}. 
	
	\textbf{Results.} In this paper, we prove the existence of level-raising congruences for regular algebraic automorphic representations $\pi$ of $\GL_n(\bA_F)$ which are conjugate self-dual, and even `$\theta$-discrete', in the sense that they are of the form $\pi = \pi_1 \boxplus \dots \boxplus \pi_r$, where $\pi_1, \dots, \pi_r$ are conjugate self-dual, cuspidal (and therefore unitary) automorphic representations of lower rank general linear groups. This is a theme explored in our earlier paper \cite{Tho14}. Here we use similar techniques to prove the following theorem, a special case of Theorem \ref{thm_level_raising_for_GL(n)}:
\begin{introthm}\label{thm_intro_thm}
Fix a partition $n = n_1 + n_2$ and let $\pi_1, \pi_2$ be cuspidal, conjugate self-dual automorphic representations of $\GL_{n_1}(\bA_F)$, $\GL_{n_2}(\bA_F)$ such that $\pi = \pi_1 \boxplus \pi_2$ is regular algebraic and $\iota$-ordinary. Suppose that the following conditions are satisfied:
\begin{enumerate}
    \item $F$ contains an imaginary quadratic field $F_0$ in which $p$ splits, and if $w$ is a place of $F$ such that $\pi_w$ is ramified, then the prime number lying below $w$ splits in $F_0$. The extension $F / F^+$ (where $F^+$ denotes the maximal totally real subfield of $F$) is everywhere unramified. 
    \item The weights $\lambda_1 \in (\bbZ^{n_1})^{\Hom(F, \bC)}, \lambda_2 \in (\bbZ^{n_2})^{\Hom(F, \bC)}$ of $\pi_1, \pi_2$ are parallel, in the sense that if $\tau, \tau' : F \to \bC$ have the same restriction to $F_0$, then $\lambda_{1, \tau} = \lambda_{1, \tau'}$ and $\lambda_{2, \tau} = \lambda_{2, \tau'}$.
    \item There exists a prime-to-$p$ place $w_0$ of $F$ and unramified characters $\psi_1, \psi_2 : F_{w_0}^\times \to \bC^\times$ such that $\pi_{1, w_0} \cong \St_{n_1}(\psi_1)$, $\pi_{2, w_0} \cong \St_{n_2}(\psi_2)$, and the character $\iota^{-1}( \psi_1 \psi_2^{-1} | \cdot |^{-n/2} )$ takes values in $1 + \m_{\overline{\bZ}_p}$. Moreover, if $w_0$ lies above the prime number $q$, then $\pi_1, \pi_2$ are both unramified at each $q$-adic place $w \neq w_0, w_0^c$ of $F$.
        \item\label{ass_intro_generic} There exists a finite place $w_1 \nmid S$ of $F$, split over $F^+$, such that there is an isomorphism $\overline{r_\iota(\pi)}|_{G_{F_{w_1}}}^{ss} \cong \oplus_{i=1}^n \overline{\chi}_i$, where $\overline{\chi}_1, \dots, \overline{\chi}_n$ are unramified characters such that if $i \neq j$ then $\overline{\chi}_i / \overline{\chi}_j \neq \epsilon$. 
\end{enumerate}
Then there exists a regular algebraic, conjugate self-dual, cuspidal, $\iota$-ordinary automorphic representation $\Pi$ of $\GL_n(\bA_F)$ such that $\overline{r_\iota(\Pi)} \cong \overline{r_\iota(\pi)}$ and $\Pi_{w_0}$ is an unramified twist of $\St_n$. 
\end{introthm}
	(The case where $\pi$ has the form $\pi = \pi_1 \boxplus \dots \boxplus \pi_r$ (and $r > 2$) may be treated by induction on $r$.) This strengthens \cite[Theorem 7.1]{Tho14}, the main improvements being that the following two hypotheses are no longer required:
	\begin{itemize}
		\item[(a)] We no longer require that $p$ is a banal characteristic for $\GL_n(F_{w_0})$, in the sense that $p$ is prime to order of $\GL_n(\cO_{F_{w_0}} / (\varpi_{w_0}))$. 
	    \item[(b)] We no longer require that the restriction  $\overline{r_\iota(\pi_i|\cdot|^{(n_i-n)/2)})}|_{G_{F_{w_0}}}$ of the residual representation to the decomposition group is as ramified as possible, in the sense that the image of a generator $t_{w_0}$ for the tame inertia group at $w_0$ has a single Jordan  block.
	\end{itemize}
	These improvements are essential for the applications to symmetric power functoriality outlined below. Indeed, the possibility of such applications was the motivation in \cite{Tho14}, but the results proved there were too weak to be of any use. 
	
	Before discussing these applications, we outline the main ideas behind the proof of Theorem \ref{thm_intro_thm}. The main principle is that if the situation modulo $p$ is too degenerate, we should try to lift to a characteristic 0 ($p$-adic) situation where these degeneracies disappear. Here we achieve this by considering a twisted representation
	\[ \rho_\psi = r_\iota(\pi_1|\cdot|^{-n_2/2}) \otimes \psi \oplus r_\iota(\pi_2|\cdot|^{-n_1/2}) \]
	for some conjugate self-dual $p$-adic character $\psi : G_F \to \overline{\bQ}_p^\times$ which is unramified at $w_0$, and such that $\rho_\psi|_{G_{F_{w_0}}}$ satisfies the necessary condition for level-raising, i.e.\ that the characteristic polynomial of $\rho_\psi(\Frob_{w_0})$ has the form $\prod_{i=1}^n(X - q_{w_0}^i \alpha)$ for some $\alpha \in \overline{\bQ}_p^\times$ (as it would if $\rho_\psi|_{G_{F_{w_0}}}$ was in fact associated, under the local Langlands correspondence, to an unramified twist of the Steinberg representation). One now sees why the hypotheses (a) and (b) above fall away -- characteristic 0 is certainly banal for $\GL_n(F_{w_0})$, and the condition on the Jordan blocks of $\rho_\psi(t_{w_0})$ holds because it holds for $r_\iota(\pi_1|\cdot|^{-n_2/2})|_{G_{F_{w_0}}}$ and $r_\iota(\pi_2|\cdot|^{-n_1/2})|_{G_{F_{w_0}}}$ (by the strong form of local-global compatibility at ${w_0}$, established in \cite{Tay07}).
	
	The representation $\rho_\psi$ is not associated to a classical automorphic representation of $\GL_n(\bA_F)$, but it can be seen as arising from a space of non-classical $p$-ordinary automorphic forms. Our hypotheses ensure that we can find such an automorphic form in a space $\cH_\lambda$ of $p$-ordinary weight $\lambda$ automorphic forms on the $\bA_{F^+}$-points of a definite unitary group $G$ over $F^+$ (for a suitable $p$-adic, non-classical weight $\lambda$). For us $\cH_\lambda$ is an admissible $\overline{\bbZ}_p[\GL_n(F_{w_0})]$-module, flat over $\overline{\bZ}_p$, such that the irreducible subquotients of $\cH_\lambda[1/p]$ are isomorphic to subquotients of the induced representation
	\begin{equation}\label{eqn_parabolic_induction} \Ind_{B_n(F_{w_0})}^{\GL_n(F_{w_0})} \overline{\bQ}_p = C^\infty( B_n(F_{w_0}) \backslash \GL_n(F_{w_0}), \overline{\bQ}_p) 
	\end{equation}
	from the Borel subgroup $B_n \subset \GL_n$. We are thus in the situation of \cite[\S 4]{Tho14}, which gives an abstract criterion on an admissible $k[\GL_n(F_{w_0})]$-module, where $k$ is a field of banal characteristic, to admit level-raising (concretely, to admit the Steinberg representation as a quotient). The key property of $\cH_\lambda$ that we require in \emph{loc. cit.} is that the homology groups 
	\[ H_i(\GL_n(F_{w_0})^0, \cH_\lambda[1/p]) \] (where $\GL_n(F_{w_0})^0 \subset \GL_n(F_{w_0})$ is the subgroup of matrices whose determinant lies in $\cO_{F_{w_0}}^\times$) are non-zero only in degree $i = n-1$. To prove this, it is enough to show that the homology groups 
	\[ H_i(\GL_n(F_{w_0})^0, \cH_\lambda \otimes_{\overline{\bZ}_p} \overline{\bF}_p) \]
	are concentrated in this degree. The difficulty faced in \cite{Tho14} is that these groups are hard to compute in non-banal characteristic. To get around this here, we introduce the complex $R \Gamma_c( \Omega_{F_{w_0}}^{n-1, ca}, \overline{\bZ}_p)$ studied in \cite{Dat06} -- the compactly supported \'etale cohomology of the Drinfeld upper half space $\Omega_{F_{w_0}}^{n-1}$. (Following \cite{Dat06}, we use the superscript $ca$ to denote base change to the completion of an algebraically closed extension.) It is a complex of smooth $\overline{\bZ}_p[\GL_n(F_{w_0})]$-modules, and according to \cite{Dat06} there is an isomorphism
	\[ R \Gamma_c( \Omega_{F_{w_0}}^{n-1, ca}, \overline{\bQ}_p) \cong \oplus_{i=0}^{n-1} \pi_{i, 1, \dots, 1}[i-2(n-1)] \]
	in the derived category of smooth $\overline{\bQ}_p[\GL_n(F_{w_0})]$-modules (and where the representations $\pi_\mu$ associated to an ordered partition $\mu$ of $n$ are certain subquotients of the representation (\ref{eqn_parabolic_induction}), whose definition we recall in \S \ref{sec_elliptic_representations}). In particular, this complex is split, in the sense that it is isomorphic (in the derived category) to the direct sum of its cohomology groups. (Such a splitting exists also when $\overline{\bQ}_p$ is replaced by a finite field of banal characteristic, but not in general, cf.\ \cite{Dat14}.)
	
	It is enough therefore for us to consider the derived tensor product
	\[ R \Gamma_c( \Omega_{F_{w_0}}^{n-1, ca}, \overline{\bQ}_p) \otimes^{\bL}_{\overline{\bQ}_p[\GL_n(F_{w_0})^0]} \cH_\lambda[1/p]. \]
	The cohomology of this complex will be concentrated in a single degree provided the same is true of the complex
	\begin{equation}\label{eqn_intro_tensor_product} R \Gamma_c (\Omega_{F_{w_0}}^{n-1, ca}, \overline{\bF}_p) \otimes^{\bL}_{\overline{\bF}_p[\GL_n(F_{w_0})^0]} (\cH_\lambda / \ffrm_{\overline{\bZ}_p} \cH_\lambda). 
	\end{equation}
	Now we use that $\cH_\lambda / \ffrm_{\overline{\bZ}_p} \cH_\lambda$ may be identified with a space of classical, $p$-ordinary automorphic forms on $G$ with coefficients in $\overline{\bF}_p$. Using results of Rapoport--Zink \cite{Rap96} on the $p$-adic uniformization of Shimura varieties, we identify the complex (\ref{eqn_intro_tensor_product}) with the $\overline{\bF}_p$-cohomology of a Shimura variety of Harris--Taylor type \cite{Har01}. The results of Boyer \cite{Boy19} and  Caraiani--Scholze \cite{Car17}  imply that this cohomology vanishes outside degree $n-1$ (after localisation a generic maximal ideal of the Hecke algebra). 	
	
	The reader familiar with existing level-raising arguments may wonder where ``Ihara's lemma'' (as formulated in e.g.\ \cite{Clo08}) appears in the argument. In \cite{Tho14}, we deduced Ihara's lemma (in restricted circumstances) from the concentration of certain cohomology groups in a single degree. By contrast, here we deduce our level-raising result directly from this concentration of cohomology groups; it does not seem that this concentration alone is strong enough to imply Ihara's lemma in the generality required to then be able to deduce Theorem \ref{thm_intro_thm} as a consequence.
		
	The appeal to the results of \cite{Boy19, Car17} is the reason that Theorem \ref{thm_intro_thm} contains the genericity hypothesis (\ref{ass_intro_generic}) on the residual representation $\overline{r_\iota(\pi)}$. In fact, we work throughout with the $p$-ordinary part of cohomology, and it seems likely to us that the $p$-ordinary part of the cohomology of the Shimura varieties considered here is already concentrated in a single degree before localisation at a generic maximal ideal of the prime-to-$p$ Hecke algebra. We don't know how to prove this, but such a result would make it possible to remove condition (\ref{ass_intro_generic}) from Theorem \ref{thm_intro_thm}, leading to a theorem which would be essentially optimal for the kinds of applications considered here.
	
	\textbf{Applications.} Here is our first main application.
	\begin{introthm}\label{introthm_unconditional_symmetric_powers} Let $F^+$ be a totally real number field, and let $\pi$ be a cuspidal automorphic representation of $\GL_2(\bA_{F^+})$ such that $\pi_\infty$ is (essentially) square-integrable and which is not automorphically induced from a quadratic extension of $F^+$. Then for each integer $n = 1, \dots, 9$, and for each odd integer $n = 11, \dots, 25$, the symmetric power $\Sym^n \pi$ exists, in the sense that there is a cuspidal automorphic representation $\Pi$ of $\GL_{n+1}(\bA_{F^+})$ such that for any place $v$ of $F^+$, we have
	\[ \rec_{F^+_v}(\Pi_v) \cong \Sym^n \circ \, \rec_{F^+_v}(\pi_v). \]
	\end{introthm}
	For a description of the importance of, and of previous progress on, the problem of symmetric power functoriality, we refer to the papers \cite{Clo14, New21}. In our previous work joint work with Clozel \cite{Clo14}, we outlined a programme to prove the existence of symmetric power functorial lifts of automorphic representations of $\GL_2(\bA_{F^+})$ of the type considered in Theorem \ref{introthm_unconditional_symmetric_powers}, relying on two conjectures: one on the existence of certain level-raising congruences, and the second on the existence of a version of  tensor product functoriality for the group $\GL_2 \times \GL_r$. The results we establish in this paper, although not identical in strength to the level-raising conjecture stated in \cite{Clo14}, are nevertheless sufficient to eliminate the need to assume this conjecture. Theorem \ref{introthm_unconditional_symmetric_powers} is then what one can prove unconditionally using the ideas in \cite{Clo14} by using known cases $r = 1, 2, 3$ of tensor product functoriality. We can also state a result conditional on the assumption of tensor product functoriality in all degrees: 
	\begin{introthm}\label{introthm_conditional_symmetric_powers}
	Let $F^+$ be a totally real number field. Assume the truth of \cite[Conjecture 3.2]{Clo14}. Then for any cuspidal automorphic representation $\pi$ of $\GL_2(\bA_{F^+})$  such that $\pi_\infty$ is (essentially) square-integrable and which is not automorphically induced from a quadratic extension of $F^+$, and for any integer $n \geq 1$, the symmetric power $\Sym^n \pi$ exists, in the sense that there is a cuspidal automorphic representation $\Pi$ of $\GL_{n+1}(\bA_{F^+})$ such that for any place $v$ of $F^+$, we have
	\[ \rec_{F^+_v}(\Pi_v) \cong \Sym^n \circ \,\rec_{F^+_v}(\pi_v). \]
	\end{introthm}
	After the first draft of this paper was completed, we further developed (together with James Newton) the strategy of \cite{Clo14} by proving enough cases of \cite[Conjecture 3.2]{Clo14} to get the conclusion of Theorem \ref{introthm_conditional_symmetric_powers} unconditionally. These results will appear elsewhere \cite{New22a}. We have chosen to leave the statements and proofs of Theorem \ref{introthm_unconditional_symmetric_powers} and Theorem \ref{introthm_conditional_symmetric_powers} in their current form.

	\textbf{Acknowledgements.} The author's work received funding from the European Research Council (ERC) under the European Union’s Horizon 2020 research and innovation programme (grant agreement No. 714405). We would like to thank Toby Gee and James Newton for useful comments on an earlier draft of this paper. 
		
	\textbf{Notation.}  If $F$ is a field of characteristic zero, we generally fix an algebraic closure $\overline{F} / F$ and write $G_F$ for the absolute Galois group of $F$ with respect to this choice. If $F$ is a number field, then we will also fix embeddings $\overline{F} \to \overline{F}_v$ extending the map $F\to F_v$ for each place $v$ of $F$; this choice determines a homomorphism $G_{F_v} \to G_F$. When $v$ is a finite place, we will write $\cO_{F_v} \subset F_v$ for the valuation ring, $\varpi_v \in \cO_{F_v}$ for a fixed choice of uniformizer, $\Frob_v \in G_{F_v}$ for a fixed choice of geometric Frobenius lift, $k(v) = \cO_{F_v} / (\varpi_v)$ for the residue field, and $q_v = \# k(v)$ for the cardinality of the residue field. 
	
	If $F$ is a CM field (totally imaginary extension of a totally real field $F^+$), then we write $c \in \Gal(F / F^+)$ for the non-trivial element. If $S$ is a finite set of finite places of $F^+$ which are unramified in $F$, then we write $F_S / F$ for the maximal subextension of $\overline{F}$ which is unramified over $F^+$, and $G_{F, S} = \Gal(F_S / F)$.

    If $p$ is a prime, then we call a \emph{coefficient field} a finite extension $E / \bbQ_p$ contained inside our fixed algebraic closure $\overline{\bbQ}_p$, and write $\cO$ for the valuation ring of $E$, $\varpi \in \cO$ for a fixed choice of uniformizer, and $k = \cO / (\varpi)$ for the residue field.

If $G$ is a locally profinite group and $U \subset G$ is an open compact subgroup, then we write $\cH(G, U)$ for the set of compactly supported, $U$-biinvariant functions $f : G \to \bZ$. It is a $\bbZ$-algebra, where convolution is defined using the left-invariant Haar measure normalized to give $U$ measure 1; see \cite[\S 2.2]{New16}. It is free as a $\bbZ$-module, with basis given by the characteristic functions $[UgU]$ of double cosets.

Let $K$ be a non-archimedean characteristic $0$ local field of normaliser $\varpi_K$ and residue field $k_K$, and let $| \cdot |_K : K^\times \to \bR_{>0}$ be the absolute value satisfying $| \varpi_K |_K = | k_K |^{-1}$. We write $W_K \subset G_K$ for the Weil group of $K$ and $I_K \subset W_K$ for the inertia subgroup. If $\rho : G_K \to \GL_n(\overline{\bbQ}_p)$ is a continuous 
representation (assumed to be de Rham if $p$ equals the residue characteristic 
of 
$K$), then we write $\mathrm{WD}(\rho) = (r, N)$ for the associated 
Weil--Deligne representation, and $\mathrm{WD}(\rho)^{F-ss}$ for its 
Frobenius semisimplification. We use the cohomological normalisation of 
class field theory: it is the isomorphism $\Art_K: K^\times \to W_K^{ab}$ which 
sends uniformizers to geometric Frobenius elements. The local Langlands correspondence $\rec_K$ is a bijection between the set of isomorphism classes of irreducible admissible $\bC[\GL_n(K)]$-modules and the set of isomorphism classes of Frobenius-semisimple Weil--Deligne representations $(r, N)$ of $W_K$ of rank $n$ over $\bC$. The Tate normalisation $\rec_{K}^T$ is defined by $\rec_K^T(\pi) = \rec_K(\pi | \cdot |^{(1-n)/2})$; it respects the action of $\Aut(\bC)$ on the set of isomorphism classes, so makes sense over any field $\Omega$ which is abstractly isomorphic to $\bC$ (such as $\overline{\bQ}_p$). 

If $P \subset \GL_n$ is a parabolic subgroup, $\Omega$ is a field, and $\pi$ is a smooth $\Omega[\GL_n(K)]$-module, then we write $\Ind_{P(K)}^{\GL_n(K)} \pi$ for the smooth induction: the set of locally constant functions $f : \GL_n(K) \to \pi$ such that for all $p \in P(K)$, $g \in \GL_n(K)$, $f(pg) = \pi(p) f(g)$. If $\Omega = \bC$ then we write $i_{P(K)}^{\GL_n(K)} \pi = \Ind_{P(K)}^{\GL_n(K)} (\pi \otimes \delta_P^{1/2})$ for the normalised induction, where $\delta_P(p) = | \det( \Ad(p) : \Lie N \to \Lie N ) |_K$. If $\sigma$ is a smooth $\bC[\GL_n(K)]$-module, then we write $r_P(\sigma) = (\sigma)_{N(K)} \otimes \delta_P^{-1/2}$ for the normalised restriction (twist of $N(K)$-coinvariants, where $N \subset P$ is the unipotent radical). 

By definition, the Steinberg representation $\St_n$ of $\GL_n(K)$ is the unique irreducible quotient of $\Ind_{B_n(K)}^{\GL_n(K)} \bC$, where $B_n \subset \GL_n$ is the upper-triangular Borel subgroup. It satisfies $\rec_K( \St_n ) = \operatorname{Sp}_n$, where $\operatorname{Sp}_n = (r, N)$ is the Weil--Deligne representation on $\bC^n = \oplus_{i=1}^n \bC e_i$ defined by $r(\Art_K(x))(e_i) = | x |^{(n+1-2i)/2} e_i$, $N e_i = e_{i-1}$ for each $i = 1, \dots, n$ (and where by convention $e_0 = 0$). 

Now let $K$ be a finite extension of $\bR$. We again write $W_K$ for the Weil group of $K$. In this case the local Langlands correspondence $\rec_K$ is a bijection between the set of isomorphism classes of infinitesimal equivalence classes of irreducible admissible representations of $\GL_n(K)$ over $\bC$ and the set of isomorphism classes of semisimple representations of $W_K$ of rank $n$ over $\bC$. We again define $\rec^T_K(\pi) = \rec_K(\pi| \cdot |^{(1-n)/2})$. These notions are reviewed in more detail in \cite[\S 2.1]{Clo14}.

If $F$ is a number field and $\pi$ is an automorphic representation of 
$\GL_n(\bA_F)$, we say that $\pi$ is regular algebraic if $\pi_\infty$ has the 
same infinitesimal character as an irreducible algebraic representation of 
$\Res_{F/\bQ}\GL_n$. Let $\bZ^n_+ \subset \bZ^n$ denote the set of tuples $(\lambda_1, \dots, \lambda_n)$ such that $\lambda_1 \geq \dots \geq \lambda_n$. We identify $\bZ^n_+$ with the set of characters of the diagonal maximal torus $T_n \subset \GL_n$ which are dominant with respect to the upper-triangular Borel subgroup $B_n \subset \GL_n$. If $\pi$ is regular algebraic, we say that it is of weight $\lambda = (\lambda_\tau)_{\tau \in \Hom(F, \bC)} \in (\bZ^n_+)^{\Hom(F, \bC)}$ if for each place $v | \infty$ of $F$, $\pi_v$ has the same infinitesimal character as the dual of the tensor product of the representations of $\GL_n(F_v)$ of highest weights $\lambda_\tau$ ($\tau \in \Hom(F_v, \bC)$). 

If $\chi : F^\times \backslash \bA_F^\times \to \bC^\times$ is a Hecke character which is regular algebraic (equivalently: algebraic), then for any isomorphism $\iota : \overline{\bQ}_p \to \bC$ there is a continuous character $r_\iota(\chi) : G_F \to \overline{\bQ}_p^\times$ which is de Rham at the places $v | p$ of $F$ and such that for each finite place $v$ of $F$, $\mathrm{WD}(r_{\chi, \iota}) \circ \Art_{F_v} = \iota^{-1} \chi|_{F_v^\times}$. Conversely, if $\chi' : G_F \to \overline{\bQ}_p^\times$ is a continuous character which is de Rham and unramified at all but finitely many places, then there exists an algebraic Hecke character $\chi : F^\times \backslash \bA_F^\times \to \bC^\times$ such that $r_\iota(\chi) = \chi'$. We write $\chi = \iota \chi'$.

Let $F^+$ be a totally real field. By definition, a RAESDC automorphic representation of $\GL_n(\bA_{F^+})$ is a pair $(\pi, \chi)$, where $\pi$ is a regular algebraic cuspidal automorphic representation of $\GL_n(\bA_{F^+})$ and $\chi : (F^+)^\times \backslash (\bA_{F^+})^\times \to \bC^\times$ is a continuous algebraic character such that $\chi_v(-1)$ is independent of the choice of place $v | \infty$ of $F^+$ and $\pi\cong \pi^\vee \otimes (\chi \circ \det)$. 

Now let $F$ be a CM field. By definition, a RAECSDC automorphic representation of $\GL_n(\bA_F)$ is a pair $(\pi, \chi)$, where $\pi$ is a regular algebraic cuspidal automorphic representation of $\GL_n(\bA_F)$ and $\chi : (F^+)^\times \backslash (\bA_{F^+})^\times \to \bC^\times$ is a continuous algebraic character such that $\chi_v(-1) = (-1)^n$ for each place $v | \infty$ of $F^+$ and $\pi \cong \pi^{c, \vee} \otimes (\chi \circ \mathbf{N}_{F / F^+} \circ \det)$. A RACSDC automorphic representation $\pi$ of $\GL_n(\bA_F)$ is a regular algebraic cuspidal automorphic representation $\pi$ such that $\pi \cong \pi^{c, \vee}$. 

Let $p$ be a prime number and let $\iota : \overline{\bQ}_p \to \bC$ be an isomorphism. Any RACSDC or RAECSDC (resp. RAESDC) automorphic representation $\pi$ has an associated Galois representation $r_\iota(\pi) : G_F \to \GL_n(\overline{\bQ}_p)$ (resp. $r_\iota(\pi) : G_{F^+} \to \GL_n(\overline{\bQ}_p)$) which satisfies $\mathrm{WD}(r_\iota(\pi)|_{G_{F_v}})^{F-ss} \cong \rec_{F_v}^T(\pi_v)$ for any place $v \nmid p \infty$ of $F$ (resp. similarly with $F$ replaced by $F^+$). This notation can be applied more generally to regular algebraic automorphic representations of $\GL_n(\bA_F)$ of the form $\pi = \pi_1 \boxplus \dots \boxplus \pi_r$, where each $\pi_i$ is cuspidal and conjugate self-dual. See \cite[Theorem 2.2]{Clo14} for more details.

	\section{Homology of smooth representations}
	
	Let $p$ be a prime, let $G$ be a topological group containing an open compact subgroup which is a pro-$p$ group containing only countably many open compact subgroups, and let $R$ be a $\bbZ[1/p]$-algebra. We write $\Mod_{sm}(R[G])$ for the category of smooth (left) $R[G]$-modules. Our main examples will be $G = \GL_n(F_v)$, where $F_v / \bbQ_p$ is a finite extension, $G = \Gamma$, where $\Gamma$ is a discrete group, and products of these. 
	
	If $X$ is a totally disconnected, locally compact space (in other words, $X$ is a Hausdorff topological space such that each point has a basis of open compact neighbourhoods), we write $\cC_c^\infty(X, R)$ for the set of locally constant and compactly supported functions $f : X \to R$. If $G$ acts on $X$ (and the action map $G \times X \to X$ is continuous), then $\cC_c^\infty(X, R) \in \Mod_{sm}(R[G])$. (In this paper we will consider only left actions on spaces and modules; thus an element $g \in G$ acts on $F \in \cC_c^\infty(X, R)$ by the formula $(g \cdot F)(h) = F(g^{-1} h)$.) If $f \in C_c^\infty(G, R)$ and $g \in G$ then we set $L_g(f)(h) = f(g^{-1} h)$ and $R_g f(h) = f ( hg )$. Thus $C_c^\infty(G, R)$ lies in $\Mod_{sm}(R[G \times G])$. When we wish to consider only one of these actions, we write $C_c^\infty(G_l, R)$ or $C_c^\infty(G_r, R) \in \Mod_{sm}(R[G])$ to indicate that we are considering only the action of $G$ either by left or by right translation. 
\begin{proposition}\label{prop_mod_G_enough_projectives}
The category $\Mod_{sm}(R[G])$ is abelian and has enough injectives and enough projectives. 
\end{proposition}
\begin{proof}
It is a standard fact that $\Mod_{sm}(R[G])$ is abelian. To see that this category has enough injectives, we recall the statement of Frobenius reciprocity: if $H \leq G$ is a closed subgroup, then there is a smooth induction functor \[ \Ind_H^G : \Mod_{sm}(R[H]) \to \Mod_{sm}(R[G]),\] sending a smooth $H$-module $M$ to the set of functions $f : G \to M$ that satisfy $f(hg) = h f(g)$ for all $h \in H$, $g \in G$, and that are invariant under right translation by some open compact subgroup of $G$. The functor $\Ind_H^G$ has an exact left adjoint, namely restriction to $H$. In particular, if $M \in \Mod_{sm}(R[G])$, we choose an embedding $M \to N$ in an injective $R$-module. Then $\Ind_1^G N$ is an injective object of $\Mod_{sm}(R[G])$ and there is, by adjunction, an embedding $M \to \Ind_1^G N$.

We next show that the category has enough projectives. Let $G_r$ denote $G$, with $G$ acting on itself by right translation. We claim that $\cC_c^\infty(G_r, R)$ is a projective object of $\Mod_{sm}(R[G])$. To see this, fix a decreasing sequence $K_1 \geq K_2 \geq \dots$ of open compact pro-$p$ subgroups of $G$ with trivial intersection. Let 
\[ \int_{g \in G} f(g) \, d\mu : C_c^\infty(R) \to R \]
denote the $R$-linear map which, for any $g \in G$ and open compact subgroup $K \leq K_1$, sends the indicator function $\mathbf{1}_{gK}$ to $[K : K \cap K_1] [ K_1 : K \cap K_1]^{-1}$. (If $R = \bR$, then this is integration against a Haar measure -- see \cite{Vig96}.) We make $\cC_c^\infty(G, R)$ into an $R$-algebra by the formula
\[ (f_1 \cdot f_2)(g) = \int_{x \in G} f_1(x) f_2(x^{-1} g) \, dx. \]
The algebra $\cC_c^\infty(G, R)$ contains the idempotent elements $e_{K_i} = [K_1 : K_i]^{-1} \mathbf{1}_{K_i}$, which satisfy $e_{K_i} e_{K_j} = e_{K_i}$ if $i \leq j$. 

There are compact induction functors 
\[ \cInd_{K_i}^G : \Mod_{sm}(R[K_i]) \to \Mod_{sm}(R[G]), \]
sending a smooth $K_i$-module $M$ to the set of compactly supported and locally constant functions $f : G \to M$ satisfying $f(kg) = k f(g)$ for all $k \in K_i$, $g \in G$. This functor has an exact right adjoint, namely restriction to $K_i$. In particular, $\cC^\infty_c(K_i \backslash G, R) = \cInd_{K_i}^G R$ is projective, since the functor of passage to $K_i$-invariants is exact. 

To see that $\cC_c^\infty(G_r, R)$ is projective, we note that there is a surjective map $\oplus_{i=1}^\infty \cC^\infty_c(K_i \backslash G, R) \to \cC_c^\infty(G_r, R)$, $(f_i)_i \mapsto \sum_i f_i$. This map has a splitting given by the formula 
\[ f \mapsto (e_{K_1} f, e_{K_2} f - e_{K_1} f, e_{K_3} f - e_{K_2} f, \dots). \]
Thus $\cC_c^\infty(G_r, R)$ is a direct summand of a projective object, and is therefore itself projective. To see that our category has enough projectives, take $M \in \Mod_{sm}(R[G])$ and choose a free $R$-module $F$ and a surjection $p : F \to M$ of $R$-modules. Then the map 
\[ \cC_c^\infty(G_r, R) \otimes_R F \to M \]
\[ (f, x) \mapsto \int_{g \in G} f(g) \cdot p(x) \, dg \]
is a surjective morphism of smooth $R[G]$-modules. This shows that $\Mod_{sm}(R[G])$ has enough projectives. 
\end{proof}
\begin{proposition}\label{prop_restriction_of_projective}
Let $H$ be a closed subgroup of $G$. Then the functor $\Mod_{sm}(R[G]) \to \Mod_{sm}(R[H])$ given by restriction to $H$ preserves projectives. 
\end{proposition}
\begin{proof}
The proof of Proposition \ref{prop_mod_G_enough_projectives} shows that any projective object of $\Mod_{sm}(R[G])$ is a direct summand of $\oplus_{i \in I} \cC_c^\infty(G_r, R)$ for some index set $I$. It's therefore enough to show that $\cC_c^\infty(G_r, R)$ is projective in $\Mod_{sm}(R[H])$. The proof of Proposition \ref{prop_mod_G_enough_projectives} also shows that $\cC_c^\infty(G_r, R)$ is a direct summand of $\oplus_{i \in I} \cC_c^\infty(K_i \backslash G, R)$ for some index set $I$ and collection $(K_i)_{i \in I}$ of open, compact, pro-$p$ subgroups of $G$. It's therefore enough to show that $\cC_c^\infty(K \backslash G, R)$ is projective in $\Mod_{sm}(R[H])$ whenever $K$ is an open, compact, pro-$p$ subgroup of $G$.

Let $(g_j)_{j \in J}$ be a set of representatives for the double cosets $K \backslash G / H$. We claim that there is an isomorphism 
\[ \cC_c^\infty(K \backslash G, R) \cong \oplus_{j \in J} \cC_c^\infty( (g_j^{-1} K g_j \cap H) \backslash H, R) \]
in $\Mod_{sm}(R[H])$. Since each group $g_j^{-1} K g_j \cap H$ is an open, compact, pro-$p$ subgroup of $H$, this will imply that $\cC_c^\infty(K \backslash G, R)$ is indeed projective in $\Mod_{sm}(R[H])$. To see the claim, we note that there is an $H$-invariant decomposition $G = \sqcup_{j \in J} K g_j H$, and each $K g_j H$ is an open subset of $G$. We therefore need to show that there is an $H$-equivariant homeomorphism $K \backslash K g_j H \cong g_j^{-1} K g_j \cap H \backslash H$. This is elementary. 
\end{proof}
We write $\mathbf{D}(R[G])$ for the derived category of cochain complexes in $\Mod_{sm}(R[G])$, and $\mathbf{D}^-(R[G])$ for its subcategory of complexes which are cohomologically bounded above, i.e. those $A^\bullet$ such that $H^i(A^\bullet) = 0$ for all $i$ sufficiently large.  Since $\Mod_{sm}(R[G])$ has enough projectives, any object of $\mathbf{D}^-(R[G])$ is isomorphic to a bounded above complex of projectives. 

If $G_1, G_2, G_3$ are groups satisfying our conditions, then we define a bifunctor
\[ - \otimes_{R[G_2]} - : \Mod_{sm}(R[G_1 \times G_2]) \times  \Mod_{sm}(R[G_2 \times G_3]) \to \Mod_{sm}(R[G_1 \times G_3]) \]
by the formula 
\[ M \otimes_{R[G_2]} N = (M \otimes_R N) / \langle g_2 m \otimes g_2 n - m \otimes n \mid g_2 \in G_2, m \in M, n \in N \rangle. \]
Given complexes $A^\bullet \in \mathbf{D}^-(R[G_1 \times G_2]), B^\bullet \in \mathbf{D}^-(R[G_2 \times G_3])$, we define
\[ A^\bullet \otimes^{\bL}_{R[G_2]} B^\bullet \in \mathbf{D}^-(R[G_1 \times G_3]) \]
as follows: choose quasi-isomorphisms $P^\bullet \to A^\bullet$, $Q^\bullet \to B^\bullet$, where $P^\bullet, Q^\bullet$ are bounded-above complexes of projective objects. Then $A^\bullet \otimes^{\bL}_{R[G_2]} B^\bullet = \operatorname{Tot}(P^\bullet \otimes_{R[G_2]} Q^\bullet)$. This defines the functor of total tensor product: 
\begin{proposition}
\begin{enumerate}
    \item $- \otimes^{\bL}_{R[G_2]} -$ is a functor 
    \[ \mathbf{D}^-(R[G_1 \times G_2]) \times \mathbf{D}^-(R[G_2 \times G_3]) \to \mathbf{D}^-(R[G_1 \times G_3]). \]
    \item For $A^\bullet \in \mathbf{D}^-(R[G_1 \times G_2])$ fixed, the functor
    \[ A^\bullet \otimes^{\bL} - : \mathbf{D}^-(R[G_2 \times G_3]) \to \mathbf{D}^-(R[G_1 \times G_3]) \]
    may be identified with the left-derived functor of the functor 
    \[ \mathbf{K}^-(R[G_2 \times G_3]) \to \mathbf{K}^-(R[G_1 \times G_3]), \]
    \[ B^\bullet \mapsto \operatorname{Tot}(A^\bullet  \otimes_{R[G_2]} B^\bullet). \]
    \item Consider $C_c^\infty(G_1, R)$ as an object of $\Mod_{sm}(R[G_1 \times G_1])$ by the formula $((g, h) \cdot f)(x) = f(g^{-1} x h)$. Then the functor 
    \[ C_c^\infty(G_1, R) \otimes^{\bL}_{R[G_1]} - : \mathbf{D}^-(R[G_1 \times G_2]) \to \mathbf{D}^-(R[G_1 \times G_2]) \]
    is naturally isomorphic to the identity functor. In particular for any $M \in \Mod_{sm}(R[G_1 \times G_2])$, there are natural isomorphisms
    \[ C_c^\infty(G_1, R) \otimes^{\bL}_{R[G_1]} M \cong C_c^\infty(G_1, R) \otimes_{R[G_1]} M \cong M \]
    in $\mathbf{D}^-(R[G_1 \times G_2])$.
\end{enumerate}
\end{proposition}
\begin{proof}
    This is all standard when the category of smooth $R[G_1 \times G_2]$-modules is replaced by the category of all $R[G_1 \times G_2]$-modules. The same arguments apply equally in our case. We give the details, following \cite[\S 10.6]{Wei94}. First, let $A^\bullet \in \mathbf{D}^-(R[G_1 \times G_2])$. We consider the functor $\mathbf{L}^-(A^\bullet \otimes_{R[G_2]} -) : \mathbf{D}^-(R[G_2 \times G_3]) \to  \mathbf{D}^-(R[G_1 \times G_3])$ which is the left-derived functor of the functor $B^\bullet \mapsto \operatorname{Tot}(A^\bullet \otimes_{R[G_2]} B^\bullet)$. Thus by definition, $\mathbf{L}^-(A^\bullet \otimes_{R[G_2]} -)(B^\bullet) = \operatorname{Tot}(A^\bullet \otimes_{R[G_2]} Q^\bullet)$, where $Q^\bullet \to B^\bullet$ is a quasi-isomorphism from a bounded above complex of projective objects, and by the general machinery of derived functors, this is canonically independent of the choice of $Q^\bullet$.
    
    If $A^\bullet \to C^\bullet$ is a morphism in $\mathbf{K}^-(R[G_1 \times G_2])$, then there is an induced morphism $\operatorname{Tot}(A^\bullet \otimes_{R[G_2]} Q^\bullet) \to \operatorname{Tot}(C^\bullet \otimes_{R[G_2]} Q^\bullet)$. This is a quasi-isomorphism if $A^\bullet \to C^\bullet$ is a quasi-isomorphism, as follows from the existence of the convergent spectral sequence
    \[ E_2^{p, q} = H^q(A^\bullet) \otimes_{R[G_2]} Q^p \Rightarrow H^{p+q}(\operatorname{Tot}(A^\bullet \otimes_{R[G_2]} Q^\bullet)). \]
    The first two parts of the proposition now follow. 
    
    The final part says that $C_c^\infty(G_1, R)$ is acyclic for the functor $\mathbf{L}^-(- \otimes_{R[G_1]} M)$ and that $C_c^\infty(G_1, R) \otimes_{R[G_1]} M \cong M$. We take these in turn. The restriction functor $\Mod_{sm}(R[G_1 \times G_1]) \to \Mod_{sm}([G_1])$ preserves projectives (by Proposition \ref{prop_restriction_of_projective}), so for $A^\bullet \in \mathbf{D}^-(G_1 \times G_2)$, the image of $\mathbf{L}^-(A^\bullet \otimes_{R[G_1]} M)$ in $\mathbf{D}^-(R)$ may be computed using a resolution of $A^\bullet$ in $\Mod_{sm}([G_1])$. In particular, since $C_c^\infty(G_{1, r}, R)$ is projective, $C_c^\infty(G_1, R)$ is acyclic. We next write down an isomorphism $p: C_c^\infty(G_1, R) \otimes_{R[G_1]} M \to M$ in $\Mod_{sm}(R[G_1 \times G_2])$. It is the map 
    \[ p(f \otimes m) = \int_{g \in G_1} f(g) g m\, dg. \]
    It is easy to check that this map is well-defined and satisfies $p(L_h(f) \otimes m) = h p(f \otimes m)$. We need to check that $p$ is an isomorphism. 
    
    It is surjective since if $K$ is an open compact pro-$p$ subgroup of $G_1$ which fixes $m$, then $p(e_{K} \otimes m) = m$.     To see that it is injective, suppose that $p(\sum_i f_i \otimes v_i) = 0$. Choose a sufficiently small open compact pro-$p$ subgroup $K$ of $G_1$ so that each $f_i$ is left $K$-invariant; then we can rewrite $\sum_i f_i \otimes v_i = \sum_j \mathbf{1}_{K g_j} \otimes m_j$ for some elements $g_j \in G_1$ and $m_j \in M$. Choose an open normal subgroup $K_1 \leq K$ such that for each $j$, $g_j m_j$ is fixed by $K_1$, and choose coset representatives $k_l$ so that $K = \sqcup_k K_1 k_l$.  Then we can further rewrite $\sum_j \mathbf{1}_{K g_j} \otimes m_j = \sum_{j, l} \mathbf{1}_{K_1 k_l g_j} \otimes m_j$. Since $K_1$ is normal in $K$, $K_1$ fixes each vector $k_l g_j m_j$, and so we have $p( \mathbf{1}_{K_1 k_l g_j} \otimes m_j ) = \operatorname{vol}(K_1) k_l g_j m_j$, and therefore $\sum_{l, j} k_l g_j m_j = 0$, and
    \[ \sum_i f_i \otimes v_i = \sum_{l, j} (  \mathbf{1}_{K_1 k_l g_j} \otimes m_j - \mathbf{1}_{K_1} \otimes k_l g_j m_j ) = \sum_{k, j} ((k_l g_j)^{-1} - 1) (\mathbf{1}_{K_1} \otimes k_l g_j m_j). \]
    This vector is zero in $C_c^\infty(G_1, R) \otimes_{R[G_1]} M$, as required. 
\end{proof}
One important special case is when the groups $G_1, G_3$ are trivial, and $G_2 = G$ (say), in which case we have a functor $\mathbf{D}^-(R[G]) \times \mathbf{D}^-(R[G]) \to \mathbf{D}^-(R)$. We remark that if $X, Y \in \Mod_{sm}(R[G])$ then $H^i(X \otimes^{\bL}_{R[G]} Y)$ can be non-zero only if $i \leq 0$. It is natural to define $H_i(G, X) = H^{-i}(X \otimes^{\bL}_{R[G]} R)$. This gives the homology groups used in the introduction to this paper (for the remainder of the paper, we stick to cohomology groups only). 

The total tensor product is associative:
\begin{proposition}
If $G_1, G_2, G_3, G_4$ are groups satisfying our conditions then for any $A^\bullet \in \mathbf{D}^-(R[G_1 \times G_2]), B^\bullet \in \mathbf{D}^-(R[G_2 \times G_3])$, and $C^\bullet \in \mathbf{D}^-(R[G_3 \times G_4])$ there is a natural isomorphism
\[ (A^\bullet \otimes^{\bL}_{R[G_2]} B^\bullet) \otimes^{\bL}_{R[G_3]} C^\bullet \cong A^\bullet \otimes^{\bL}_{R[G_2]} (B^\bullet \otimes^{\bL}_{R[G_3]} C^\bullet) \]
in $\mathbf{D}^-(R[G_1 \times G_4])$.
\end{proposition}
\begin{proof}
After replacing $A^\bullet, B^\bullet$, and $C^\bullet$ by complexes of projective objects, we can apply \cite[\href{https://stacks.math.columbia.edu/tag/08BI}{Remark 08BI}]{stacks-project}.
\end{proof}
We also have a change of ring functor. If $G$ is a group satisfying our conditions, $S$ is a (commutative) $R$-algebra, and $M \in \Mod_{sm}(R[G])$, then $M \otimes_R S \in \Mod_{sm}(S[G])$. We write $- \otimes^{\bL}_R S$ for the left-derived functor of $- \otimes_R S$.
\begin{proposition}\label{prop_homology_change_of_coefficients}
If $G_1, G_2, G_3$ are groups satisfying our assumptions and $S$ is an $R$-algebra, then for any $A^\bullet \in \mathbf{D}^-(R[G_1 \times G_2]), B^\bullet \in \mathbf{D}^-(R[G_2 \times G_3])$, there is a natural isomorphism
\[ (A^\bullet \otimes^{\bL}_{R[G_2]} B^\bullet) \otimes^{\bL}_R S \cong (A^\bullet \otimes^{\bL}_R S) \otimes^{\bL}_{S[G_2]} (B^\bullet \otimes^{\bL}_R S) \]
in $\mathbf{D}^-(S[G_1 \times G_3])$.
\end{proposition}
\begin{proof}
After replacing $A^\bullet$, $B^\bullet$ by complexes of projective objects, it is enough to check that if $P \in \Mod_{sm}(R[G_1 \times G_2])$ and $Q \in \Mod_{sm}(R[G_2 \times G_3])$ are projective then $P \otimes_{R[G_2]} Q$ is a projective object of $\Mod_{sm}(R[G_1 \times G_3])$ and moreover that there is a natural isomorphism $(P \otimes_{R[G_2]} Q) \otimes_R S \cong (P \otimes_R S) \otimes_{S[G_2]} (Q \otimes_R S)$. The existence of the latter isomorphism follows from the right-exactness of $-\otimes_R S$; to show that $P \otimes_{R[G_2]} Q$ is projective, we can assume that $P = C_c^\infty(G_1 \times G_2, R)$ and $Q = C_c^\infty(G_2 \times G_3, R)$, in which case there is an isomorphism $P \otimes_{R[G_2]} Q \cong C_c^\infty(G_1 \times G_2 \times G_3, R)$. This is projective in $\Mod_{sm}(R[G_1 \times G_3])$.
\end{proof}
We can also use the derived tensor product to construct an internal tensor product, i.e.\ a functor $\mathbf{D}^-(R[G]) \times \mathbf{D}^-(R[G]) \to \mathbf{D}^-(R[G])$, which sends $(A^\bullet, B^\bullet)$ to the restriction of $A^\bullet \otimes^{\bL}_R B^\bullet$ to the diagonally embedded subgroup $G \subset G \times G$. The following lemma concerns a variant of this construction and will be used in the proof of Theorem \ref{thm_torsion_in_derived_tensor_product}.
\begin{lemma}\label{lem_two_actions_of_G}
Consider the two functors $F_1, F_2 : \mathbf{D}^-(R[G]) \to \mathbf{D}^-(R[G \times G])$, where $F_1(A^\bullet)$ is the restriction of $A^\bullet \otimes^{\bL}_R C_c^\infty(G, R) \in \mathbf{D}^-(R[G \times G \times G])$ to the subgroup $G \times G$, embedded as $(g_1, g_2) \mapsto (g_1, g_1, g_2)$, and $F_2(A^\bullet)$ is the restriction of $A^\bullet \otimes^{\bL}_R C_c^\infty(G, R) \in \mathbf{D}^-(R[G \times G \times G])$ to the subgroup $G \times G$, embedded as $(g_1, g_2) \mapsto (g_2, g_1, g_2)$. Then $F_1$ and $F_2$ are naturally isomorphic. 
\end{lemma}
\begin{proof}
Both $F_1, F_2$ are induced by functors $\Mod_{sm}(R[G]) \to \Mod_{sm}(R[G \times G])$ and it suffices to show that these functors are naturally isomorphic, i.e.\ that if $M \in \Mod_{sm}(R[G])$ then there is a natural isomorphism $M \otimes_R C_c^\infty(G, R) \cong M \otimes_R C_c^\infty(G, R)$, where $(g_1, g_2)$ acts on the source by $(g_1, g_2) \cdot (m \otimes f) = g_1 m \otimes L_{g_1} R_{g_2} f$ and on the target by $(g_1, g_2) \cdot (m \otimes f) = g_2 m \otimes L_{g_1} R_{g_2} f$. We may identify $M \otimes_R C_c^\infty(G, R) = C_c^\infty(G, M)$, where the two actions are now given respectively by $((g_1, g_2) \cdot f)(g) = g_1 f(g_1^{-1} g g_2)$ and $((g_1, g_2) \cdot f)(g) = g_2 f (g_1^{-1} g g_2)$. These two actions are intertwined by the map $f \mapsto (F_f(g) = g^{-1} f(g))$. This is the desired natural isomorphism. 
\end{proof}
Finally, we establish some useful finiteness properties.
\begin{lemma}
Suppose that $R$ is a local ring of residue field $k = R / \ffrm_R$, and let $V$ be an admissible $R[G]$-module.
\begin{enumerate}
    \item If $W \subset V$ is an $R[G]$-submodule such that $V = W + \ffrm_R V$, then $V = W$. 
    \item If $V / \ffrm_R V$ is a finitely generated $k[G]$-module, then $V$ is a finitely generated $R[G]$-module.
\end{enumerate}
\end{lemma}
Recall that $V \in \Mod_{sm}(R[G])$ is admissible if for each open compact subgroup $K \subset G$, $V^K$ is a finitely generated $R$-module. 
\begin{proof}
The second part follows from the first (take $W$ to be the submodule generated by the lifts of a generating set for $V / \ffrm_R V$). We prove the first. If $K \subset G$ is an open compact subgroup of pro-$p$ order, then $V^K = W^K + \ffrm_R V^K$. Since $V$ is admissible, $V^K$ is a finite $R$-module, so Nakayama's lemma implies that $V^K = W^K$. Since $V$ is the union of the submodules $V^K$ (as $K$ ranges over all open compact pro-$p$-subgroups of $G$), we find $V = W$.
\end{proof}
\begin{lemma}\label{lem_coinvariants_finitely_generated}
Suppose that $R$ is Noetherian, and let $V, W \in \Mod_{sm}(R[G])$ be such that $V$ is finitely generated as an $R[G]$-module and $W$ is admissible. Then $V \otimes_{R[G]} W$ is a finitely generated $R$-module.
\end{lemma}
\begin{proof}
Let $v_1, \dots, v_n \in V$ be generators for $V$ as $R[G]$-module and let $K \subset G$ be an open compact pro-$p$ subgroup such that $v_1, \dots, v_n \in V^K$. We claim that the map $(W^K)^n \to V \otimes_{R[G]} W$, $(w_i) _i \mapsto \sum_i v_i \otimes w_i$, is surjective. Since $W$ is admissible and $R$ is Noetherian, this will imply the truth of the lemma. 

It is easy to see that $V \otimes_{R[G]} W$ is generated as an $R$-module by elements of the form $v_i \otimes w$, with $w \in W$. Take such an element, and choose an open subgroup $K' \subset K$ such that $w \in W^{K'}$. Then in $V \otimes_{R[G]} W$ we have
\[ v_i \otimes w \equiv [K : K']^{-1} \sum_{k \in K / K'} v_i \otimes k w \equiv [K : K']^{-1}( v_i \otimes \tr_{K / K'}(w) ). \]
Since $\tr_{K / K'}(w) \in W^{K}$, we're done. 
\end{proof}
\begin{proposition}\label{prop_tor_is_fg}
Suppose that $R$ is Noetherian, let $K \subset G$ be an open compact subgroup, and let $W \in \Mod_{sm}(R[K])$ be finitely generated as $R$-module. If $V$ is an admissible $R[G]$-module, then for each $i \in \bZ$ $H^i(\cInd_K^G W \otimes^{\bL}_{R[G]} V)$ is a finitely generated $R$-module. The same holds more generally if $V$ is replaced by a bounded above complex $A^\bullet \in \mathbf{D}^-(R[G])$ whose cohomology groups are admissible $R[G]$-modules.
\end{proposition}
\begin{proof}
If $K$ is pro-$p$, then $\cInd_K^G W \otimes^{\bL}_{R[G]} V = \cInd_K^G W \otimes_{R[G]} V$ (as this compact induction is projective, cf. the proof of Proposition \ref{prop_mod_G_enough_projectives}). Since $\cInd_K^G W$ is finitely generated, Lemma \ref{lem_coinvariants_finitely_generated} implies that $\cInd_K^G W \otimes_{R[G]} V$ is a finitely generated $R$-module. 

If $K$ is not pro-$p$, then we can find a resolution $P^\bullet \to W$ by projective $R[K]$-modules which are finitely generated as $R$-modules. Then 
\[ \cInd_K^G W \otimes^{\bL}_{R[G]} V = \cInd_K^G P^\bullet \otimes_{R[G]} V \]
and the result follows from the fact that each $\cInd_K^G P^i \otimes_{R[G]} V$ is a finitely generated $R$-module and our assumption that $R$ is Noetherian.
\end{proof}

\begin{theorem}\label{thm_tor_is_fg}
Suppose that $R$ is Noetherian and that $G = \GL_n(F_v)$ for number field $F$ and $p$-adic place $v$ of $F$. Let $G^0 = \{ g \in G \mid \det(g) \in \cO_{F_v}^\times \}$.
\begin{enumerate}
    \item Suppose that $V, W \in \Mod_{sm}(R[G])$ are admissible and that $V$ is finitely generated as an $R[G^0]$-module. Then for each $i \in \bZ$, $H^i(V \otimes^{\bL}_{R[G^0]} W)$ is finitely generated as an $R$-module.
    \item More generally, suppose that $A^\bullet, B^\bullet \in \mathbf{D}^-(R[G])$ are such that for each $i \in \bZ$, $H^i(A^\bullet)$ is admissible and finitely generated as an $R[G^0]$-module, and $H^i(B^\bullet)$ is admissible. Then for each $i \in \bZ$, $H^i(A^\bullet \otimes^{\bL}_{R[G^0]} W)$ is finitely generated as an $R$-module.
\end{enumerate}
\end{theorem}
A similar statement in the case that $R$ is a field appears in \cite{Vig97}.
\begin{proof}
The second part follows from the first by a spectral sequence argument, so we prove just the first. Since $V$ is finitely generated, we can find $m \geq 2$ such that, taking $K$ to be the principal congruence subgroup of level $m$, $V$ is generated by $V^K$. Since $V$ is assumed admissible, $V^K$ is a finitely generated $R$-module. We now appeal to \cite[Theorem 3.2]{Gro05}; this states that there is an isomorphism $C^\bullet(X, \underline{V}) \to V$ in $\mathbf{D}^-(R[G])$, where:
\begin{itemize}
    \item $X$ is the Bruhat--Tits building of $\PGL_n(F_v)$, considered as  simplicial complex,  each simplex $\sigma$ corresponding to a homothety class of lattice chains $L_0 \supset L_1 \supset \dots \supset L_k \supset \varpi_v L_0$ being endowed with the cyclic orientation $(L_0, \dots, L_k)$. 
    \item $\underline{V}$ is the coefficient system on $X$ which assigns to each simplex $\sigma$ as above the finitely generated $R$-module $V^{K_\sigma}$, where $K_\sigma$ is the subgroup of $\GL_n(F_v)$ defined in \cite[\S 3]{Gro05}.
    \item $C^q(X, \underline{V}) = \oplus_{\sigma : \dim \sigma = -q} V^{K_\sigma}$, endowed with the action of $G = \GL_n(F_v)$ given by the formula $(g \cdot (v_\sigma)_\sigma)_\tau = g v_{g^{-1} \tau}$.
\end{itemize}
Let $\sigma_0, \dots, \sigma_r$ be representatives of the finitely many $G$-orbits of simplices of dimension $-q$. Let $G_j = \operatorname{Stab}_G(\sigma_j)$. Then $G_j$ contains $K_{\sigma_j}$ as a normal subgroup and is the pre-image in $G$ of an open compact subgroup of $\PGL_n(F_v)$, and there is an isomorphism
\[ C^q(X, \underline{V}) = \oplus_{j=1}^r \cInd_{G_j}^G V^{K_{\sigma_j}} \]
in $\Mod_{sm}(R[G])$. We claim that for each $i \in \bZ$, $H^i(C^q(X, \underline{V}) \otimes^{\bL}_{R[G^0]} W)$ is a finitely generated $R$-module. This implies the statement of the theorem, by a spectral sequence argument. It suffices to show that for each $i \in \bZ$, $H^i(\cInd_{G_j}^G V^{K_{\sigma_j}} \otimes^{\bL}_{R[G^0]} W)$ is a finitely generated $R$-module. This is the content of Proposition \ref{prop_tor_is_fg}. 
\end{proof}
We will apply this theorem in conjunction with the following proposition.
\begin{proposition}\label{prop_GGT}
Suppose that $R$ is Noetherian and that $G = \GL_n(F_v)$ for number field $F$ and $p$-adic place $v$ of $F$. Let $G^0 = \{ g \in G \mid \det(g) \in \cO_{F_v}^\times \}$. Let $P$ be a standard parabolic subgroup of $G$ and let $M = C_c^\infty(P / G, R)$. Then $M$ is a finitely generated $R[G^0]$-module.
\end{proposition}
\begin{proof}
Since the centre of $G$ acts trivially on $M$, and $G^0$ has finite index in $G$ mod centre, it suffices to show that $M$ is finitely generated as an $R[G]$-module. Guillem Garcia Tarrach has shown us an elementary proof of this statement. It also follows from the deeper statement that parabolic induction from a Levi subgroup of $\GL_n(F_v)$ preserves the property of being finitely generated (see \cite[Corollaire 3.9, Proposition 7.5]{Dat09}).
\end{proof}

\section{The elliptic representations}\label{sec_elliptic_representations}

Let $F$ be a number field and let $v$ be a finite place of $F$. If 
\[ \lambda : n = \lambda_1 + \dots + \lambda_k \]
is an ordered partition of $n$, then we write $\pi_\lambda$ for the irreducible smooth $\bC[\GL_n(F_v)]$-module such that 
\[ \rec_{F_v}(\pi_\lambda) = \oplus_{i=1}^k \Sp_{\lambda_i}(| \cdot |^{(n+1)/2 - (\lambda_1 + \dots + \lambda_{i-1} + (\lambda_i + 1 )/2)}). \]
Pictorially, we can represent $ \rec_{F_v}(\pi_\lambda) $ by putting the characters $ | \cdot |^{(n-1)/2}, \dots,  | \cdot |^{(1-n)/2}$ on the diagonal and making the monodromy operator $N$ block diagonal with nilpotent Jordan blocks of sizes $\lambda_1, \dots, \lambda_k$. It is convenient to associate to $\lambda$ the partition (as sets) $\{ 1, \dots, n \} = \Lambda_1 \sqcup \dots \sqcup \Lambda_k$, where $\Lambda_i = \{ \lambda_1 + \dots + \lambda_{i-1} + 1, \dots, \lambda_1 + \dots + \lambda_i \}$.

We have the following observations.
\begin{proposition}\label{prop_Jacquet_module_of_elliptic_representation}
\begin{enumerate}
    \item $\pi_n$ is the Steinberg representation, and $\pi_{1, 1, \dots, 1}$ is the trivial representation.
    \item Each representation $\pi_\lambda$ has the same cuspidal support as $\pi_n$. The Jordan--H\"older factors of the representation $i_{B_n(F_v)}^{\GL_n(F_v)} \otimes_{i=1}^n | \cdot |^{(n+1-2i)/2}$ are precisely the representations $\pi_\lambda$, each occurring with multiplicity one. 
    \item The normalised Jacquet module $r_{B_n}(\pi_\lambda)$ may be described as follows: it is a direct sum of the characters $\sigma(1) \otimes \dots \otimes \sigma(n)$ of $T_n(F_v) = (F_v^\times)^n$, where $\sigma$ ranges over the set of bijective functions 
    \[ \sigma : \{ 1, \dots, n \} \to \{ | \cdot|^{(n-1)/2}, \dots, | \cdot |^{(1-n)/2} \} \]
    such that for each $1 \leq i \leq n-1$, we have 
    \[ \sigma^{-1}( | \cdot |^{(n+1-2i)/2} ) < \sigma^{-1}( | \cdot |^{(n-1-2i)/2}) \]
    if and only if there is $1 \leq j \leq k$ such that $\{ i, i+1 \} \subset \Lambda_j$.
\end{enumerate}
\end{proposition}
\begin{proof}
According to the construction of the local Langlands correspondence $\rec_{F_v}$ using the Bernstein--Zelevinsky classification (as explained e.g.\ in \cite{Wed08}), $\pi_\lambda$ is the unique irreducible quotient of the induced representation 
\[ i_{P_\lambda(F_v)}^{\GL_n(F_v)} \St_{\lambda_1}( | \cdot |^{(n-\lambda_1)/2}) \otimes \dots \otimes \St_{\lambda_i}( | \cdot |^{(n - 2 \lambda_1 - \dots - 2 \lambda_{i-1} - \lambda_i)/2}) \otimes \dots \otimes \St_{\lambda_k}( | \cdot |^{(\lambda_k - n)/2}). \]
This representation is in turn a quotient of the induced representation
\begin{multline*} i_{B_n(F_v)}^{\GL_n(F_v)} | \cdot |^{(n+1 - 2 \lambda_1)/2} \otimes \dots \otimes | \cdot |^{(n-1)/2} \\ \otimes | \cdot |^{(n+1 - 2(\lambda_1 + \lambda_2))/2} \otimes \dots \otimes | \cdot |^{(n+1 - 2 (\lambda_1 + 1))/2}\\  \otimes \dots \otimes | \cdot |^{(2-(n+1))/2} \otimes \dots \otimes | \cdot |^{(2 \lambda_k - (n+1))/2}.
\end{multline*}
By \cite[Proposition 2.10]{Zel80}, this induction from $B_n(F_v)$ in fact has $\pi_\lambda$ as its unique irreducible quotient. The same result contains a description of $r_{B_n}(\pi_\lambda)$. The identification of $\pi_n$ and $\pi_{1, \dots, 1}$ in the first part of the proposition is a very special case of this. Finally, \cite[Proposition 2.1]{Zel80} shows that the induced representation decomposes with multiplicity one.
\end{proof}
\begin{corollary}\label{cor_invariant_WD_filtrations}
Let $n = n_1 + n_2$ be a partition with $n_1, n_2 \geq 1$, and let $\lambda : n = \lambda_1 + \dots + \lambda_k$ be any partition such that $\lambda \neq n, (n_1, n_2)$. Then there exists a character $\chi = \chi_1 \otimes \dots \otimes \chi_n$ of $T_n(F_v) = (F_v^\times)^n$ which is a subquotient of the normalised Jacquet module $r_{B_n}(\pi_\lambda)$ and such that there is no increasing filtration $0 \subset \Fil_1 \subset \Fil_2 \subset \dots \subset \Fil_n = \bC^n$ by sub-Weil--Deligne representations of $\rec_{F_v}(\pi_{n_1, n_2})$ such that for each $i = 1, \dots, n$, $\Fil_i / \Fil_{i-1} \cong \bC(\chi_i)$. 
\end{corollary}
\begin{proof}
Using the third part of Proposition \ref{prop_Jacquet_module_of_elliptic_representation}, we see that $r_{B_n}(\pi_\lambda)$ contains the character 
\begin{multline*} \chi = | \cdot |^{(n + 1 - 2 (\lambda_1 + \dots + \lambda_{k-1} + 1))/2} \otimes \dots \otimes  | \cdot |^{(1-n)/2} \\ \otimes | \cdot |^{(n + 1 - 2 (\lambda_1 + \dots + \lambda_{k-2} + 1))/2} \otimes \dots \otimes | \cdot |^{(n + 1 - 2 (\lambda_1 + \dots + \lambda_{k-1}))/2} \\ \otimes \dots \otimes | \cdot |^{(n-1)/2} \otimes \dots \otimes | \cdot |^{(n + 1 - 2 \lambda_1)/2}. 
\end{multline*}
We will show that the corollary holds with this choice of $\chi$. In order to see this, we split into cases. The existence of a filtration $\Fil_\bullet$ as in the statement of the corollary would imply that $\rec_{F_v}(\pi_{n_1, n_2})$ contains an invariant line on which $W_{F_v}$ acts by the character $| \cdot |^{(2 \lambda_k - (n+1))/2}$. However, there are exactly two invariant lines in $\rec_{F_v}(\pi_{n_1, n_2})$, on which $W_{F_v}$ acts by the characters $| \cdot |^{(n-1)/2}$ and $| \cdot |^{(n - 1 - 2n_1)/2}$. If $\lambda_k \neq n_2$, then neither of these equal $| \cdot |^{(2 \lambda_k - (n+1))/2}$, leading to the desired contradiction. Suppose then that $\lambda_k = n_2$. Then $\lambda_{k-1} < n_1$ and we look at $\Fil_{\lambda_k + 1} = \Fil_{n_2 + 1}$, which would be the sum of those 1-dimensional subspaces of $\rec_{F_v}(\pi_{n_1, n_2})$ where $W_{F_v}$ acts by the characters 
\[ | \cdot |^{(n + 1 - 2 (\lambda_1 + \dots + \lambda_{k-2} + 1))/2}, | \cdot |^{(n + 1 - 2 (\lambda_1 + \dots + \lambda_{k-1} + 1))/2}, \dots, | \cdot |^{(1-n)/2}. \]
 This subspace of $\rec_{F_v}(\pi_{n_1, n_2})$ is not in fact a sub-Weil--Deligne representation (as it is not invariant by $N$), so we obtain the desired contradiction in this case also. 
\end{proof}
We now compute some derived tensor products of these representations. We need some notation. Let $G = \GL_n(F_v)$, $G^0 = \{ g \in G \mid \det(g) \in \cO_{F_v}^\times \}$. We identify $\{ 1, \dots, n-1 \}$ with the set of simple roots of the maximal torus $T_n \subset \GL_n$, sending $i$ to $(t_1, \dots, t_n) \in T_n \mapsto t_i / t_{i+1}$. If $\lambda$ is a partition of $n$, let $P_\lambda$ denote the associated standard parabolic subgroup of $\GL_n$, $M_\lambda$ its standard Levi subgroup, and let $I_\lambda \subset \{ 1, \dots, n-1 \}$ denote the set of simple roots which are roots of $M_\lambda$. Identifying $S_n$ with the Weyl group of $T_n$ in $\GL_n$, we write $w_0 \in S_n$ for the longest element of $S_n$ with respect to the Borel subgroup $B_n$ and, for subsets $I, J \subset \{ 1, \dots, n-1 \}$, $\delta(I, J)$ for the cardinality of the symmetric difference $(I \cup J) - (I \cap J)$. We write $I^c = \{1, \dots, n-1 \} - I$.
\begin{theorem}\label{thm_homology_of_PGL_n_elliptic_representations}
Let $\lambda$ be a partition of $n$ and let $\pi$ be an irreducible admissible $\bC[G]$-module of trivial central character. If $\pi_\lambda \otimes^{\bL}_{\bC[\PGL_n(F_v)]} \pi \neq 0$, then there is a partition $\mu$ such that $\pi \cong \pi_\mu$. In this case, $H^{-i}(\pi_\lambda \otimes^{\bL}_{\bC[\PGL_n(F_v)]} \pi)$ is non-zero if and only if $i = \delta(I_\lambda, -w_0(I_\mu))$, in which case it is a 1-dimensional $\bC$-vector space.
\end{theorem}
\begin{proof}
Tensor-hom adjunction gives for any $i \in \bZ$ an isomorphism 
\[ H^{-i}(\pi_\lambda \otimes^{\bL}_{\bC[\PGL_n(F_v)]} \pi)^\vee \cong \Ext^i_{\Mod_{sm}(\bC[\PGL_n(F_v)])}(\pi_\lambda, \pi^\vee) \]
(where the two $\vee$'s denote vector space dual and smooth dual, respectively). For a subset $J \subset \{ 1, \dots, n-1 \}$, let $\pi_J$ denote the representation defined in \cite{Dat06}; it is the unique irreducible quotient of the unnormalised induction $\Ind_{P_J(F_v)}^{\GL_n(F_v)} \bC$, where $P_J$ is the standard parabolic subgroup generated by $B_n$ and the negative roots corresponding to elements of $J$. Then \cite[Remarque 2.3.4]{Dat06} shows that there is an isomorphism $\pi_\lambda \cong \pi_{-w_0(I_\lambda^c)}$, while \cite[Th\'eorème 2.1.4]{Dat06} shows that $\Ext^i_{\Mod_{sm}(\bC[\PGL_n(F_v)])}(\pi_I, \pi_{J})$ is non-zero only when $i = \delta(I, J)$, in which case it is a 1-dimensional $\bC$-vector space. Finally, \cite[Lemme 2.3.3]{Dat06} shows that the smooth dual of $\pi_J$ is $\pi_{-w_0(J)}$. We conclude that $H^{-i}(\pi_\lambda \otimes^{\bL}_{\bC[\PGL_n(F_v)]} \pi_\mu)$ is non-zero only if $i = \delta(-w_0(I_\lambda^c), I_\mu^c) = \delta(I_\lambda, -w_0(I_\mu))$, in which case it is a 1-dimensional $\bC$-vector space.

To complete the proof of the theorem, we need to show that if $\pi$ is an irreducible admissible representation of trivial central character that is not of the form $\pi_\mu$ for any partition $\mu$, then $H^\ast(\pi_\lambda \otimes^{\bL}_{\bC[\PGL_n(F_v)]} \pi) = 0$. It's enough to show that $\Ext^\ast_{\Mod_{sm}(\bC[\PGL_n(F_v)])}(\pi_\lambda, \pi^\vee) = 0$. If $\alpha$ is an element of the centre of $\Mod_{sm}(\bC[\PGL_n(F_v)])$, then the two actions of $\alpha$ on $\Ext^\ast_{\Mod_{sm}(\bC[\PGL_n(F_v)])}(\pi_\lambda, \pi^\vee)$, one induced by the endomorphism $\alpha_{\pi_\lambda} : \pi_\lambda \to \pi_\lambda$, the other induced by the endomorphism $\alpha_{\pi^\vee} : \pi^\vee \to \pi^\vee$, are the same (as follows e.g. from the fact that 
\[ \Ext^\bullet_{\Mod_{sm}(\bC[\PGL_n(F_v)])}(-, \pi^\vee) \]
 is a universal $\delta$-functor). 

 Since $\pi$ is not of the form $\pi_\mu$, the representations $\pi$ and $\pi^\vee$ have distinct cuspidal supports. Bernstein's description of the centre of $\Mod_{sm}(\bC[\PGL_n(F_v)])$ implies that we can choose $\alpha$ so that $\alpha_{\pi_\lambda} = 0$ but $\alpha_{\pi^\vee}$ is the identity. This implies that $\Ext^\bullet_{\Mod_{sm}(\bC[\PGL_n(F_v)])}(\pi_\lambda, \pi^\vee)$ = 0, as required.
 \end{proof}
The definition of the representations $\pi_J$ appearing in the proof makes sense over any field $K$ of characteristic 0. We extend the definition of  $\pi_\lambda$ to any such field by setting $\pi_\lambda = \Ind_{P_{-w_0(I_\lambda^c)}(F_v)}^{\GL_n(F_v)} K$. If $K$ is abstractly isomorphic to $\bC$, then there is an isomorphism
\[ \rec_{F_v}^T(\pi_\lambda) \cong \oplus_{i=1}^k \Sp_{\lambda_i}(| \cdot |^{n - (\lambda_1 + \dots + \lambda_{i-1} + (\lambda_i + 1 )/2)}), \]
where we note that the definition of the representation on the right-hand side makes sense over any such field because only integer powers of the norm character $| \cdot |$ appear in the definition.

Here is a variant of Theorem \ref{thm_homology_of_PGL_n_elliptic_representations}.
\begin{theorem}\label{thm_homology_of_G_0_elliptic_representations}
Let $\lambda$ be a partition of $n$ and let $\pi$ be an irreducible admissible $\bC[G]$-module. If $\pi_\lambda \otimes^{\bL}_{\bC[G^0]} \pi \neq 0$, then there is a partition $\mu$ and an unramified character $\psi : F_v^\times \to \bC^\times$ such that $\pi \cong \pi_\mu \otimes \psi$. In this case, $H^{-i}(\pi_\lambda \otimes^{\bL}_{\bC[G^0]} \pi)$ is non-zero if and only if $i = \delta(I_\lambda, -w_0(I_\mu))$, in which case it is a 1-dimensional $\bC$-vector space.
\end{theorem}
\begin{proof}
Let $Z$ denote the centre of $G$ and let $Z^0 = Z \cap G^0$. Let $\omega_\pi : Z \to \bC^\times$ denote the central character of $\pi$. If $\omega_\pi|_{Z^0} \neq 1$, then $\pi_\lambda \otimes^{\bL}_{\bC[G^0]} \pi = 0$ and $\pi$ is not an unramified twist of any $\pi_\mu$. If $\omega_\pi|_{Z^0} = 1$, then we can assume without loss of generality (after replacing $\pi$ by an unramified twist) that in fact $\omega_\pi = 1$ and $\pi$ is a smooth $\bC[\PGL_n(F_v)]$-module.

Let $\overline{G}^0 = G^0 / Z^0$ denote the image of $G^0$ in $\PGL_n(F_v)$. Then we have
\[ \Ext^i_{\Mod_{sm}(\bC[G^0])}(\pi_\lambda, \pi) = \Ext^i_{\Mod_{sm}(\bC[\overline{G}^0])}(\pi_\lambda, \pi) \]
because $C_c^\infty(\overline{G}^0, \bC)$ is projective in $\Mod_{sm}(\bC[G^0])$, and
\[ \Ext^i_{\Mod_{sm}(\bC[\PGL_n(F_v)])}(\pi_\lambda, \pi) = \Ext^i_{\Mod_{sm}(\bC[\overline{G}^0])}(\pi_\lambda, \pi)^{\PGL_n(F_v) / \overline{G}^0}, \]
noting that $\PGL_n(F_v) / \overline{G}^0$ is a cyclic group of order $n$. If $\Ext^i_{\Mod_{sm}(\bC[G^0])}(\pi_\lambda, \pi) \neq 0$, then we can find a character $\psi : \PGL_n(F_v) / \overline{G}^0 \to \bC^\times$ (otherwise said, an unramified character of $\GL_n(F_v)$ of order dividing $n$) such that $\psi$ appears in $\Ext^i_{\Mod_{sm}(\bC[\overline{G}^0])}(\pi_\lambda, \pi)$, and hence such that $\Ext^i_{\bC[\PGL_n(F_v)]}(\pi_\lambda, \pi \otimes \psi) \neq 0$. Theorem \ref{thm_homology_of_PGL_n_elliptic_representations} implies that $\pi \otimes \psi \cong \pi_\mu$ for some $\mu$. If $\psi'$ is another character which appears in this way then we similarly obtain $\pi \otimes \psi' \cong \pi_{\mu'}$ for some $\mu'$. Looking at cuspidal supports, we see that this is possible only if $\psi = \psi'$, and moreover that there is an isomorphism
\[ \Ext^i_{\Mod_{sm}(\bC[G^0])}(\pi_\lambda, \pi) \cong \Ext^i_{\Mod_{sm}(\bC[\PGL_n(F_v)])}(\pi_\lambda, \pi \otimes \psi). \]
The desired statement now follows from Theorem \ref{thm_homology_of_PGL_n_elliptic_representations}.
\end{proof}
The representations $\pi_\lambda$ are of interest to us for two reasons. The first is that they show up in spaces of automorphic representations when the level-raising congruence is satisfied (because of the second part of Proposition \ref{prop_Jacquet_module_of_elliptic_representation}). The second is that they appear in the cohomology of the Drinfeld upper half-space. Indeed, let $p$ be a prime not dividing $q_v$ and let $E / \bQ_p$ be a finite extension. Let $R$ be one of $E$, $\cO_E$, or $\cO_E / (\varpi^c)$ for some $c \geq 1$. The Drinfeld upper half-space is the rigid analytic space $\Omega_{F_v}^n$ which is the admissible open subspace of $\bP^{n-1}_{F_v}$ obtained by deleting all $F_v$-rational hyperplanes. The group $\PGL_n(F_v)$ acts on it and, according to \cite[\S 4.1.3]{Dat06}, its compactly supported cohomology is computed by a complex
\[ R \Gamma_c(\Omega_{F_v}^{n-1, ca}, R) \in \mathbf{D}^-(R[\PGL_n(F_v)]). \]
(Here the superscript $ca$, following the notation of \cite{Dat06}, denotes base extension to the completion $\widehat{\overline{F}}_v$ of the fixed algebraic closure of $F_v$.) Formation of this complex is compatible with change of coefficients $R \to R'$ (where $R'$ is also of our allowed form). When $R = E$, we have the following explicit description of this complex, which is part of \cite[Proposition 4.2.2]{Dat06}:
\begin{theorem}\label{thm_decomposition_of_R_Gamma_Omega}
There is an isomorphism in $\mathbf{D}^-(E[\PGL_n(F_v)])$:
\[ R \Gamma_c(\Omega_{F_v}^{n-1, ca}, E) \cong \bigoplus_{i=0}^{n-1} \pi_{i+1, 1, 1, \dots, 1}[i-2(n-1)]. \]
\end{theorem}
For general $R$, we can at least prove:
\begin{theorem}
Let $R$ be one of $E$, $\cO_E$, or $\cO_E / (\varpi^c)$. Let $G = \GL_n(F_v)$, $G^0 = \{ g \in G \mid \det(g) \in \cO_{F_v}^\times \}$.  Let $A^\bullet \in \mathbf{D}^-(R[G])$ be a complex such that for each $i \in \bZ$, $H^i(A^\bullet)$ is an admissible $R[G]$-module. Then for each $i \in \bZ$, 
\[ H^i(R \Gamma_c(\Omega_{F_v}^{n-1, ca}, R) \otimes^{\bL}_{R[G^0]} A^\bullet) \]
is a finitely generated $R$-module.
\end{theorem}
\begin{proof}
This follows from Theorem \ref{thm_tor_is_fg}, provided we can show that each group $H^i_c(\Omega_{F_v}^{n-1, ca}, R)$ is a finitely generated admissible $R[G^0]$-module. This is a consequence of Proposition \ref{prop_GGT} (finite generation of parabolic induction) and \cite[Th\'eor\`eme 3.1.1]{Dat06} (which shows that each group $H^i_c(\Omega_{F_v}^{n-1, ca}, R)$ is a quotient of a parabolic induction). 
\end{proof}
	
\section{The definite unitary group}\label{sec_definite_unitary_group}

The proof of our level-raising theorem will take place on a definite unitary group. We introduce the notation and assumptions we need here. Let $F_0$ be an imaginary quadratic field and let $F$ be a CM number field of the form $F = F^+ F_0$, such that $F / F^+$ is everywhere unramified. Let $n \geq 1$ is an integer, and let $G$ be the reductive group over $\cO_{F^+}$ whose functor of points is
\begin{equation}\label{eqn_defn_of_G} G(R) = \{ g \in \GL_n(R \otimes_{\cO_{F^+}} \cO_F) \mid g = (1 \otimes c)(g)^{-t} \}. 
\end{equation}
Note that for each finite place $v$ of $F^+$, $G_{F^+_v}$ is unramified (hence quasi-split), while for each infinite place $v$, $G(F^+_v)$ is compact. We write $U_n$ for the reductive group over $\cO_{F^+}$ whose functor of points is 
\begin{equation} U_n(R) = \{ g \in \GL_n(R \otimes_{\cO_{F^+}} \cO_F) \mid g = \Phi_n (1 \otimes c)(g)^{-t} \Phi_n^{-1} \}, 
\end{equation}
where $\Phi_n = ((-1)^i \delta_{i, n+1-j}) \in \GL_n(\bZ)$. The group $U_n$ is quasi-split (with Borel subgroup $B$ and maximal torus $T$ defined over $\cO_{F^+}$ respectively by the upper-triangular matrices and diagonal matrices) and admits an $F$-splitting (given by the standard matrices $X_i = E_{i, i+1}$ for $i = 1, \dots, n-1$). For any finite place $v$ of $F^+$, the tuple $(B, T, \{ X_i \})$ is admissible for $G_{\cO_{F^+_v}}$, in the sense of \cite[Definition 7.1]{Hal93}. 

We can identify
\[ U_{n, F} = \{ (g_1, g_2) \in \GL_n \times \GL_n \mid g_2 = \Phi_n g_1^{-t} \Phi_n^{-1} \} \]
and
\[ G_F = \{ (g_1, g_2) \in \GL_n \times \GL_n \mid g_2 = g_1^{-t} \}. \]
We define an inner twist $\xi : U_{n, F} \to G_F$ by the formula $\xi(g_1, g_2) = (g_1, \Phi_n g_2 \Phi_n^{-1})$. By definition, this means that $\xi^{-1} {}^c \xi$ is an inner automorphism of $U_{n, F}$. We can lift this to a pure inner twist. We recall that a pure inner twist (see \cite[\S 2]{Kal11}) of $U_n$ is a triple $(H, f, x)$ where $H$ is a reductive group over $F^+$, $f : U_{n, \overline{F}} \to H_{\overline{F}}$ is an isomorphism, and $x \in Z^1(F^+, U_n)$ is a cocycle such that for all $\sigma \in \Gal(\overline{F} / F^+)$, $f^{-1} {}^\sigma f = \Ad(x_\sigma)$. An isomorphism $(H, f, x) \to (H', f', x')$ of pure inner twists is a pair $(g, h)$ consisting of an isomorphism $g : H \to H'$ and an element $h \in U_n(\overline{F})$ such that $(f')^{-1} \circ g \circ f = \Ad(h)$ and for all $\sigma \in \Gal(\overline{F} / F^+)$, we have $x'_\sigma = h x_\sigma {}^\sigma h^{-1}$. Two pure inner twists are isomorphic if and only if the cocycles $x, x'$ have the same image in $H^1(F^+, U_n)$. Similar definitions and remarks apply with $F^+$ replaced by any completion $F^+_v$. 

To lift $\xi$ to the data of pure inner twist, it suffices to give $z \in U_n(F)$ such that $z {}^c z = 1$ and $\xi^{-1} {}^c \xi = \Ad(z)$. If $n$ is odd, then $(\Phi_n, \Phi_n) \in U_n(F^+)$ and we take $z = (\Phi_n, \Phi_n)$. If $n$ is even, then we choose $\zeta \in F_0^\times$ such that ${}^c \zeta = - \zeta$ and take $z = (\zeta \Phi_n, - \zeta \Phi_n)$. Our constructions will depend on this choice but the precise choice is unimportant. We do make use of the observation that since $F_0 \subset F$, the group $G$ can be naturally defined over $\bQ$, and moreover that the pure inner twist is also defined over $\bQ$. 

If $v$ is a finite place of $F^+$, then the cocycle in $Z^1(F / F^+, U_n)$ sending $c$ to $z$ has trivial image in $H^1(F^+_v, U_n)$ (cf. \cite[\S 1.9]{Rog90}). It follows that $(\xi, z)$ becomes isomorphic to the trivial pure inner twist $(U_n, \operatorname{1}, 1)$ over $F^+_v$; such an isomorphism is determined uniquely up to automorphisms of the trivial pure inner twist, or in other words up to conjugation by $U_n(F^+_v)$. We thus have for each finite place $v$ of $F^+_v$ an isomorphism $\iota_v :G(F^+_v) \to U_n(F^+_v)$, determined up to $U_n(F^+_v)$-conjugacy. 
\begin{lemma}
We can choose the isomorphisms $\iota_v$ so that for all but finitely many places $v$ of $F^+$, and for every $v$ which splits in $F$, $\iota_v^{-1}(U_n(\cO_{F^+_v})) = G(\cO_{F^+_v})$.
\end{lemma}
\begin{proof}
The claim is vacuous at split places $v$ (as then there is a unique conjugacy class of hyperspecial maximal compact subgroups). It is also vacuous at inert places when $n$ is odd, for the same reason, although we don't need this. In general, let $v$ be any finite place of $F^+$ such that, if $w$ denotes a place of $F$ lying above $v$, then $z$ lies in $U_n(\cO_{F_w})$. Let $F_w^u$ denote the ring of integers of the maximal unramified extension of $F_w$. Lang's theorem implies that $H^1(F_w^u / F^+_v, U_n(\cO_{F_w^u}))$ is trivial, and hence that we can find $h \in U_n(\cO_{F_w^u})$ such that $z = h^{-1} {}^\sigma h$ (where $\sigma$ denotes arithmetic Frobenius). Then $\Ad(h^{-1}) \circ \xi : U_{n, \cO_{F_w^u}} \to G_{\cO_{F_w^u}}$ descends to an isomorphism $U_{n, \cO_{F^+_v}} \to G_{\cO_{F^+_v}}$, and we can take $\iota_v^{-1}$ to be the map obtained after passage to $F^+_v$-points.  
\end{proof}
Henceforth we fix a system of isomorphisms $\iota_v$ satisfying the conclusion of the lemma. If $v$ is a finite place of $F^+$ which splits $v = w w^c$ in $F$, then we also write $\iota_w$ for the isomorphism $G(F^+_v) \to \GL_n(F_w)$ given by canonical inclusion $G(F^+_v) \subset \GL_n(F^+_v \otimes_{F^+} F) = \GL_n(F_w) \times \GL_n(F_{w^c})$ followed by projection to $\GL_n(F_w)$. Then $\iota_w(G(\cO_{F^+_v})) = \GL_n(\cO_{F_w})$.

Automorphic representations of $G(\bA_{F^+})$ are related to conjugate self-dual automorphic representations of $\GL_n(\bA_F)$ by base change, and have associated Galois representations. Theorem \ref{thm_base_change}, Corollary \ref{cor_Gal_reps_for_G} and Theorem \ref{thm_descent_for_cusp_forms} below follow from \cite[Theorem 1.2, Corollary 1.3, Theorem 1.4]{New21}, where they are deduced from the results of \cite{Lab11}; we state them again here for ease of reference.
\begin{theorem}\label{thm_base_change}
Let $\sigma$ be an automorphic representation of $G(\bA_{F^+})$. Then we can find a partition $n = n_1 + \dots + n_k$ and discrete, conjugate self-dual automorphic representations $\pi_i$ of $\GL_{n_i}(\bA_F)$ ($i = 1, \dots, k$) with the following properties:
\begin{enumerate}
    \item Let $\pi = \pi_1 \boxplus \dots \boxplus \pi_k$. Then for any finite place $v$ of $F^+$ such that $\sigma_v$ is unramified, and any place $w | v$ of $F$, $\pi_w$ is unramified and is related to $\sigma_v$ by unramified base change.
    \item For each place $v$ of $F^+$ which splits $v = w w^c$ in $F$, there is an isomorphism $\pi_w \cong \sigma_v \circ \iota_w^{-1}$.
    \item For each place $v | \infty$ of $F$, $\pi_v$ has the same infinitesimal character as the algebraic representation $\otimes_{\tau : F_v \to \bC} W_\tau$, where $W_\tau$ is the unique algebraic representation of $\GL_n(F_v) \cong \GL_n(\bC)$ such that $W_\tau|_{G(F^+_v)} \cong \sigma_v$. 
\end{enumerate}
\end{theorem}
\begin{corollary}\label{cor_Gal_reps_for_G}
Let $\sigma$ be an automorphic representation of $G(\bA_{F^+})$, let $p$ be a prime number, and let $\iota : \overline{\bQ}_p \to \bC$ be an isomorphism. Then there exists a continuous semisimple representation $r_{\sigma, \iota} : G_F \to \GL_n(\overline{\bQ}_p)$ satisfying the following conditions:
\begin{enumerate}
    \item For any finite, prime-to-$p$ place $v$ of $F^+$ such that $\sigma_v$ is unramified, and for any place $w | v$ of $F$, $r_{\sigma, \iota}|_{G_{F_w}}$ is unramified.
    \item For each place $w | p$ of $F$, $r_{\sigma, \iota}|_{G_{F_w}}$ is de Rham. 
    \item For each place $v$ of $F^+$ which splits $v = w w^c$ in $F$, there is an isomorphism $\mathrm{WD}(r_{\sigma, \iota}|_{G_{F_w}})^{F-ss} \cong \rec_{F_v}^T(\sigma_v \circ \iota_w^{-1})$. 
\end{enumerate}
\end{corollary}
In the situation of Corollary \ref{cor_Gal_reps_for_G}, it is useful to note that if there is a split place $v = w w^c$ such that $\sigma_w$ is not generic (in the sense of having a Whittaker model), then $r_{\sigma, \iota}$ has two irreducible subquotients which differ by a twist of the cyclotomic character (see the formula for $r_{\sigma, \iota}$ given in the proof of \cite[Corollary 3.4]{Ana21} -- this happens precisely when one of the constituents $\pi_i$ appearing in the statement of Theorem \ref{thm_base_change} is not cuspidal). Conversely, if the residual representation $\overline{r}_{\sigma, \iota}$ satisfies the genericity condition (3) of Theorem \ref{thm_intro_thm}, then $\sigma_w$ must be generic at every split place $w$. 
\begin{theorem}\label{thm_descent_for_cusp_forms}
Let $\pi$ be a RACSDC automorphic representation of $\GL_n(\bA_F)$. Suppose that for each place $w$ of $F$ such that $\pi_w$ is ramified, $w$ is split over $F^+$. Then there exists an automorphic representation $\sigma$ of $G(\bA_{F^+})$ satisfying the following conditions:
\begin{enumerate}
    \item For each finite place $v$ of $F^+$ which is inert in $F$, $\sigma_v^{\iota_v^{-1}(U_n(\cO_{F^+_v}))} \neq 0$ and if $w$ denotes the unique place of $F$ lying above $v$, then $\pi_w$ is related to $\sigma_v$ by unramified base change.
    \item For each finite place $v$ of $F$ which splits $v = w w^c$ in $F$, $\pi_w \cong \sigma_v \circ \iota_w^{-1}$.
    \item  For each place $v | \infty$ of $F$, $\pi_v$ has the same infinitesimal character as the algebraic representation $\otimes_{\tau : F_v \to \bC} W_\tau$, where $W_\tau$ is the unique algebraic representation of $\GL_n(F_v) \cong \GL_n(\bC)$ such that $W_\tau|_{G(F^+_v)} \cong \sigma_v$. 
\end{enumerate}
\end{theorem}
The rest of this section is devoted to an analogue of Theorem \ref{thm_descent_for_cusp_forms} for conjugate self-dual automorphic representations of $\GL_n(\bA_F)$ which are not cuspidal. It is a special case of the endoscopic classification of automorphic representations of $G(\bA_{F^+})$. It would follow from the results of \cite{Kal14}. Here it is relatively simple to deduce what we need from the results of \cite{Lab11}, following similar lines to the arguments given in \cite{New21}.
\begin{theorem}\label{thm_endoscopic_forms_exist}
Fix a partition $n = n_1 + n_2$ and let $\pi_1, \pi_2$ be cuspidal, conjugate self-dual automorphic representations of $\GL_{n_1}(\bA_F), \GL_{n_2}(\bA_F)$ such that $\pi = \pi_1 \boxplus \pi_2$ is regular algebraic. Suppose that the following conditions are satisfied:
\begin{enumerate}
    \item\label{ass_even} Let $\lambda_i = (\lambda_{i, \tau}) \in (\bZ_+^{n_i})^{\Hom(F, \bC)}$ be the weight of $\pi_i| \cdot |^{(n_i-n)/2}$. Then for each $i = 1, 2$, for any embedding $\tau_0 : F_0 \to \bC$ and any $(\mu_1, \mu_2) \in \bZ_+^{n_1} \times \bZ_+^{n_2}$, the number of $\tau \in \Hom(F, \bC)$ such that $\tau|_{F_0} = \tau_0$ and $(\lambda_{1, \tau}, \lambda_{2, \tau}) = (\mu_1, \mu_2)$ is even. (For example, this condition holds if $\pi_1, \pi_2$ arise by base change from a quadratic CM extension $F / F'$.) 
    \item If $w$ is a finite place of $F$ such that $\pi_w$ is ramified, then $w$ is split over $F^+$.
\end{enumerate}
Then there exists an automorphic representation $\sigma$ of $G(\bA_{F^+})$ with the following properties:
\begin{enumerate}
    \item For each finite place $v$ of $F^+$ which is inert in $F$, $\sigma_v^{\iota_v^{-1}(U_n(\cO_{F^+_v}))} \neq 0$ and if $w$ denotes the unique place of $F$ lying above $v$, then $\pi_w$ is related to $\sigma_v$ by unramified base change.
    \item For each finite place $v$ of $F^+$ which splits $v = w w^c$ in $F$, $\pi_w \cong \sigma_v \circ \iota_w^{-1}$.
    \item For each infinite place $v$ of $F^+$, $\pi_v$ has the same infinitesimal character as the representation $\otimes_{\tau : F_v \to \bC} W_\tau$ of $\GL_n(F_v)$, where $W_\tau$ is the irreducible algebraic representation of $\GL_n(F_v) \cong \GL_n(\bC)$ whose restriction to $G(F^+_v) \subset \GL_n(F_v)$ contains $\sigma_v$. 
\end{enumerate}
\end{theorem}
\begin{proof}
The theorem is proved using a comparison of trace formulae, using the results of \cite{Lab11}. Most of the the necessary set-up is detailed in \cite[\S 1]{New21}, see especially \cite[\S 1.5]{New21} for a description of the equivalence classes of endoscopic data for $G$ and a normalisation of local transfer factors under the assumption that $n$ is odd. We first describe a normalisation of local transfer factors also in the case that $n$ is even. As outlined in \cite[\S 1.5]{New21}, the choice of pinning determines the Langlands--Shelstad \cite{Lan87} normalisation of local transfer factors for the quasi-split group $U_n$ which can be transferred, using the fixed structure of pure inner twist, to a normalisation of local transfer factors $\Delta_v^\cE$ for the group $G$ (with respect to any given endoscopic datum $\cE$). It follows from \cite[Proposition 4.4.1]{Kal18} that this normalisation of local transfer factors satisfies the adelic product formula. More precisely, if a choice of non-trivial character $\psi : F^+ \backslash \bA_{F^+} \to \bC$ is fixed, one can define the associated Whittaker normalisation of the local transfer factors for $U_n$, hence for $G$, which differs from $\Delta_v^\cE$ by a local root number $\epsilon_v$, a sign which is equal to 1 for all but finitely many places $v$ of $F^+$, and \cite[Proposition 4.4.1]{Kal18} shows that the local transfer factors $\epsilon_v \Delta_v^\cE$ satisfy the adelic product formula. However, the product $\prod_v \epsilon_v$ is 1 (as it is the root number of a real representation -- this observation appears also in the proof of the cited proposition). 

We can then state the following result, which is a generalisation of \cite[Proposition 1.7]{New21} to include the case where $n$ is even:
\begin{proposition} Let $v$ be a finite place of $F^+$, and let $f_v \in C_c^\infty(G(F^+_v))$. Suppose given an extended endoscopic triple $\cE = (H, s, \eta)$.
\begin{enumerate}
    \item Suppose that $v$ is inert in $F$ and that $f_v$ is $\iota_v^{-1}(U_n(\cO_{F^+_v}))$-biinvariant. Suppose given an unramified Langlands parameter $\varphi_H : W_{F^+_v} \to {}^L H$ and let $\sigma_{v, H}$, $\sigma_v$ be the  $H(\cO_{F^+_v})$-unramified and $\iota_v^{-1}(U_n(\cO_{F^+_v}))$-unramified irreducible representations of $H(F^+_v)$ and $G(F^+_v)$ associated to $\varphi_H$ and $\eta \circ \varphi_H$, respectively. Let $\pi_v$ be the unramified irreducible representation of $M^H(F^+_v)$ associated to $\varphi_H|_{W_{F_v}}$. Then there is an identity
    \[ \widetilde{\pi}_v(\widetilde{f}_v^H) = \sigma_{v, H}(f_v^H) = \sigma_v(f_v), \]
    where the twisted trace is normalised so that $\theta$ fixes the unramified vector of $\pi$. (If $\pi$ is generic, then this agreed with the Whittaker normalisation of the twisted trace.)
    \item Suppose that $v = w w^c$ splits in $F$. Suppose given a bounded Langlands parameter $\varphi_H : W_{F^+_v} \to {}^L H$ and let $\sigma_{v, H}$, $\sigma_v$ be the irreducible representations of $H(F^+_v)$ and $G(F^+_v)$ associated to the parameters $\varphi_H$ and $\eta \circ \varphi_H$, respectively, by the local Langlands correspondence for general linear groups. Then $\pi_v$ be the irreducible representation of $M^H(F^+_v)$ associated to $\varphi_H|_{W_{F_w}}$. Then there is an identity
    \[ \widetilde{\pi}_v(\widetilde{f}_v^H) = \sigma_{v, H}(f_v^H) = \sigma_v(f_v), \]
    where the twisted trace is Whittaker normalised.
\end{enumerate}
\end{proposition}
The proof is essentially the same as the proof of \cite[Proposition 1.7]{New21}, which relies only on the properties of the local transfer factors associated to the $F^+_v$-pinning of the quasi-split group $U_n$. Here is another result valid at the infinite places, which is a generalisation of \cite[Proposition 1.6]{New21} that includes the case that $n$ is even:
\begin{proposition}
Let $v$ be an infinite place of $F^+$. Suppose given an extended endoscopic triple $\cE = (H, s, \eta)$ for $G$ and a Langlands parameter $\varphi_H : W_{F^+_v} \to {}^L H$ such that $\eta \circ \varphi_H$ is the Langlands parameter of an irreducible representation $\sigma_v$ of $G(F^+_v)$. Let $\pi_v$ be the (necessarily tempered, $\theta$-invariant) irreducible admissible representation of $H(F_v)$ associated to the Langlands parameter $\varphi_H|_{W_{F_v}}$, and let $f_v \in C_c^\infty(G(F^+_v))$ be a coefficient for $\sigma_v$. Then there is a sign $\epsilon(v, \cE, \varphi_H) \in \{ \pm 1 \}$ such that the identities
\[ \widetilde{\pi}_v(\widetilde{f}^H_v)  = \epsilon(v, \cE, \varphi_H) \sigma_v(f_v) = \epsilon(v, \cE, \varphi_H) \]
hold, where the twisted trace is Whittaker normalised.
\end{proposition}
The proof is again the same as the proof of the proof of  \cite[Proposition 1.6]{New21}, taking into account our normalisation of local transfer factors in the case $n$ even.  Repeating the argument of \cite[Proposition 4.6]{New21} (if $n_1 \neq n_2$) or \cite[\S 3.5]{Ana21} (if $n_1 = n_2$) and using the above two propositions, we see that the multiplicity of $\sigma$ as an automorphic representation of $G(\bA_{F^+})$ is $\frac{1}{2}(1 + \prod_{v | \infty} \epsilon(v, \cE, \varphi_H))$, where $\cE$ is the endoscopic triple associated to the group $H = U_{n_1} \times U_{n_2}$. We claim that the product of signs $\prod_{v | \infty} \epsilon(v, \cE, \varphi_H)$ equals 1. Indeed, we are assuming that $F$ has the form $F^+ {F_0}$ and that our choice of pure inner twist is defined over ${F_0}$. It follows that if $v, v'$ are infinite places of $F$ then there is a canonical isomorphism $F_v \cong \bR \otimes_\bQ  {F_0}  \cong F_{v'}$ and that if $\sigma_v$ and $\sigma_{v'}$ are identified under this isomorphism, then $\epsilon(v, \cE, \varphi_H) = \epsilon(v', \cE, \varphi_H)$. Our claim therefore follows from assumption \ref{ass_even} of Theorem \ref{thm_endoscopic_forms_exist}. 
\end{proof}
	
	\section{Analysis of a derived tensor product}

We continue with the set-up of the previous section. Thus we let ${F_0}$ be an imaginary quadratic field and let $F$ be a CM number field of the form $F = F^+ {F_0}$, such that $F / F^+$ is everywhere unramified, and take $G$ to be the reductive group over $\cO_{F^+}$ defined by (\ref{eqn_defn_of_G}). We need to set up some more notation. Recall that we have fixed for each finite place $v$ of $F^+$ an isomorphism $\iota_v :  G(F^+_v) \to U_n(F^+_v) $. Let $p$ be a prime number, and let $E / \bQ_p$ be a coefficient field. Let $S_p$ denote the set of $p$-adic places of $F^+$. If $S$ is a finite set of places of $F^+$, then we write $\cJ_S$ for the set of open compact subgroups $K = \prod_v K_v \subset G(\bA_{F^+}^\infty)$ which are sufficiently small, in the sense that such that if $g \in G(\bA_{F^+}^\infty)$ then $G(F^+) \cap g K g^{-1}$ is torsion-free (hence trivial), and moroever such that for each $v\notin S$, $K_v = \iota_v^{-1}(U_n(\cO_{F^+_v}))$. If $v_0 \in S$, then we write $\cJ_S^{v_0}$ for the set of open compact subgroups $K = \prod_{v \neq v_0} K_v \subset G(\bA_{F^+}^{\infty, v_0})$ which are sufficiently small, in the sense that for each $g \in G(\bA_{F^+}^{\infty, v_0})$ the group $G(F^+) \cap g K g^{-1}$ (intersection in $G(\bA_{F^+}^{\infty, v_0})$) is torsion-free, and such that if $v\notin S$, then $K_v = \iota_v^{-1}(U_n(\cO_{F^+_v}))$. If $K = \prod_v K_v \subset G(\bA_{F^+}^\infty)$ is an open compact subgroup such that $K^{v_0} \in \cJ_S^{v_0}$, then $K \in \cJ_S$. (The finite $\cO_{F^+}$-congruence subgroups of $K$ are contained in the torsion-free $\cO_{F^+}[1/v_0]$-congruence subgroups of $K^{v_0}$.)

We write $\T^S \subset \cH(G(\bA_{F^+}^{\infty, S}), K^S) \otimes_\bZ \cO$ for the $\cO$-subalgebra generated by the operators $\iota_w(T_w^i)$, $\iota_w((T_w^n)^{-1})$, for each split place $v = w w^c$ of $F^+$ with $v \not\in S$. Here $T_w^i$ denotes the usual unramified Hecke operator
\[ T_w^i = [ \GL_n(\cO_{F_w}) \diag(\underbrace{\varpi_w, \dots, \varpi_w}_{i}, \underbrace{1, \dots, 1}_{n-i}) \GL_n(\cO_{F_w})] \in \cH(\GL_n(F_w), \GL_n(\cO_{F_w})). \]
We define as usual the polynomial
\[ P_w(X) = \sum_{i=0}^n (-1)^i X^{n-i} q_w^{i(i-1)/2} T_w^i \in \T^S[X] \]
(characteristic polynomial of Frobenius under $\rec_{F_w}^T$).

Now fix a prime number $q \neq p$ which splits $q = u_0 u_0^c$ in ${F_0}$, let $v_0$ be a place of $F^+$ dividing $q$, and let $w_0$ be the unique place of $F$ lying above both $u_0$ and $v_0$. Let $S$ be a finite set of places of $F^+$ containing $S_p$ and such that $S \cap S_q = \{ v_0 \}$. Let $K^{v_0} \in \cJ_S^{v_0}$. We define
\[ \cA = \cA(K^{v_0}, k) = C_c^\infty( G(F^+) \backslash G(\bA_{F^+}^\infty) / K^{v_0}, k). \]
We view $\cA$ as a $k[\GL_n(F_{w_0})]$-module using $\iota_{w_0}$. Then $\cA \in \Mod_{sm}(k[\GL_n(F_{w_0})])$ is admissible and there is a natural map 
\[ \T^S \to \End_{\Mod_{sm}(k[\GL_n(F_{w_0})])}(\cA). \]
\begin{proposition} Let notation and assumptions be as above. Then:
\begin{enumerate}
    \item If $\m \subset \T^S$ is a maximal ideal in the support of $\cA$, then the localisation $\cA_\m$ is naturally isomorphic to $\varinjlim_{K_{v_0}} (\cA^{K_{v_0}})_\m$. In particular, the residue field of $\m$ is a finite extension of $k$, $\m$ occurs in the support of $\cA^{K_{v_0}}$ for some open compact subgroup $K_{v_0} \subset G(F^+_{v_0})$, and $\cA_\m$ may be identified with a $\T^S$-invariant subspace of $\cA$.
    \item There is a $\T^S$-invariant direct sum decomposition $\cA = \oplus_\m \cA_\m$, where the sum runs over the set of maximal ideals of $\T^S$ which are in the support of $\cA$.
    \item  If $\m$ is a maximal ideal of $\T^S$ in the support of 
    \[ H^\ast( R \Gamma_c( \Omega_{F_{w_0}}^{n-1, ca}, k) \otimes^{\bL}_{k[\GL_n(F_{w_0})^0]} \cA), \]
    then there is an isomorphism of $\T^S$-modules
    \[ H^\ast( R \Gamma_c( \Omega_{F_{w_0}}^{n-1, ca}, k) \otimes^{\bL}_{k[\GL_n(F_{w_0})^0]} \cA)_\m \cong H^\ast( R \Gamma_c( \Omega_{F_{w_0}}^{n-1, ca}, k) \otimes^{\bL}_{k[\GL_n(F_{w_0})^0]} \cA_\m). \]
    In particular, $\m$ occurs in the support of $\cA$.
\end{enumerate}
\end{proposition}
\begin{proof}
    We have $\cA = \varinjlim_{K_{v_0}} \cA^{K_{v_0}}$ as $\T^S$-modules, where the direct limit runs over the set of open compact subgroups $K_{v_0} \subset G(F^+_{v_0})$. Since localisation commutes with direct limits, this shows that for any maximal ideal $\m \subset \T^S$, we have
    \[  \cA_\m = \varinjlim_{K_{v_0}} (\cA^{K_{v_0}})_\m. \]
    In particular, $\cA_\m$ is non-zero if and only if some $(\cA^{K_{v_0}})_\m$ is non-zero. Each $\cA^{K_{v_0}}$ is finite-dimensional as $k$-vector space, which shows that if $\cA_\m$ is non-zero then the residue field of $\m$ is a finite extension of $k$ and that $(\cA^{K_{v_0}})_\m$ may be identified with the submodule $\cA^{K_{v_0}}[\m^\infty] \subset \cA^{K_{v_0}}$ of $\m^\infty$-torsion. By passage to the direct limit, we find that $\cA_\m$ may be identified with the submodule of $\cA$ of $\m^\infty$-torsion. This shows the first and second parts of the proposition. For the third, we note that there is a $\T^S$-invariant direct sum decomposition
    \[  H^\ast( R \Gamma_c( \Omega_{F_{w_0}}^{n-1, ca}, k) \otimes^{\bL}_{k[\GL_n(F_{w_0})^0]} \cA) = \oplus_\m H^\ast( R \Gamma_c( \Omega_{F_{w_0}}^{n-1, ca}, k) \otimes^{\bL}_{k[\GL_n(F_{w_0})^0]} \cA_\m). \]
    It suffices therefore to show that if $H^\ast( R \Gamma_c( \Omega_{F_{w_0}}^{n-1, ca}, k) \otimes^{\bL}_{k[\GL_n(F_{w_0})^0]} \cA_\m)$ is non-zero, then $\m$ is the only maximal ideal of $\T^S$ in its support. It even suffices to show that the natural map
   \begin{multline*} \varinjlim_N H^\ast( R \Gamma_c( \Omega_{F_{w_0}}^{n-1, ca}, k) \otimes^{\bL}_{k[\GL_n(F_{w_0})^0]} \cA[\m^N]) \\ \to  H^\ast( R \Gamma_c( \Omega_{F_{w_0}}^{n-1, ca}, k) \otimes^{\bL}_{k[\GL_n(F_{w_0})^0]} \cA_\m) 
   \end{multline*}
    is an isomorphism. This is a consequence of the fact that if $A^\bullet \in \mathbf{D}^-(k[\GL_n(F_{w_0})^0])$ then the functor $A^\bullet \otimes^{\bL}_{k[\GL_n(F_{w_0})^0]} - $ commutes with direct limits. 
\end{proof}
A standard consequence of Corollary \ref{cor_Gal_reps_for_G} and the theory of algebraic modular forms states that if $\m$ occurs in the support of $\cA^{K_{v_0}}$, then there is a continuous semisimple representation $\overline{\rho}_\m : G_{F, S} \to \GL_n(\T^S / \m)$ such that for each finite place $w \nmid S$ of $F$ which splits over $F^+$, the characteristic polynomial $\det(X - \overline{\rho}_\m(\Frob_w))$ equals $P_w(X)$. The goal of this section is to prove the following theorem concerning the finite-dimensional $k$-vector spaces $H^j( R \Gamma_c( \Omega_{F_{w_0}}^{n-1, ca}, k) \otimes^{\bL}_{k[\GL_n(F_{w_0})^0]} \cA)$:
\begin{theorem}\label{thm_torsion_in_derived_tensor_product}
With notation and assumptions as above, let $\m \subset \T^S$ be a maximal ideal of residue field $k$ with the following property:
\begin{enumerate}
    \item There exists a finite place $w \nmid S$ of $F$ such that $\overline{\rho}_\m|_{G_{F_w}}^{ss} \cong \oplus_{i=1}^n \overline{\chi}_i$, where $\overline{\chi}_i : G_{F_w} \to k^\times$ are unramified characters such that if $i \neq j$ then $\overline{\chi}_i / \overline{\chi}_j \neq \epsilon$. 
\end{enumerate}
Then $H^j( R \Gamma_c( \Omega_{F_{w_0}}^{n-1, ca}, k) \otimes^{\bL}_{k[\GL_n(F_{w_0})^0]} \cA )_\m = 0$ if $j \neq n-1$.
\end{theorem}
We will prove this theorem by comparing the cohomology groups in the statement of the theorem with the cohomology groups of a $q$-adically uniformized Shimura variety using \cite{Rap96}, and then using the vanishing theorems proved in \cite{Boy19} or \cite{Car17}. We begin by finding PEL data that will allow a comparison with the group $G$. Let $D$ be a central simple algebra over $F$ of rank $n$ and let $\ast$ be a positive involution of $D$ (therefore of the second kind). Let $V = D$, considered as left $D$-module, and let $\psi  : V \times V \to \bQ$ be a non-degenerate alternating bilinear form such that for all $d \in D$, $v, w \in V$, we have $\psi (dv, w) = \psi (v, d^\ast w)$. Then the tuple $(D, \ast, V, \psi)$ is a PEL datum and gives rise to a Shimura datum $(\widetilde{H}, X)$, where $\widetilde{H}$ is the reductive group over $\bbQ$ whose functor of points is
\[ \widetilde{H}(R) = \{ (g, \mu) \in \End_{D \otimes_\bQ R}(V \otimes_\bQ R) \times R^\times \mid \forall v, w \in V \otimes_\bQ R, \psi(gv, gw) = \mu \psi(v, w) \}, \]
and $X$ is the set of homomorphisms $h : \bC^\times \to \widetilde{H}(\bR)$ which arise by restriction from an $\bR$-algebra homomorphism $h : \bC \to \End_{D \otimes_\bQ \bR}(V \otimes_\bQ \bR)$ such that the pairing $(v, w) = \psi(v, h(i) w)$ is symmetric and either positive or negative definite. (The set $X$ forms a single $\widetilde{H}(\bR)$-conjugacy class, by \cite[Lemma 4.3]{Kot92}, and the pair $(\widetilde{H}, X)$ defines a Shimura datum, by \cite[Lemma 4.1]{Kot92}.) If $\tau : F \to \bC$ is an embedding, we define numbers $p_\tau, q_\tau$ as follows: choose $h \in X$ such that $(v, w) = \psi(v, h(i) w)$ is positive definite. Then $\dim_\bC (V \otimes_{F, \tau} \bC)^{h(i) = i} = n p_\tau$ and $\dim_\bC (V \otimes_{F, \tau} \bC)^{h(i) = -i} = n q_\tau$ (so $p_\tau + q_\tau = n$). 
\begin{lemma}\label{lem_existence_of_PEL_datum}
Fix an embedding $\tau : F \to \bC$. Then we can find data $(D, \ast, V, \psi)$ such that the following conditions are satisfied:
\begin{enumerate}
    \item $\operatorname{inv} D_{w_0} = 1 / n$, $\operatorname{inv} D_{w_0^c} = - 1 / n$, and $\operatorname{inv} D_w = 0$ for every place $w \neq w_0, w_0^c$ of $F$.
    \item For each prime $r \neq q$, the group $\widetilde{H}_{\bQ_r}$ is quasi-split.
    \item We have $(p_\tau, q_\tau) = (1, n-1)$ and $(p_{\tau'}, q_{\tau'}) = (0, n)$ for any $\tau' : F \to \bC$ such that $\tau'|_{F_0} = \tau|_{F_0}$ but $\tau' \neq \tau$. 
\end{enumerate}
\end{lemma}
\begin{proof}
This is \cite[Lemma I.7.1]{Har01}. It uses the fact that $[F^+ : \bQ]$ is even. 
\end{proof}
We henceforth fix a choice of $\tau : F \to \bC$ and data $(D, \ast, V, \psi)$ as in the statement of the lemma. Then the reflex field of the Shimura datum $(\widetilde{H}, X)$ equals $\tau(F)$ and for any sufficiently small open compact subgroup $\widetilde{K} \subset \widetilde{H}(\bA_\bQ^\infty)$, there is an associated Shimura variety $\operatorname{Sh}_{\widetilde{K}}$ over $\tau(F)$. Let $\iota : \overline{\bQ}_p \to \bC$ be an isomorphism such that $\iota^{-1} \circ \tau$ induces the place $w_0$ of $F$. We write $\widetilde{G}$ for the unitary similitude group associated to $G$. More precisely, $\widetilde{G}$ is the reductive group over $\bQ$ whose functor of points is given by
\[ \widetilde{G}(R) = \{ (g, \mu) \in \GL_n(F \otimes_\bQ R) \times R^\times \mid g (c \otimes 1)(g)^{t} = \mu \}. \]
If $R$ is an ${F_0}$-algebra, then we can use the isomorphism $F \otimes_\bQ R \cong F \otimes_{F_0} R \times F \otimes_{{F_0}, c} R$ to identify
\[ \widetilde{G}(R) \cong \GL_n(F \otimes_{F_0} R) \times R^\times. \]
In particular, viewing $\bQ_q$ as an ${F_0}$-algebra via the isomorphism $F_{0,u_0} \cong \bQ_q$, we can identify
\begin{equation}\label{eqn_split_prime} \widetilde{G}(\bQ_q) \cong \prod_{w | u_0} \GL_n(F_{w}) \times \bQ_q^\times. \end{equation}
This identification appears in the statement of the following theorem. 
\begin{theorem}\label{thm_rap_zink}
Let $\widetilde{K}^q \subset \widetilde{G}(\bA_\bQ^{\infty, q})$ be a sufficiently small open compact subgroup, and let 
\[ \widetilde{K}_q = \GL_n(F_{w_0})^0 \times \prod_{w | u_0, w \neq w_0} \GL_n(\cO_{F_w}) \times \bZ_p^\times, \]
an open (but not compact) subgroup of $\widetilde{G}(\bQ_q)$. Then we can find an isomorphism $f : \widetilde{G}(\bA_\bQ^{\infty, q}) \to \widetilde{H}(\bA_\bQ^{\infty, q})$ and an $\cH(\widetilde{G}(\bA_\bQ^{\infty, q}), \widetilde{K}^q)$-equivariant isomorphism of rigid analytic spaces over $\widehat{\overline{F}}_{w_0}$:
    \begin{equation}\label{eqn_q-adic_uniformization}  \widetilde{G}(\bQ) \backslash \left( \Omega_{F_{w_0}}^{n-1, ca} \times \widetilde{G}(\bA_\bQ^\infty) / \widetilde{K} \right) \cong \operatorname{Sh}_{\widetilde{K}_{\widetilde{H}}}^{rig, ca}, 
    \end{equation}
    where $\widetilde{K}_{\widetilde{H}}^q = f(\widetilde{K}^q)$ and $\widetilde{K}_{\widetilde{H}, q} \subset \widetilde{H}(\bQ_q)$ is the  maximal compact subgroup described in \cite[Proposition 6.49]{Rap96}.
\end{theorem}
The quotient in the left-hand side (\ref{eqn_q-adic_uniformization}) can be formed in the naive sense: the group $\widetilde{G}(\bQ)$ is discrete and each point of $\Omega_{F_{w_0}}^{n-1, ca} \times \widetilde{G}(\bA_\bQ^\infty) / \widetilde{K}$ admits an open affinoid neighbourhood $\cU$ such that if $\gamma \in \widetilde{G}(\bQ)$ and $\gamma \cU \cap \cU \neq \emptyset$, then $\gamma = 1$.
\begin{proof}
Let $C = \End_D(V)$, a central simple $F$-algebra which is isomorphic to $D^{op}$. The conditions of \cite[\S 6.40]{Rap96} are satisfied and we can invoke \cite[Theorem 6.50]{Rap96}, which has the following consequences.  We can find a central simple $F$-algebra $B$, together with a positive involution $\dagger$ and an isomorphism $B \otimes_\bQ \bA^{\infty, q}_\bQ \cong C \otimes_\bQ \bA_\bQ^{\infty, q}$ which identifies $\dagger$ with adjoint with respect to $\psi$, with the following properties:
\begin{itemize}
    \item For each place $w$ of $F$, $B_{F_w}$ is split. Consequently, we can find an isomorphism $B \cong M_n(F)$ of $F$-algebras.
    \item Let $\widetilde{I}$ denote the unitary similitude group of the pair $(B, \dagger)$, so that there is an induced isomorphism $\widetilde{I}(\bA_{\bQ}^{\infty, q}) \cong \widetilde{H}(\bA_{\bQ}^{\infty, q})$. For any neat open compact subgroup $\widetilde{U}^q \subset \widetilde{I}(\bA_\bQ^{\infty, q})$, there is an 
    $\cH(\widetilde{I}(\bA_\bQ^{\infty, q}), \widetilde{U}^q)$-equivariant isomorphism of rigid analytic spaces over $\widehat{\overline{F}}_{w_0}$:
    \[ \widetilde{I}(\bbQ) \backslash \left( \Omega_{F_{w_0}}^{n-1, ca} \times \widetilde{H}(\bA_\bQ^\infty) / \widetilde{U} \right) \cong \operatorname{Sh}_{\widetilde{U}}^{rig, ca}, \]
    where $\widetilde{U}_q \subset \widetilde{H}(\bQ_q)$ is a certain maximal compact subgroup and the action of $\widetilde{I}(\bQ_q)$ on $\Omega_{F_{w_0}}^{n-1, ca} \times \widetilde{H}(\bQ_q) / \widetilde{U}_q$ is as described in \cite[Proposition 6.49]{Rap96}.
\end{itemize}
 Let $G^{ad}$ denote the adjoint group of $G$, and let $I$ denote the unitary group over $F^+$ defined by $(B, \dagger)$. Since $\widetilde{I}_{\bQ_r}$ is quasi-split for every prime $r$, $I_{F^+_v}$ is quasi-split for every finite place $v$ of $F^+$. It follows that for every place $v$ of $F^+$, the groups $G_{F^+_v}$ and $I_{F^+_v}$ are isomorphic (if $v$ is finite, these groups are both quasi-split; if $v$ is infinite, then they are both the compact form of $\GL_{n, F^+_v}$). Choose $g \in \GL_n(F)$ such that the involution $b \mapsto b^\dagger$ of $B$ is identified with the involution $X \mapsto g {}^t X^c g^{-1}$. Then $g {}^t g^c$ is a scalar matrix and $g$ defines a 1-cocycle in $Z^1(F / F^+, G^{ad})$. For any place $v$ of $F^+$, the image of this 1-cocycle in $H^1(F^+_v, \Aut(G))$ is trivial. The map  $H^1(F^+_v, G^{ad}) \to H^1(F^+_v, \Aut(G))$ has trivial pointed kernel, because the map $\Aut(G)(F^+_v) \to \operatorname{Out}(G)(F^+_v)$ is surjective. It follows that the image of our 1-cocycle in $H^1(F^+_v, G^{ad})$ is trivial. By the Hasse principle for $H^1(F^+, G^{ad})$  (proved using the description of $\ker^1$ given in \cite[\S 4]{Kot84}), we conclude that our 1-cocycle in $Z^1(F / F^+, G^{ad})$ is in fact a coboundary and that we can in fact choose the isomorphism $B \cong M_n(F)$ so that $\dagger$ is identified with the involution $X \mapsto {}^t X^c$ and $\widetilde{I} = \widetilde{G}$. This leads to the claimed statement.
\end{proof}
Before stating the result we need from \cite{Car17}, we need to introduce a bit more notation. Suppose given a finite set $T$ of prime numbers containing $p$, $q$, and all the primes which are ramified in $F$, and let $\widetilde{K}_{\widetilde{H}} \subset \widetilde{H}(\bA_{\bQ}^\infty)$ be a sufficiently small open compact subgroup of the form $\widetilde{K}_{\widetilde{H}} = \prod_\ell \widetilde{K}_{\widetilde{H}, \ell}$ such that for each $\ell \not\in T$, $\widetilde{K}_{\widetilde{H}, \ell} \subset \widetilde{H}(\bQ_\ell)$ is a hyperspecial maximal compact subgroup. Let $\operatorname{Spl}_{{F_0} / \bQ}$ denote the set of prime numbers which split in $F_0$. We define 
\[ \widetilde{\bT}^T \subset \otimes'_{\ell \in \operatorname{Spl}_{{F_0} / \bQ} - T} \cH(\widetilde{H}(\bQ_\ell), \widetilde{K}_{\widetilde{H}, \ell}) \otimes_\bZ \cO, \]
to be the $\cO$-subalgebra generated by all the Hecke operators $T_w^i$ and $(T_w^n)^{-1}$, as $w$ ranges over places of $F$ lying above a prime $\ell \in \operatorname{Spl}_{{F_0} / \bQ} - T$, and where these operators can be considered as elements of $\cH(\widetilde{H}(\bQ_\ell), \widetilde{K}_{\widetilde{H}, \ell})$ using the analogue of the isomorphism (\ref{eqn_split_prime}) for the group $\widetilde{H}$. 
\begin{theorem}\label{thm_car_sch}
Let $T$, $\widetilde{K}_{\widetilde{H}}$ be as in the previous paragraph. Let $T_{F^+}$ denote the set of places of $F^+$ lying above an element of $T$. Suppose given a maximal ideal $\widetilde{\m} \subset \widetilde{\T}^T$ which is in the support of $H^\ast(\operatorname{Sh}_{\widetilde{K}_{\widetilde{H}}}^{rig, ca}, k)$. Then:
\begin{enumerate}
    \item There exists a continuous semisimple representation 
    \[ \overline{\rho}_{\widetilde{\m}} : G_{F, T_{F^+}} \to \GL_n(\widetilde{\T}^S / \widetilde{\m}) \]
    such that for each prime number $\ell \in \operatorname{Spl}_{{F_0} / \bQ} - T$ and each place $w | \ell$ of $F$, $\det(X - \overline{\rho}_{\widetilde{\m}}(\Frob_w)) = P_w(X) \text{ mod }\widetilde{\m}$.
    \item Suppose further that there exists a finite place $w \nmid T$ of $F$ such that $\overline{\rho}_{\widetilde{\m}}|_{G_{F_w}}^{ss} \cong \oplus_{i=1}^n \overline{\chi}_i$, where $\overline{\chi}_i : G_{F_w} \to \overline{k}^\times$ ($i = 1, \dots, n$) are unramified characters such that if $i \neq j$ then $\overline{\chi}_i / \overline{\chi}_j \neq \epsilon$. Then if $H^j(\operatorname{Sh}_{\widetilde{K}_{\widetilde{H}}}^{rig, ca}, k)_{\widetilde{\m}} \neq 0$, then $j = n-1$.
\end{enumerate}
\end{theorem}
\begin{proof}
Let $\cH^T = \otimes'_{\ell \in \operatorname{Spl}_{{F_0} / \bQ} - T} \cH(\widetilde{H}(\bQ_\ell), \widetilde{K}_{\widetilde{H}, \ell}) \otimes_\bZ \cO$. Then \cite[Theorem 6.3.1(1)]{Car17} states that for any maximal ideal $\widetilde{\mathfrak{n}} \subset \cH^T$ that occurs in the support of $H^\ast(\operatorname{Sh}_{\widetilde{K}_{\widetilde{H}}}^{rig, ca}, k)$, there exists a continuous representation $\overline{\rho}_{\widetilde{\mathfrak{n}}} : G_{F, T_{F^+}} \to \GL_n( \cH^T / \widetilde{\mathfrak{n}} )$ such that for each prime number $\ell \in \operatorname{Spl}_{{F_0} / \bQ} - T$ and each place $w | \ell$ of $F$, $\det(X - \overline{\rho}_{\widetilde{\mathfrak{n}}}(\Frob_w)) = P_w(X) \text{ mod }\widetilde{\mathfrak{n}}$. If $\widetilde{\m} \subset \widetilde{\T}^S$ is a maximal ideal that is in the support of the cohomology, then $\widetilde{\m} = \widetilde{\T}^S \cap \widetilde{\mathfrak{n}}$ for some maximal ideal $\widetilde{\mathfrak{n}} \subset \cH^T$ that is in the support, and we can take $\overline{\rho}_{\widetilde{\m}}$ to be a conjugate of $\overline{\rho}_{\widetilde{\mathfrak{n}}}$ that is defined over  $\widetilde{\T}^S / \widetilde{\m}$. Such a conjugate exists because the conjugacy classes of elements $\Frob_w$ ($w$ a place of $F$ lying above $\operatorname{Spl}_{{F_0} / \bQ} - T$) are dense in $G_{F, T_{F^+}}$. This shows the first part.

For the second, we can decompose
\[  H^\ast(\operatorname{Sh}_{\widetilde{K}_{\widetilde{H}}}^{rig, ca}, k)_{\widetilde{\m}} = \oplus_{\widetilde{\mathfrak{n}}} H^\ast(\operatorname{Sh}_{\widetilde{K}_{\widetilde{H}}}^{rig, ca}, k)_{\widetilde{\mathfrak{n}}}, \]
where the sum runs over the set of maximal ideals of $\cH^T$ lying above $\widetilde{\m}$. If $\widetilde{\m}$ satisfies the hypothesis in the second part of the theorem then each $\widetilde{\mathfrak{n}}$ satisfies the hypothesis of \cite[Theorem 6.3.1(2)]{Car17} (using here \cite[Remark 6.3.3]{Car17}) and we deduce that each $H^\ast(\operatorname{Sh}_{\widetilde{K}_{\widetilde{H}}}^{rig, ca}, k)_{\widetilde{\mathfrak{n}}}$, and therefore $ H^\ast(\operatorname{Sh}_{\widetilde{K}_{\widetilde{H}}}^{rig, ca}, k)_{\widetilde{\m}}$, is concentrated in degree $n-1$. (In fact, \cite[Theorem 6.3.1(2)]{Car17} includes the additional condition that $\overline{\chi}_i /\overline{\chi}_j \neq 1$ if $i \neq j$, but this can be suppressed, as observed by Koshikawa, see \cite[Remark 1.4]{Car19}.) We remark that the \'etale cohomology of rigid spaces used in \cite{Dat06} is that of Berkovich \cite{Ber93}. Berkovich's paper contains the comparison theorem for the \'etale cohomology of schemes of finite type with tame torsion coefficients, which we invoke here.
\end{proof}
We can now complete the proof of Theorem \ref{thm_torsion_in_derived_tensor_product}.
\begin{proof}[Proof of Theorem \ref{thm_torsion_in_derived_tensor_product}]
Recall that we are given a maximal ideal $\m \subset \T^S$ and we need to show that $H^j( R \Gamma_c(\Omega_{F_{w_0}}^{n-1, ca}, k) \otimes^{\bL}_{k[\GL_n(F_{w_0})^0]} \cA)_\m$ is non-zero only if $j = n-1$. We first note that there is a $\T^S$-equivariant isomorphism 
\begin{multline}\label{eqn_isomorphism_of_cohomology_groups}
H^j( R \Gamma_c(\Omega_{F_{w_0}}^{n-1, ca}, k) \otimes^{\bL}_{k[\GL_n(F_{w_0})^0]} \cA) \cong \\ H^j( G(F^+) \backslash \left( \Omega_{F_{w_0}}^{n-1, ca} \times G(\bA_{F^+}^\infty) / K^{v_0} \iota_{w_0}^{-1}(\GL_n(F_{w_0})^0) \right), k). 
\end{multline} 
To see this, recall that 
\[ \cA = C_c^\infty(G(F^+) \backslash G(\bA_{F^+}^\infty) / K^{v_0}, k) = C_c^\infty(G(\bA_{F^+}^\infty) / K^{v_0}, k) \otimes^{\bL}_{k[G(F^+)]} k. \]
Using \cite[Proposition B.3.1]{Dat06}, we see that it is enough to construct a $\bT^S$-equivariant isomorphism
\begin{multline*}  R \Gamma_c(\Omega_{F_{w_0}}^{n-1, ca}, k) \otimes^{\bL}_{k[\GL_n(F_{w_0})^0]} C_c^\infty(G(\bA_{F^+}^\infty) / K^{v_0}, k) \cong \\ R \Gamma_c( \Omega_{F_{w_0}}^{n-1, ca}  \times  G(\bA_{F^+}^\infty) / K^{v_0} \iota_{w_0}^{-1}(\GL_n(F_{w_0})^0), k) 
\end{multline*} 
in $\mathbf{D}(k[G(F^+)])$. The right hand-side here may be further decomposed as
\[   R \Gamma_c( \Omega_{F_{w_0}}^{n-1, ca}, k) \otimes_k C_c^\infty( G(\bA_{F^+}^\infty) / K^{v_0} \iota_{w_0}^{-1}(\GL_n(F_{w_0})^0), k). \]
Since there is an isomorphism  
\[ C_c^\infty( G(\bA_{F^+}^\infty) / K^{v_0} \iota_{w_0}^{-1}(\GL_n(F_{w_0})^0), k) \cong C_c^\infty( G(\bA_{F^+}^\infty) / K^{v_0}, k) \otimes^{\bL}_{k[\GL_n(F_{w_0})^0]} k, \]
we see that it is even enough to construct a $\T^S$-equivariant isomorphism
\[ R \Gamma_c(\Omega_{F_{w_0}}^{n-1, ca}, k) \otimes_k C_c^\infty(G(\bA_{F^+}^\infty) / K^{v_0}, k) \cong  R \Gamma_c(\Omega_{F_{w_0}}^{n-1, ca}, k) \otimes_k C_c^\infty(G(\bA_{F^+}^\infty) / K^{v_0}, k) \]
in $\mathbf{D}(k[\GL_n(F_{w_0}) \times \GL_n(F_{w_0})])$, where $(g_1, g_2) \in \GL_n(F_{w_0}) \times \GL_n(F_{w_0})$ acts by $(g_1, g_2) \cdot (\omega \otimes f) = g_1 \omega \otimes L_{g_1} R_{g_2} f$ on the left-hand side and by $(g_1, g_2) \cdot (\omega \otimes f) = g_2 \omega \otimes L_{g_1} R_{g_2} f$ on the right-hand side (restricting these actions to the subgroup $G(F^+) \times \GL_n(F_{w_0})^0$ and passing to derived coinvariants then recovers the complexes whose cohomology groups appear in (\ref{eqn_isomorphism_of_cohomology_groups})). The existence of such an isomorphism is the content of Lemma \ref{lem_two_actions_of_G}. 

We next note that we are free to enlarge $S$ if necessary (replacing $\m$ by its intersection with the subalgebra $\bT^{S'} \subset \T^S$) and also to replace $K^{v_0}$ by an open normal subgroup $K_1^{v_0}$. Indeed, the Hochschild--Serre spectral sequence implies that if the Theorem holds for $K_1^{v_0}$ then the groups $H^j( R \Gamma_c(\Omega_{F_{w_0}}^{n-1, ca}, k) \otimes^{\bL}_{k[\GL_n(F_{w_0})^0]} \cA)_\m$ are at least non-zero only in degrees $j \geq n-1$. The isomorphism (\ref{eqn_isomorphism_of_cohomology_groups}), together with the fact that the quotient $G(F^+) \backslash \left( \Omega_{F_{w_0}}^{n-1, ca} \times G(\bA_{F^+}^\infty) / K^{v_0} \iota_{w_0}^{-1}(\GL_n(F_{w_0})^0) \right)$ is the rigid space associated to a smooth proper scheme, which therefore satisfies Poincar\'e duality for \'etale cohomology with $k$-coefficients, then implies that the truth of the Theorem for $K^{v_0}$. 

We can therefore assume that the following additional conditions hold:
\begin{itemize}
    \item There is a sufficiently small open compact subgroup $\widetilde{K}^q \subset \widetilde{G}(\bA_\bQ^{\infty, q})$ such that $\widetilde{K}^q \cap G(\bA_\bQ^{\infty, q}) = K^q$.
    \item Defining $\widetilde{K}_q$ as in the statement of Theorem \ref{thm_torsion_in_derived_tensor_product}, and writing $\mu : \widetilde{G} \to \bG_m$ for the similitude character, we have $\mu(\widetilde{K}) \cap \bQ^\times = \{ 1 \}$.
\end{itemize}
Let $T$ be a finite set of rational primes including $p$, $q$, and all those primes $l$ such that $\widetilde{K}_l$ is not a hyperspecial maximal compact subgroup of $\widetilde{G}(\bQ_l)$ or lying below an element of $S$. These conditions imply that the induced map
\[ G(F^+) \backslash \left( \Omega_{F_{w_0}}^{n-1, ca} \times G(\bA_{F^+}^\infty) / K^{v_0} \iota_{w_0}^{-1}(\GL_n(F_{w_0})^0) \right) \to \widetilde{G}(\bQ) \backslash \left( \Omega_{F_{w_0}}^{n-1, ca} \times \widetilde{G}(\bA_\bQ^\infty) / \widetilde{K} \right) \]
is an isomorphism onto a union of connected components, which is moreover equivariant for the action of the Hecke algebra $\widetilde{\T}^T \subset \T^S$. (Compare \cite[\S 3.1]{Kar21}.) Let $\widetilde{\m}$ denote the pullback of the ideal $\m$ to $\widetilde{\T}^T$. Then Theorem \ref{thm_car_sch} implies that $H^\ast(\widetilde{G}(\bQ) \backslash \left( \Omega_{F_{w_0}}^{n-1, ca} \times \widetilde{G}(\bA_\bQ^\infty) / \widetilde{K} \right), k)_{\widetilde{\m}}$ is concentrated in degree $n-1$; since the map 
\begin{multline*} H^\ast(\widetilde{G}(\bQ) \backslash \left( \Omega_{F_{w_0}}^{n-1, ca} \times \widetilde{G}(\bA_\bQ^\infty) / \widetilde{K} \right), k)_{\widetilde{\m}} \to \\ H^\ast( G(F^+) \backslash \left( \Omega_{F_{w_0}}^{n-1, ca} \times G(\bA_{F^+}^\infty) / K^{v_0} \iota_{w_0}^{-1}(\GL_n(F_{w_0})^0) \right), k)_{\widetilde{\m}} 
\end{multline*}
is surjective, and $H^\ast(G(F^+) \backslash \left( \Omega_{F_{w_0}}^{n-1, ca} \times G(\bA_{F^+}^\infty) / K^{v_0} \iota_{w_0}^{-1}( \GL_n(F_{w_0})^0 ) \right), k)_\m$ is a direct summand of $H^\ast(G(F^+) \backslash \left( \Omega_{F_{w_0}}^{n-1, ca} \times G(\bA_{F^+}^\infty) / K^{v_0} \iota_{w_0}^{-1}( \GL_n(F_{w_0})^0 ) \right), k)_{\widetilde{\m}}$, the desired result follows in view of the isomorphism (\ref{eqn_isomorphism_of_cohomology_groups}).
\end{proof}

\section{Raising the level}

We continue with the set-up of the previous section. Thus we let ${F_0}$ be an imaginary quadratic field, $F = F^+ {F_0}$ a CM field such that $F / F^+$ is everywhere unramified, and take $G$ to be the reductive group over $\cO_{F^+}$ whose functor of points is given by (\ref{eqn_defn_of_G}), with its associated isomorphisms $\iota_v : G(F^+_v) \to  U_n(F^+_v) $ for each finite place $v$ of $F^+$ and isomorphisms $\iota_w : G(F^+_v) \to \GL_n(F_w)$ for places $w$ of $F$ split over a place $v$ of $F^+$. We fix a prime number $p$, an isomorphism $\iota : \overline{\bQ}_p \to \bC$, a coefficient field $E / \bQ_p$, and a finite set $S$ of places of $F^+$, each of which splits in $F$, and containing the set $S_p$ of $p$-adic places. We assume that $E$ is large enough to contain the image of any embedding $F \to \overline{\bQ}_p$.

We also fix for each $v \in S$ a place $\wv$ of $F$ lying above $v$, and set $\widetilde{S} = \{ \wv \mid v \in S \}$. We write $\widetilde{I}_p$ for the set of embeddings $\tau : F \to \overline{\bQ}_p$ which induce a place of $\widetilde{S}_p$. We have defined $\cJ_S$ to be the set of sufficiently small open compact subgroups $K = \prod_v K_v$ of $G(\bA_{F^+}^\infty)$ such that for each finite place $v \not\in S$ of $F^+$, $K_v = \iota_v^{-1}(U_n(\cO_{F^+_v}))$. We define $\cJ_S^p$ to be the set of sufficiently small open compact subgroups $K^p = \prod_{v \nmid p} K_v$ of $G(\bA_{F^+}^{\infty, p})$ such that for each finite place $v\nmid p$ of $F^+$, $K_v = \iota_v^{-1}(U_n(\cO_{F^+_v}))$ (a slight variation of the notation $\cJ_S^{v_0}$ also established in the previous section). 

If $w$ is a place of $F$, then we define $\Iw_w \subset \GL_n(\cO_{F_w})$ to be the standard Iwahori subgroup (pre-image of $B_n(k(w))$ under reduction modulo $\varpi_w$) and for $d \geq c \geq 1$, $\Iw_w(c, d)$ to be the intersection of the pre-image of $B_n(\cO_{F_w} / \varpi_w^d)$ under reduction modulo $\varpi_w^d$ and the pre-image of $N_n(\cO_{F_w} / \varpi_w^c)$ under reduction modulo $\varpi_w^c$ (recall $B_n \supset N_n$ denote the standard Borel subgroup of $\GL_n$ and its unipotent radical, respectively). 

If $K^p \in \cJ_S^p$, then we have defined the Hecke algebra $\bT^S \subset \cH(G(\bA_{F^+}^S), K^S)$. Suppose that $v \in S - S_p$ and $K_v = \iota_\wv^{-1}(\Iw_{\wv})$. In this case we define an enlarged (and still commutative) algebra $\bT^S_v \subset \cH(G(\bA_{F^+}^{S - \{ v \}}), K^{S - \{ v \}}) \otimes_\bZ \cO$. Let $\Lambda_\wv = T_n(F_\wv) / T_n(\cO_{F_\wv})$. Bernstein's description of the algebra $\cH( \GL_n(F_\wv), \Iw_\wv) \otimes_\bZ \cO$ implies the existence of a unique $\cO$-algebra homomorphism 
\[ \cO[\Lambda_\wv] \to \cH( \GL_n(F_\wv), \Iw_\wv) \otimes_\bZ \cO \]
with the property that if $t = (\varpi_\wv^{m_1}, \varpi_\wv^{m_2}, \dots, \varpi_\wv^{m_n}) \in T_n(F_\wv)$ with $m_1 \geq m_2 \geq \dots \geq m_n$, then $t$ is mapped to the $[\Iw_\wv t \Iw_\wv]$. We define $\T^S_v$ to be the algebra generated by $\T^S$ and the image of $\cO[\Lambda_\wv]$. We define unramified characters $\chi_{\wv, i} : F_\wv^\times \to (\T^S_v)^\times$ ($i = 1, \dots, n$) by declaring that $\chi_{\wv, i}(\varpi_\wv)$ is the image of $(1, \dots, 1, \varpi_\wv, 1, \dots, 1) \in \Lambda_\wv$ (where $\varpi_\wv$ sits in the $i^\text{th}$ position). We can define a group determinant $D_\wv$ of $W_{F_\wv}$ of rank $n$ with coefficients in $\T^S_v$ as the one associated to the representation $\oplus_{i=1}^n \chi_{\wv, i} \circ \Art_{F_\wv}^{-1}$. Then \cite[Proposition 2.2.8]{All18} shows that $D_\wv$ ``is'' the semi-simplified Tate-normalised local Langlands correspondence $\rec_{F_\wv}^T$ for Iwahori-spherical representations of $\GL_n(F_\wv)$. We define $\Delta_\wv \in \T^S_v$ to be the discriminant of the characteristic polynomial $D_\wv(X - \Frob_\wv)$ of Frobenius. 

We need to define spaces of ordinary algebraic modular forms. We first define what we mean by ordinary parts. If $M$ is an $\cO$-module equipped with an endomorphism $U : M \to M$, then the ordinary part $M^{ord} \subset M$ is the $\cO$-submodule consisting of elements $f \in M$ such that there is a polynomial $P(X) \in \cO[X]$ such that $P(0) \in \cO^\times$ and $P(U) f = 0$.
\begin{lemma} Let $M$ be an $\cO[U]$-module.
\begin{enumerate}
    \item $M^{ord}$ is an $\cO[U]$-submodule of $M$ and the functor $M \mapsto M^{ord}$ is left exact and respects direct sums. 
    \item If $M$ is a finitely generated $\cO$-module then we can decompose $M = \oplus_\m M_\m$ as a direct sum of its non-zero localisations at maximal ideals of $\cO[U]$, and then $M^{ord} = \oplus_{\m : U \not\in \m} M_\m$. In particular, $M^{ord}$ is a direct summand of $M$.
    \item If $M$ is a finite dimensional $E$-vector space, then $M^{ord}$ is the largest $E$-vector subspace of $M$ where $U$ acts with eigenvalues that have non-positive valuation. 
    \item More generally, if $M$ is a direct limit of $\cO[U]$-modules which are finitely generated $\cO$-modules or finite dimensional $E$-vector spaces, then $M^{ord}$ is a direct summand of $M$ as $\cO[U]$-module. 
\end{enumerate}
\end{lemma}
\begin{proof}
The first part is elementary. For the second, let $A$ denote the $\cO$-subalgebra of $\End_\cO(M)$ generated by $U$. After replacing $M$ by one of its localisations, we can assume that $A$ is a local ring with maximal ideal $\m$. If $U \in \m$, we need to show that $M^{ord} = 0$. First, note that $M^{ord}$ is an $A$-submodule of $M$. If $f \in M^{ord}$ and $P(U) f = 0$, where $P(X) \in \cO[X]$ and $P(0) \in \cO^\times$, then $f \in \m M^{ord}$. Thus $M^{ord} = \m M^{ord}$ and Nakayama's lemma shows that $M^{ord} = 0$. If $U \not\in \m$, we need to show that $M^{ord} = M$. We can find a monic polynomial $P(X) \in \cO[X]$ such that $P(U) = 0$ in $A$. Let $\overline{P}(X) \in k[X]$ denote the reduction of $P(X)$ modulo $\varpi$ and factor $\overline{P}(X) = \overline{P}_1(X) \overline{P}_2(X)$, where $\overline{P}_1$ is a power of $X$ and $\overline{P}_2(0) \neq 0$. Hensel's lemma implies that this lifts uniquely to a factorisation $P(X) = P_1(X) P_2(X)$ in $\cO[X]$ where $P_1, P_2$ are also monic. Then $P_1(U) P_2(U) = 0$ in $A$. The image of $P_1(U)$ in $A / \m_A$ is a power of $U$, showing that $P_1(U) \in A^\times$ and therefore that $P_2(U) = 0$ in $A$. This implies that $M = M^{ord}$.

For the third part, suppose that $M$ is a finite-dimensional $E$-vector space, and let $P(X) = \det(X - U : M \to M)$. We can assume that all of the roots of $P(X)$ in $\overline{\bQ}_p$ have the same $p$-adic valuation $s$. If $s > 0$ (or if $s = \infty$), we need to show that $M^{ord} = 0$. If $f \in M^{ord}$ then $Q(U)f = 0$ for some $Q(X) \in \cO[x]$ with $Q(0) \in \cO^\times$. The polynomials $P(X)$, $Q(X)$ are coprime in $E[X]$, so we find that the relations $Q(U) f = P(U) f = 0$ imply $f = 0$, hence $M^{ord} = 0$. If $s \leq 0$, we need to show that $M^{ord} = M$. Since the Newton polygon of $P$ has the single slope $s$, the constant term $P(0)$ must have minimal valuation. In particular, $Q(X) = P(0)^{-1} P(X) \in \cO[X]$, $Q(U) = 0$, and $Q(0) = 1$, showing that $M = M^{ord}$ in this case.

The final part of the lemma is also elementary. 
\end{proof}
If $H \subset G(\bA_{F^+}^\infty)$ is a closed subgroup and $R = \cO$, $E$, or $\cO / (\varpi^c)$, then  we write $\cA(H, R)$ for the set of functions $f : G(\bA_{F^+}^\infty) \to R$, continuous with respect to the $\varpi$-adic topology on $R$, such that for each $\gamma \in G(F^+)$, $g \in G(\bA_{F^+}^\infty)$, and $h \in H$, we have $f(\gamma g h ) = f(g)$. We define compact subgroups
\[ \mathfrak{n}_p = \prod_{v \in S_p} \iota_\wv^{-1}(N_n(\cO_{F_\wv})) \subset G(F^+_p) \]
and
\[ \mathfrak{t}_p = \prod_{v \in S_p} \iota_\wv^{-1}( \ker(T_n(\cO_{F_\wv}) \to T_n(k(\wv)))) \subset G(F^+_p) \]
and an element
\[ \eta_p = (\iota_\wv^{-1}\diag(\varpi_\wv^{n-1}, \varpi_\wv^{n-2}, \dots, 1))_{v \in S_p} \in G(F^+_p). \]
Observe that $\mathfrak{t}_p$ normalises $\mathfrak{n}_p$. We define an endomorphism $U_p$ of $\cA(\mathfrak{n}_p, R)$ by the formula $U_p(f)(g) = \sum_{n \in \mathfrak{n}_p / \eta_p \mathfrak{n}_p \eta_p^{-1}} f(g n \eta_p)$. Then the action of $U_p$ on $\cA(\mathfrak{n}_p, R)$ commutes with the action of the group $G(\bA_{F^+}^{p, \infty}) \times \mathfrak{t}_p$. We define $\cA(\mathfrak{n}_p, R)^{ord}$ to be the ordinary part of $\cA(\mathfrak{n}_p, R)$ with respect to the action of $U_p$. Thus 
\[ \cA(\mathfrak{n}_p, R)^{ord} \subset \cA(\mathfrak{n}_p, R) \]
is an $R$-submodule which is invariant both under $U_p$ and the right translation action of $G(\bA_{F^+}^{p, \infty}) \times \mathfrak{t}_p$. 
If $\lambda : \mathfrak{t}_p \to \cO^\times$ is a continuous character and $K^p \in \cJ^p_S$, then we define $H_\lambda(K^p) = \cA(K^p \times \mathfrak{n}_p, \cO)^{ord, \mathfrak{t}_p = w_0 \lambda^{-1}}$ (where $w_0$ denotes the longest element in the Weyl group of $\prod_{v \in S_p} \GL_n$). It is a $\T^S[U_p]$-module. 

We now explain how $H_\lambda(K^p)$ may be compared with spaces of classical algebraic modular forms. If $\mu = (\mu_\tau)_\tau \in (\bbZ^n_+)^{\widetilde{I}_p}$, then we write $V_\mu$ for the $E[\prod_{v \in S_p} G(F^+_v)]$-module denoted by $W_\mu$ in \cite[Definition 2.3]{Ger19} and $\cV_\mu \subset V_\mu$ for the $\cO[\prod_{v \in S_p} G(\cO_{F^+_v})]$-submodule denoted by $M_\mu$ in \emph{loc. cit.}, which is an $\cO$-lattice. There is an embedding 
\[ (\bbZ^n)^{\widetilde{I}_p} \hookrightarrow \Hom_{cts}(\mathfrak{t}_p, \cO^\times), \]
which sends $\mu$ to the character $\mu : (\iota_\wv^{-1}(t_v))_{v \in S_p} \mapsto \prod_{\tau \in \widetilde{I}_p} \tau(\mu(t_{v(\tau)}))$ (here $v(\tau)$ denotes the place of $F$ induced by $\tau$). We say that a character $\lambda \in \Hom_{cts}(\mathfrak{t}_p, \cO^\times)$ is locally algebraic if it has the form $\lambda = \mu \psi$ for some $\mu \in (\bbZ^n)^{\widetilde{I}_p}$ and finite order character $\psi : \mathfrak{t}_p \to \cO^\times$, and locally dominant algebraic if further $\mu \in (\bZ^n_+)^{\widetilde{I_p}}$.

 If $K^p \in \cJ_S^p$ and $c \geq b \geq 1$, then we define
\[ \mathfrak{t}_p(c) = \mathfrak{t}_p = \prod_{v \in S_p} \iota_\wv^{-1}( \ker(T_n(\cO_{F_\wv}) \to T_n(\cO_{F_\wv} / (\varpi_\wv^c)))) \]
and
\[ K^p(b, c) = K^p \prod_{v \in S_p} \iota_\wv^{-1}( \Iw_\wv(b, c) ). \]
Choose an integer $c \geq 1$ such that $\psi|_{\mathfrak{t}_p(c, c)}$ is trivial. We can then view $\psi$ as a character of the group $K^p(1, c)$ (by projection to the factor at $p$). If $1 \leq b \leq c$ and $\lambda = \mu \psi$ is locally dominant algebraic then we define $M_\lambda(K^p(b, c), R)$ to be the set of functions
\[ f : G(\bA_{F^+}^\infty) \to \cV_\mu \otimes_\cO \cO(w_0 \psi) \otimes_\cO R \]
such that for any $\gamma \in G(F^+)$, $g \in G(\bA_{F^+}^\infty)$, and $u \in K^p(b, c)$, we have $u f ( \gamma g u ) = f(g)$. Then $M_\lambda(K^p(b, c), R)$ is a finite free $R$-module (cf. the discussion on \cite[p. 1351]{Ger19}). If $b \leq b' \leq c$ then $\mathfrak{t}_p(b) / \mathfrak{t}_p(b') \cong \prod_{v \in S_p} \Iw_\wv(b, c) / \Iw_\wv(b', c)$ acts on $M_\lambda(K^p(b', c), R)$ by the formula $(u \cdot f)(g) = uf(gu)$. By definition, the invariants of this action are $M_\lambda(K^p(b, c), R)$. It follows from \cite[Lemma 2.6]{Ger19} and the fact that $K^p$ is sufficiently small that $M_\lambda(K^p(b', c), R)$ is in fact a free $R[\mathfrak{t}_p(b) / \mathfrak{t}_p(b')]$-module and that there is an isomorphism
\begin{equation}\label{eqn_coinvariants_of_algebraic_modular_forms}
M_\lambda(K^p(b', c), R) \otimes_{R[\mathfrak{t}_p(b) / \mathfrak{t}_p(b')]} R \cong M_\lambda(K^p(b, c), R)
\end{equation}
of $\T^S$-modules, induced by the trace. 

We define the $U_p$-operator on $M_\lambda(K^p(b, c), R)$  as in \cite[Definition 2.8]{Ger19}, as a Hecke operator at level $K^p(b, c)$ (normalized by a scalar if $\mu \neq 0$). Then \cite[Lemma 2.10]{Ger19} shows that the natural inclusion $M_\lambda(K^p(b, c), R) \to M_\lambda(K^p(b, c'), R)$ is a homomorphism of $\bT^S[U_p]$-modules for any $c' \geq c \geq b$, while \cite[Lemma 2.19]{Ger19} shows that this inclusion induces an isomorphism on ordinary parts. We define $S_\lambda(K^p, R) = M_\lambda(K^p(1, c), R)^{ord}$, safe in the knowledge that this is independent of the choice of integer $c \geq 1$ such that $\psi|_{\mathfrak{t}_p \cap K^p(c, c)}$ is trivial. The description of $S_\lambda(K^p, \cO) \otimes_{\cO, \iota} \bC$ in terms of automorphic representations of $G(\bA_{F^+})$ is given by \cite[Lemma 2.5]{Ger19}. If $\lambda$ is the trivial character, we omit it from the notation. It follows from \cite[Proposition 2.22]{Ger19} that for any locally dominant algebraic character there is a $\bT^S[U_p]$-equivariant isomorphism
\[ S_\lambda(K^p, \cO) \otimes_\cO k \cong S(K^p, k). \]
The following proposition includes the comparison between $H_\lambda(K^p)$ and $S_\lambda(K^p, \cO)$.
\begin{proposition}\label{prop_Hida_theory} Let $K^p \in \cJ_S^p$.
\begin{enumerate}
    \item For any $\lambda \in \Hom_{cts}(\mathfrak{t}_p, \cO^\times)$, $H_\lambda(K^p)$ is a finite free $\cO$-module.
    \item For any $\lambda_1, \lambda_2 \in \Hom_{cts}(\mathfrak{t}_p, \cO^\times)$ such that $\lambda_1 \equiv \lambda_2 \text{ mod }(\varpi^c)$, there is a $\T^S$-equivariant isomorphism $H_{\lambda_1}(K^p) / (\varpi^c) \cong H_{\lambda_2}(K^p) / (\varpi^c)$.
    \item For any $\lambda \in  \Hom_{cts}(\mathfrak{t}_p, \cO^\times)$ which is locally dominant algebraic, there is a $\T^S[U_p]$-equivariant isomorphism $H_\lambda(K^p) \cong S_\lambda(K^p, \cO)$.
\end{enumerate}
 \end{proposition}
\begin{proof}
Let $\lambda \in \Hom_{cts}(\mathfrak{t}_p, \cO^\times)$. There is an isomorphism
\[ \cA(K^p \times \mathfrak{n}_p, \cO) = \varprojlim_c \cA(K^p \times \mathfrak{n}_p, \cO / (\varpi^c)), \]
and hence an embedding
\[ H_\lambda(\cO) \hookrightarrow \varprojlim_c \left( \cA(K^p \times \mathfrak{n}_p, \cO / (\varpi^c))^{ord, \mathfrak{t}_p = w_0\lambda^{-1}} \right). \]
We will first show that for each $c \geq 1$, $\cA(K^p \times \mathfrak{n}_p, \cO / (\varpi^c))^{ord, \mathfrak{t}_p = w_0\lambda^{-1}}$ is a finite free $\cO / (\varpi^c)$-module, and the natural map
\[ \cA(K^p \times \mathfrak{n}_p, \cO / (\varpi^{c+1}))^{ord, \mathfrak{t}_p = w_0\lambda^{-1}} \otimes_\cO \cO / (\varpi^c) \to \cA(K^p \times \mathfrak{n}_p, \cO / (\varpi^c))^{ord, \mathfrak{t}_p = w_0\lambda^{-1}} \]
is an isomorphism. This will imply that $H_\lambda(K^p)$ is a finite free $\cO$-module and that for each $c \geq 1$, there is an isomorphism
\[ H_\lambda(K^p) / (\varpi^c) \cong \cA(K^p \times \mathfrak{n}_p, \cO / (\varpi^c))^{ord, \mathfrak{t}_p = w_0\lambda^{-1}}. \]
Since the right-hand side here depends only on the reduction of $\lambda$ mod $(\varpi^c)$, this will imply the truth of the first two parts of the proposition.

If $d \geq 1$ is such that $\lambda \text{ mod }(\varpi^c)$ is trivial on $\mathfrak{t}_p(d)$, then $\lambda \text{ mod }(\varpi^c)$ extends to a character of $K^p(1, d)$. We can compute
\[ \cA(K^p \times \mathfrak{n}_p, \cO / (\varpi^c))^{ord, \mathfrak{t}_p = w_0\lambda^{-1}} = \varinjlim_d M(K^p(d, d), \cO / (\varpi^c)(\lambda^{-1}))^{ord, \mathfrak{t}_p = w_0 \lambda^{-1}}, \]
where the maps
\[ M(K^p(d, d), \cO / (\varpi^c)) \to M(K^p(d+1, d+1), \cO / (\varpi^c)) \]
are the natural inclusions (which respect the action of the $U_p$ operator, by \cite[Lemma 2.10]{Ger19}). Then \cite[Lemma 2.19]{Ger19} shows that these inclusions become isomorphisms after passage to ordinary part, and so we have simply
\[ \cA(K^p \times \mathfrak{n}_p, \cO / (\varpi^c))^{ord, \mathfrak{t}_p = w_0\lambda^{-1}} = M( K^p(d, d), \cO / (\varpi^c))^{ord, \mathfrak{t}_p = w_0 \lambda^{-1}} \]
for any $d \geq 1$ such that $\lambda \text{ mod }(\varpi^c)$ is trivial on $\mathfrak{t}_p(d)$. To complete the proof of the first two parts of the proposition, it is enough to observe that if $d \geq 1$ is such that $\lambda \text{ mod }(\varpi^{c+1})$ is trivial on $\mathfrak{t}_p(d)$, then $M(K^p(d, d), \cO / (\varpi^{c+1}))^{ord, \mathfrak{t}_p = w_0 \lambda^{-1}}$ is a finite free $\cO / (\varpi^{c+1})$-module and the natural map
\[  M( K^p(d, d), \cO / (\varpi^{c+1}))^{ord, \mathfrak{t}_p = w_0 \lambda^{-1}} \otimes_\cO \cO / (\varpi^c) \to M( K^p(d, d), \cO / (\varpi^c))^{ord, \mathfrak{t}_p = w_0 \lambda^{-1}} \]
is an isomorphism (both claims being true before passage to ordinary direct summands and as a consequence of $K^p$ being sufficiently small, cf. \cite[Lemma 2.6]{Ger19}).

To prove the third part of the lemma, it is enough to construct a compatible sequence of isomorphisms
\[ H_\lambda(K^p) / (\varpi^c) \cong S_\lambda(K^p, \cO) / (\varpi^c), \]
or equivalently a compatible sequence of isomorphisms
\[ M( K^p(d, d), \cO / (\varpi^c))^{ord, \mathfrak{t}_p = w_0 \lambda^{-1}} \cong S_\lambda(K^p, \cO / (\varpi^c)). \]
The existence of such a sequence is the content of \cite[Proposition 2.22]{Ger19}.
\end{proof}
To formulate a corollary, we introduce some notation. Let $\lambda \in \Hom_{cts}(\mathfrak{t}_p, \cO^\times)$ and let $\m \subset \T^S$ be a maximal ideal  which is in the support of $S(K^p, k)$. Fix a place $v_0 \in S - S_p$ of $F^+$ which splits $v_0 = w_0 w_0^c$ in $F$, and define 
\[ \cH_\lambda = \varinjlim_{V_{v_0}} H_\lambda(K^{p, v_0} V_{v_0})_\m, \]
\[ \cS = \varinjlim_{V_{v_0}} S(K^{p, v_0} V_{v_0}, k)_\m, \]
the limit in each case running over the directed system of open compact subgroups $V_{v_0} \subset G(F^+_{v_0})$. Then:
\begin{corollary}\label{cor_vanishing_of_derived_tensor_product}
\begin{enumerate}
    \item $\cH_\lambda$ is an admissible $\cO[\GL_n(F_{w_0})]$-module, which is $\cO$-flat and $\varpi$-adically separated.
    \item There is a $\T^S$-equivariant isomorphism $\cH \otimes_\cO k \cong \cS$ of admissible $k[\GL_n(F_{w_0})]$-modules. 
    \item Suppose that $\cS \otimes^{\bL}_{k[\GL_n(F_{w_0})^0]} R \Gamma_c( \Omega_{F_{w_0}}^{n-1, ca}, k) = 0$. Then we have
    \[ \cH_\lambda \otimes^{\bL}_{\cO[\GL_n(F_{w_0})^0]} R \Gamma_c( \Omega_{F_{w_0}}^{n-1, ca}, E) = 0, \]
    and consequently $\cH_\lambda[1/p] \otimes^{\bL}_{E[\GL_n(F_{w_0})^0]} \pi_{i+1,1,1,\dots,1} = 0$ for each $i = 0, \dots, n-1$. 
\end{enumerate}
\end{corollary}
We recall that the representations $\pi_\nu$ associated to an ordered partition $\nu$ of $n$ have been defined in \S \ref{sec_elliptic_representations}.
\begin{proof}
The $\cO[\GL_n(F_{w_0})]$-module $\cH_\lambda$ is a saturated submodule of $\cA(K^{p, v_0} \times \mathfrak{n}_p, \cO)$ (i.e. the cokernel of the inclusion map is $\varpi$-torsion-free). In particular, it is $\varpi$-adically separated, and if $V_{v_0} \subset G(F^+_{v_0})$, then $\cH_\lambda^{V_{v_0}} = H_\lambda(K^{p, v_0} V_{v_0})_\m$. Since $\cH_\lambda$ is smooth by definition, and $H_\lambda(K^{p, v_0} V_{v_0})_\m$ is a finite free $\cO$-module, this shows that $\cH_\lambda$ is even admissible. For any open compact subgroup $V_{v_0}$, there is an isomorphism $H_\lambda(K^{p, v_0} V_{v_0})_\m \otimes_\cO k \cong S(K^{p, v_0} V_{v_0}, k)_\m$, and the second part of the corollary follows by passage to the limit.

For the last part, let $C^\bullet = \cH_\lambda \otimes^{\bL}_{\cO[\GL_n(F_{w_0})^0]} R \Gamma_c( \Omega_{F_{w_0}}^{n-1, ca}, \cO) \in \mathbf{D}^-(\cO)$. Then each group $H^i(C^\bullet)$ is a finitely generated $\cO$-module, by Theorem \ref{thm_tor_is_fg}. Invoking Proposition \ref{prop_homology_change_of_coefficients} and our assumption, we find that
\[ C^\bullet \otimes^{\bL}_\cO k = \cS \otimes^{\bL}_{k[\GL_n(F_{w_0})^0]} R \Gamma_c( \Omega^{n-1, ca}_{F_{w_0}}, k) = 0. \]
Applying a version of Nakayama's lemma for complexes (e.g.\ \cite[\href{https://stacks.math.columbia.edu/tag/0G1U}{Lemma 0G1U}]{stacks-project}) we find that $C^\bullet = 0$, and therefore that 
\[ C^\bullet \otimes^{\bL}_\cO E = \cH_\lambda \otimes^{\bL}_{\cO[\GL_n(F_{w_0})^0]} R \Gamma_c( \Omega_{F_{w_0}}^{n-1, ca}, E) = 0. \]
The final claim in the corollary follows from Theorem \ref{thm_decomposition_of_R_Gamma_Omega} (namely, the decomposition of $R \Gamma_c( \Omega_{F_{w_0}}^{n-1, ca}, E)$).
\end{proof}
In the next proposition, we introduce some new notation: if $M$ is a $\T^S$-module, then $\T^S(M)$ denotes the image of the map $\T^S \to \End_\cO(M)$ (and similarly for $\T^S_{v_0}$).
\begin{proposition}\label{prop_local_hecke_algebra}
Let $\lambda \in \Hom_{cts}(\mathfrak{t}_p, \cO^\times)$ and $K^p \in \cJ^p_S$. Then:
\begin{enumerate} 
\item There exists an $n$-dimensional group determinant $D$ of $G_{F, S}$ with coefficients in $\T^S(H_\lambda(K^p))$ such that for any place $w \nmid S$ of $F$ which splits over $F^+$, $D(X - \Frob_w)$ equals the image of $P_w(X)$ in $\T^S(H_\lambda(K^p))$. 
\item Let $v_0 \in S - S_p$ and suppose that $K_{v_0} = \iota_{\wv_0}^{-1}(\Iw_{\wv_0})$. Then the pushforward of $D|_{W_{F_{\wv_0}}}$ under the inclusion
\[ \T^S(H_\lambda(K^p)) \to \T^S_{v_0}(H_\lambda(K^p)) \]
equals the pushforward of $D_{\wv_0}$ under the map
\[ \T^S_{v_0} \to \T^S_{v_0}(H_\lambda(K^p)). \]
\item Let $T : G_{F, S} \to \T^S(H_\lambda(K^p))$ be the trace associated to the group determinant $D$. Then for all $\sigma \in G_{F, S}$ and for all $\tau_1, \dots, \tau_n \in W_{F_{\wv_0}}$, we have the relation
\[ \Delta_{\wv_0}^{n!} T(\sigma(\tau_1 - \chi_{\wv_0, 1}(\tau_1)) \dots (\tau_n - \chi_{\wv_0, n}(\tau_n)) ) = 0 \]
in $\T^S_{v_0}(H_\lambda(K^p))$.
\end{enumerate}
\end{proposition}
\begin{proof}
The first part is a standard consequence of Corollary \ref{cor_Gal_reps_for_G}. Indeed, we have
\[ \T^S(H_\lambda(K^p)) = \varprojlim_c \T^S(H_\lambda(K^p) / (\varpi^c)), \]
while for any $c \geq 1$ and any $d \geq c$ such that $\lambda \text{ mod }(\varpi^c)$ is trivial on $\mathfrak{t}_p \cap K^p(d, d)$, there is a surjective homomorphism of $\T^S$-algebras
\[ \T^S(S(K^p(d, d), \cO)) \to \T^S( H_\lambda(K^p) / (\varpi^c)) \]
which is induced by the inclusion $H_\lambda(K^p) / (\varpi^c) \cong S(K^p(d, d), \cO / (\varpi^c))^{\mathfrak{t}_p = w_0 \lambda^{-1}} \subset S(K^p(d, d), \cO) / (\varpi^c)$. We can decompose
\[ \T^S(S(K^p(d, d), \cO)) \otimes_\cO \overline{\bQ}_p = \oplus \overline{\bQ}_p, \]
the direct summands being indexed by the finitely many systems of $\T^S$-eigenvalues that appear in $S(K^p(d, d), \cO) \otimes_\cO \overline{\bQ}_p$. The group determinant $D$ exists as a consequence of existence of the family of $\overline{\bQ}_p$-valued determinants indexed by these summands and glueing for group determinants (\cite[Example 2.32]{Che14}). 

The second part of the lemma may be proved similarly. It is enough to show that for any $d \geq 1$, the group determinant $D_d$ of $G_{F, S}$ valued in $\T^S_{v_0}(S(K^p(d, d), \cO))$ has the property that $D_d|_{W_{F_{\wv_0}}}$ equals the pushforward of $D_{\wv_0}$. This can be checked after extending scalars to $\overline{\bQ}_p$, in which case it follows from \cite[Proposition 2.2.8]{All18}. 

The third part of the lemma is slightly more subtle, since we need to take into account Hecke operators which do not act as scalars on automorphic representations (in other words, which do not lie in the centre of the Iwahori--Hecke algebra at the place $v_0$). What we need to check is that if $d \geq 1$ and $T_d : G_{F, S} \to \T^S_{v_0}(S(K^p(d, d), \cO))$ is the trace associated to the group determinant $D_d$, then for all $\sigma \in G_{F, S}$ and for all $\tau_1, \dots, \tau_n \in W_{F_{\wv_0}}$, we have the relation
\[  \Delta_{\wv_0}^{n!} T_d(\sigma(\tau_1 - \chi_{\wv_0, 1}(\tau_1)) \dots (\tau_n - \chi_{\wv_0, n}(\tau_n)) ) = 0 \]
in $\T^S_{v_0}(S(K^p(d, d), \cO))$. This can be checked in $\T^S_{v_0}(S(K^p(d, d), \cO)) \otimes_\cO \overline{\bQ}_p$, which now splits as a sum of its localisations at the finitely many systems of $\T^S_{v_0}$-eigenvalues which appear in $S(K^p(d, d), \cO) \otimes_\cO \overline{\bQ}_p$.

Let $\q \subset \bT^S_{v_0}(S(K^p(d, d), \cO))$ be a minimal prime ideal, kernel of a homomorphism $\T^S_{v_0}(S(K^p(d, d), \cO)) \to \overline{\bQ}_p$. If $\Delta_{\wv_0} \in \q$, then $\Delta_{\wv_0}^{n!} = 0$ in $\bT^S_{v_0}(S(K^p(d, d), \cO))_{(\q)}$, because this Artinian local ring has length at most $n!$ (cf. \cite[Proposition 3.11]{Tho22}). The required relation thus certainly holds in this localisation. If $\Delta_{\wv_0} \not\in \q$, then $\bT^S_{v_0}(S(K^p(d, d), \cO))_{(\q)}$ is a field, the map $\T^S_{v_0}(S(K^p(d, d), \cO))_{(\q)} \to \overline{\bQ}_p$ is an isomorphism, and we need to check that if $\sigma$ is an automorphic representation of $G(\bA_{F^+})$ which contributes to $S(K^p(d, d), \cO)$ and $\rho = r_\iota(\sigma)$, then the relation
\[ (\rho(\tau_1) - (\chi_{\wv_0, 1} \text{ mod } \q)(\tau_1)) (\rho(\tau_2) - (\chi_{\wv_0, 2} \text{ mod }\q)(\tau_2)) \dots (\rho(\tau_n) - (\chi_{\wv_0, n} \text{ mod }\q)(\tau_n)) = 0 \]
holds in $M_n(\overline{\bQ}_p)$. This is a consequence of local-global compatibility at the place $v_0$ and \cite[Lemma 2.18]{New21}. 
\end{proof}
We next record some consequences of the results of \S \ref{sec_definite_unitary_group}.
\begin{proposition}\label{prop_ordinary_endoscopic_descent}
Fix a partition $n = n_1 + n_2$ and let $\pi_1, \pi_2$ be cuspidal, conjugate self-dual automorphic representations of $\GL_{n_1}(\bA_F)$, $\GL_{n_2}(\bA_F)$, respectively, such that $\pi = \pi_1 \boxplus \pi_2$ is regular algebraic of weight $\mu$ and $\iota$-ordinary. Define representations $\rho_i = r_\iota(\pi_i| \cdot |^{(n_i-n)/2})$, so that $r_\iota(\pi) = \rho_1 \oplus \rho_2$. Then:
\begin{enumerate}
     \item For each place $v \in S_p$, there are uniquely determined characters $\alpha_{v, 1}, \dots, \alpha_{v, n} : F_\wv^\times \to \bC^\times$ with the following properties:
\begin{enumerate}
    \item If $\val : \overline{\bQ}_p^\times \to \bQ$ denotes the $p$-adic valuation, then $\val( \iota^{-1} \alpha_{v, 1}(\varpi_\wv) ) < \dots < \val( \iota^{-1} \alpha_{v, n}(\varpi_\wv) )$.
    \item $\pi_\wv$ is an irreducible subquotient of the normalised induction $i_{B_n(F_\wv)}^{\GL_n(F_\wv)} \alpha_{v, 1} \otimes \dots \otimes \alpha_{v, n}$.
\end{enumerate}
    \item\label{desc_conditions} Suppose that condition (\ref{ass_even}) of Theorem \ref{thm_endoscopic_forms_exist} is satisfied. Assume that each character $\iota^{-1}\alpha_{v, i}$ takes values in $E^\times$, and let $\alpha_\pi : \mathfrak{t}_p \to \cO^\times$ be the restriction to $\mathfrak{t}_p$ of the character 
    \[ \iota^{-1} (\otimes_{v \in S_p} (\alpha_{v, 1} \otimes \dots \otimes \alpha_{v, n}) ) : \prod_{v \in S_p} T_n(F_\wv) \to E^\times. \]
    Let $\lambda_\pi = w_0(\alpha_\pi)^{-1} \cdot (\mu_{\iota \tau})_{\tau \in \widetilde{I}_p} \in \Hom_{cts}(\mathfrak{t}_p, \cO^\times)$. Assume that each finite place $w$ of $F$ such that $\pi_w$ is ramified lies over $S$. Let $K^p \in \cJ_S^p$ be such that for each $v \in S - S_p$, $\pi_\wv^{\iota_\wv(K_v)} \neq 0$. Then there exists a unique $\cO$-algebra homomorphism
    \[ f_\pi : \T^S(S_{\lambda_\pi}(K^p, \cO)) \to \overline{\bQ}_p \]
    such that  for each place $w \nmid S$ of $F$ which is split over $F^+$, the image of $P_w(X) \in \T^S[X]$ equals 
    \[ \det(X - r_\iota(\pi)(\Frob_w)) = \det(X - \rho_1(\Frob_w)) \det(X - \rho_2(\Frob_w)). \]
\end{enumerate}
\end{proposition}
\begin{proof}
The first part follows from \cite[Lemma 2.5]{Clo14}. The second follows from Theorem \ref{thm_endoscopic_forms_exist}. 
\end{proof}
\begin{corollary}\label{cor_twisted_endoscopic_descent}
With assumptions as in Proposition \ref{prop_ordinary_endoscopic_descent}(2), suppose given a continuous character $\psi : G_{F, S} \to \cO^\times$ of finite order such that $\psi \psi^c = 1$ and for each $v \in S - S_p$, $\psi|_{G_{F_\wv}} \circ \Art_{F_\wv} \circ \det$ is trivial on $\iota_\wv(K_v)$. Let $\pi_\psi = (\pi_1 \otimes \iota \psi) \boxplus \pi_2$. Then 
there exists a unique $\cO$-algebra homomorphism
    \[ f_{\pi_\psi} : \T^S(S_{\lambda_{\pi_\psi}}(K^p, \cO)) \to \overline{\bQ}_p \]
    such that  for each place $w \nmid S$ of $F$ which is split over $F^+$, the image of $P_w(X) \in \T^S[X]$ equals  
    \[ \det(X - (\rho_1 \otimes \psi)(\Frob_w)) \det(X - \rho_2(\Frob_w)). \]
\end{corollary}
\begin{proof}
We need only observe that if $\pi$ satisfies the hypotheses of Proposition \ref{prop_ordinary_endoscopic_descent}(2), then so does $\pi_\psi$.
\end{proof}
We can describe $\lambda_{\pi_\psi}$ in terms of $\lambda_\pi$. Indeed, $\pi_\psi$ and $\pi$ have the same weight $\mu$. In the notation of Proposition \ref{prop_ordinary_endoscopic_descent}(1), we can define for each $v \in S_p$ subsets $I_v, J_v \subset \{ 1, \dots, n \}$ such that $\alpha_{v, i}$ is a subquotient of $\rec_{F_\wv}(\pi_{1, \wv})$ if $i \in I_v$ and $\alpha_{v, i}$ is a subquotient of $\rec_{F_\wv}(\pi_{2, \wv})$ if $i \in J_v$. Let $\beta_\pi : \mathfrak{t}_p \to \prod_{v \in S_p} (1 + \varpi_\wv \cO_{F_\wv})$ be defined by $t_v = (t_{v, 1}, \dots, t_{v, n}) \in T_n(\cO_{F_\wv}) \mapsto \prod_{i \in I_v} t_{v, i}$. Then we have
\begin{equation}\label{eqn_weight_of_twisted_representation} \lambda_{\pi_\psi} = \lambda_\pi \cdot w_0( ( \psi|_{G_{F_\wv}} \circ \Art_{F_\wv}^{-1} \circ \beta_\pi )_{v \in S_p} )^{-1}.
\end{equation}
We need a version of Corollary \ref{cor_twisted_endoscopic_descent} ``in families''. With assumptions as in Proposition \ref{prop_ordinary_endoscopic_descent}(2), fix an integer $d \geq 1$ such that $\alpha_\pi|_{\mathfrak{t}_p(d)}$ is trivial, and let $F_d / F$ be the largest abelian extension of $F$ with the following properties:
\begin{itemize}
    \item $F_d$ is unramified outside $p$.
    \item For each place $v \in S_p$, the conductor of $F_{d, \wv} / F_\wv$ divides $(\varpi_\wv^d)$.
    \item The extension $F_d / F^+$ is Galois and $c \in \Gal(F / F^+)$ acts on $F_d$ by $-1$.
\end{itemize}
Then $\Delta_{d} = \Gal(F_d / F)$ is finite and any character $\psi : \Delta_{d} \to \cO^\times$ satisfies the conditions of Corollary \ref{cor_twisted_endoscopic_descent}. The map $\beta_\pi$, together with the Artin map, defines a homomorphism $\mathfrak{t}_p \to \Delta_d$. 

Let $\Psi : \Delta_{d} \to \cO[\Delta_{d}]^\times$ be the tautological character. Let $\T^S_{\mathfrak{t}_p} = \T^S \otimes_\cO \cO \llbracket \mathfrak{t}_p \rrbracket$. Assume (as we may, after possibly enlarging $E$) that $\rho_1, \rho_2$ take values in $\GL_{n_1}(\cO)$, $\GL_{n_2}(\cO)$. We claim that there is an $\cO \llbracket \mathfrak{t}_p \rrbracket$-algebra homomorphism
\[ \T^S_{\mathfrak{t}_p}(  M_{\lambda_\pi}(K^p(d, d), \cO)^{ord}) \to \cO[\Delta_{d}] \]
that, for each place $w \nmid S$ of $F$ which is split over $F^+$, sends $P_w(X)$ to $\det(X - ((\rho_1 \otimes \Psi) \oplus \rho_2)(\Frob_w))$. To see this, we consider the natural embedding 
\[ \cO[\Delta_{d}] \hookrightarrow \cO[\Delta_{d}] \otimes_\cO \overline{\bQ}_p \cong \prod_{\psi : \Delta_{d} \to \overline{\bQ}_p^\times} \overline{\bQ}_p. \]
Taking the product over the maps $f_{\pi_\psi}$ coming from Corollary \ref{cor_twisted_endoscopic_descent} determines a map $\T^S_{\mathfrak{t}_p}(  M_{\lambda_\pi}(K^p(d, d), \cO)^{ord}) \to \cO[\Delta_{d}] \otimes_\cO \overline{\bQ}_p$ which in fact takes values in $\cO[\Delta_{d}]$ (because the image of $P_w(X)$ for each place $w$ as above lies in $\cO[\Delta_{ d}][X]$).

The next proposition shows what happens when we allow the character $\psi$ to have infinite order (the proof is by reduction to the finite order case):
\begin{proposition}\label{prop_non-classical_twisted_endoscopic_descent}
With assumptions as in Proposition \ref{prop_ordinary_endoscopic_descent}(2), suppose given  a continuous character $\psi : G_{F, S_p} \to \cO^\times$ such that $\psi \psi^c = 1$. Let $\lambda_{\pi_\psi} \in \Hom_{cts}(\mathfrak{t}_p, \cO^\times)$ be the character defind by the formula (\ref{eqn_weight_of_twisted_representation}). Then, after possibly enlarging $\cO$, there exists a homomorphism $\T^S(H_{\lambda_{\pi_\psi}}(K^p)) \to \cO$ of $\cO$-algebras such that for each finite place $w \nmid S$ of $F$ which is split over $F^+$, the image of $P_w(X)$ in $\cO[X]$ equals $\det(X - ((\rho_1 \otimes \psi) \oplus \rho_2)(\Frob_w))$.
\end{proposition}
\begin{proof}
Let $\gamma = w_0( ( \psi|_{G_{F_\wv}} \circ \Art_{F_\wv}^{-1} \circ \beta_\pi )_{v \in S_p} )^{-1} \in \Hom_{cts}(\mathfrak{t}_p, \cO^\times)$, so that $\lambda_{\pi_\psi} = \lambda_\pi \cdot \gamma$. Let $m = \dim_k M_{\lambda_\pi}(K^p(1, 1), k)^{ord}$. Let $I_\gamma$ denote the kernel of the homomorphism $\cO\llbracket \mathfrak{t}_p \rrbracket \to \cO$ determined by $w_0(\gamma)^{-1}$. 

Fix $c \geq 1$, and let $d \geq c$ be an integer such that $\gamma \text{ mod }(\varpi^c)$ is trivial on $\mathfrak{t}_p(d)$. By Proposition \ref{prop_Hida_theory}, there is a $\T^S[U_p]$-equivariant isomorphism
\[ H_{\lambda_{\pi_\psi}}(K^p) / (\varpi^c) \cong M_{\lambda_{\pi_\psi}}(K^p(d, d), \cO / (\varpi^c))^{ord, \mathfrak{t}_p = w_0(\gamma)^{-1}}. \]
Moreover, $M_{\lambda_\pi}(K^p(d, d), \cO / (\varpi^c))^{ord}$ is a free $\cO / (\varpi^c)[\mathfrak{t}_p / \mathfrak{t}_p(d)]$-module and the trace induces a $\bT^S[U_p]$-equivariant isomorphism
\[ M_{\lambda_\pi}(K^p(d, d), \cO / (\varpi^c))^{ord} / (I_\gamma) \cong M_{\lambda_\pi}(K^p(d, d), \cO / (\varpi^c))^{ord, \mathfrak{t}_p = w_0(\gamma)^{-1}} \]
(cf. Equation (\ref{eqn_coinvariants_of_algebraic_modular_forms})). There is also an isomorphism
\[  M_{\lambda_\pi}(K^p(d, d), \cO / (\varpi^c))^{ord, \mathfrak{t}_p = w_0(\gamma)^{-1}} / (\varpi) \cong M_{\lambda_\pi}(K^p(1, d), k)^{ord}. \]
Let $R = \T^S_{\mathfrak{t}_p}(M(K^p(d, d), \cO)^{ord})$. Then $R$ acts faithfully on $M(K^p(d, d), \cO)^{ord}$, hence $\Fitt_R(M(K^p(d, d), \cO)^{ord}) = 0$, hence
\[ \Fitt_R(M_{\lambda_\pi}(K^p(d, d), \cO / (\varpi^c))^{ord, \mathfrak{t}_p = w_0(\gamma)^{-1}}) = (I_\gamma, \varpi^c). \] Since $M_{\lambda_\pi}(K^p(d, d), \cO / (\varpi^c))^{ord, \mathfrak{t}_p = w_0(\gamma)^{-1}}$ is generated as an $\cO / (\varpi^c)$-module, hence as $R$-module, by $m$ elements, it follows that 
\[ \Ann_{R / (I_\gamma, \varpi^c)}( M_{\lambda_\pi}(K^p(d, d), \cO / (\varpi^c))^{ord, \mathfrak{t}_p = w_0(\gamma)^{-1}})^m = 0. \]
The discussion preceding the statement of the Proposition shows that there is an $\cO[\mathfrak{t}_p]$-algebra homomorphism $R \to \cO[\Delta_d]$, hence (by composition with $\psi$) a homomorphism $R / (I_\gamma, \varpi^c) \to \cO / (\varpi^c)$, hence a homomorphism
\[ \T^S(M_{\lambda_\pi}(K^p(d, d), \cO / (\varpi^c))^{ord, \mathfrak{t}_p = w_0(\gamma)^{-1}}) \to \cO / (\varpi^{\lfloor c / m \rfloor}). \]
which for each split place $w \nmid S$ of $F$, sends $P_w(X)$ to the image of $\det(X - ((\rho_1 \otimes \psi) \oplus \rho_2)(\Frob_w))$ in $\cO / (\varpi^{\lfloor c / m \rfloor})[X]$. Using the identification of $H_{\lambda_{\pi_\psi}}(K^p) / (\varpi^c)$, we see that there is a homomorphism
\[ \T^S(H_{\lambda_{\pi_\psi}}(K^p)) \to \cO / (\varpi^{\lfloor c / m \rfloor}) \]
with the same property with respect to the polynomials $P_w(X)$. Since $m$ was fixed and $c$ was arbitrary, we can pass to the limit to obtain the desired homomorphism $\T^S(H_{\lambda_{\pi_\psi}}(K^p)) \to \cO$.
\end{proof}

\begin{corollary}\label{cor_non_classical_twisted_endoscopic_descent}
Fix a partition $n = n_1 + n_2$ and let $\pi_1, \pi_2$ be cuspidal, conjugate self-dual automorphic representations of $\GL_{n_1}(\bA_F)$, $\GL_{n_2}(\bA_F)$, respectively, such that $\pi = \pi_1 \boxplus \pi_2$ is regular algebraic of weight $\mu$ and $\iota$-ordinary.  Let $K^p \in \cJ_S^p$ be such that for each $v \in S - S_p$, $\pi_\wv^{\iota_\wv(K_v)} \neq 0$. Suppose that condition (\ref{ass_even}) of Theorem \ref{thm_endoscopic_forms_exist} is satisfied.

Suppose given  a continuous character $\psi : G_{F, S_p} \to \cO^\times$ such that $\psi \psi^c = 1$ and a place $v_0 \in S - S_p$ such that $\pi_{1, \wv_0} \cong \St_{n_1}(\psi_1)$, $\pi_{2, \wv_0} \cong \St_{n_2}(\psi_2)$, where $\psi_1, \psi_2 : F_{\wv_0}^\times \to \bC^\times$ are unramified characters such that $| \cdot |^{(1-n_1)/2} \psi_1 (\iota \psi \circ \Art_{F_{\wv_0}}) = | \cdot |^{(1+n_2)/2} \psi_2$. Suppose that $K_{v_0} = \iota_{\wv_0}^{-1}(\Iw_{\wv_0})$, and let $\p$ denote the kernel of the homomorphism $\T^S(H_{\lambda_{\pi_\psi}}(K^p)) \to \cO$ associated to $\pi_1, \pi_2, \psi$ by Proposition \ref{prop_non-classical_twisted_endoscopic_descent}. Let $\tau$ be an irreducible $E[\GL_n(F_{\wv_0})]$-subquotient of 
\[ \cH_{{\lambda_{\pi_\psi}}, \p} = \varinjlim_{V_{v_0}} H_{\lambda_{\pi_\psi}}(K^{p, v_0} V_{v_0})_{\p}. \]
If $\tau^{\Iw_{\wv_0}} \neq 0$, then either $\iota \tau \cong \pi_n(\psi_2 |\cdot|^{n_1/2})$ or $\iota \tau \cong \pi_{n_1, n_2}(\psi_2 |\cdot|^{n_1/2})$. 
\end{corollary}
\begin{proof}
If $\tau$ is an irreducible subquotient of $\cH_{{\lambda_{\pi_\psi}}, \p}$, then $\tau^{\Iw_{\wv_0}}$ is a subquotient $\cO[\Lambda_{\wv_0}]$-module of $\cH_{{\lambda_{\pi_\psi}}, \p}^{\Iw_{\wv_0}} = H_{{\lambda_{\pi_\psi}}}(K^p)_\p$. For any prime ideal $\mathfrak{r} \subset \cO[\Lambda_{\wv_0}]$ of characteristic 0 which is in the support of $\tau^{\Iw_{\wv_0}}$ we can therefore find a prime ideal $\q \subset \T^S_{v_0}(H_{{\lambda_{\pi_\psi}}}(K^p))$ lying above both $\p$ and $\mathfrak{r}$. By Proposition \ref{prop_local_hecke_algebra}, the restriction of $\mathfrak{r}$ to $\cO[\Lambda_{\wv_0}]^{S_n}$ corresponds to the polynomial
\[ \det(X - ((\rho_1 \otimes \psi) \oplus \rho_2)(\Frob_{\wv_0})) = \prod_{i=1}^n(X - (\iota^{-1} \psi_2 |\cdot|^{(n_1 + 2 (i-1) )/2})(\varpi_{\wv_0})), \]
which has distinct roots; therefore $\Delta_{\wv_0} \not\in \mathfrak{q}$. Let $\chi_{\wv_0, i, \mathfrak{r}} = \chi_{\wv_0, i} \text{ mod }\mathfrak{r}$ and let $\rho = (\rho_1 \otimes \psi) \oplus \rho_2$. Proposition \ref{prop_local_hecke_algebra} implies that for any $\sigma \in G_F, \tau_1, \dots, \tau_n \in W_{F_{\wv_0}}$ we have
\[ \tr \left(  \rho(\sigma)(\rho(\tau_1) - \chi_{\wv_0, 1, \mathfrak{r}}(\tau_1))(\rho(\tau_2) - \chi_{\wv_0, 2, \mathfrak{r}}(\tau_2)) \dots (\rho(\tau_n) - \chi_{\wv_0, n, \mathfrak{r}}(\tau_n))  \right) = 0. \]
Since $\rho$ is semisimple, it follows that for any $ \tau_1, \dots, \tau_n \in W_{F_{\wv_0}}$, we have
\[ (\rho(\tau_1) - \chi_{\wv_0, 1, \mathfrak{r}}(\tau_1))(\rho(\tau_2) - \chi_{\wv_0, 2, \mathfrak{r}}(\tau_2)) \dots (\rho(\tau_n) - \chi_{\wv_0, n, \mathfrak{r}}(\tau_n)) = 0 \]
in $M_n(\overline{\bQ}_p)$. Using once more that the characters $\{ \chi_{\wv_0, i, \mathfrak{r}} \}_{i = 1, \dots, n} = \{ (\iota^{-1}\psi_2 |\cdot|^{(n_1 + 2 (i-1) )/2}) \circ \Art_{F_{\wv_0}} \}_{i = 1, \dots, n}$ are pairwise distinct, we can apply \cite[Lemma 6.2.11]{All18} to conclude that there is a $W_{F_{\wv_0}}$-invariant filtration $0 \subset \Fil_1 \subset \dots \subset \Fil_n = \overline{\bQ}_p^n$ of $\overline{\bQ}_p^n$ such that for each $i = 1, \dots, n$, $\Fil_i / \Fil_{i-1} \cong \overline{\bQ}_p( \chi_{\wv_0, i, \mathfrak{r}} )$. This contradicts Lemma \ref{lem_invariant_filtrations} below unless $\tau  \cong \pi_n(\psi_2 |\cdot|^{n_1/2})$ or $\tau \cong \pi_{n_1, n_2}(\psi_2 |\cdot|^{n_1/2})$. 
\end{proof}
\begin{lemma}\label{lem_invariant_filtrations}
Let $v_0 \in S - S_p$ and let $\rho : G_{F_{\wv_0}} \to \GL_n(\overline{\bQ}_p)$ be a continuous representation such that  $\mathrm{WD}(\rho)^{F-ss} \cong \rec_{F_{\wv_0}}^T(\pi_{n_1, n_2})$. Let $\lambda : n = \lambda_1 + \dots + \lambda_k$ be a partition of $n$, and suppose that $\lambda \neq n$, $(n_1, n_2)$. Then there exists a prime ideal $\mathfrak{r} \subset \cO[\Lambda_{\wv_0}]$ in the support of $\pi_\lambda^{\Iw_{\wv_0}}$ such that there is no $W_{F_{\wv_0}}$-invariant filtration $0 \subset \Fil_1 \subset \dots \subset \Fil_n = \overline{\bQ}_p^n$ of $\overline{\bQ}_p^n$ such that for each $i = 1, \dots, n$, $\Fil_i / \Fil_{i-1} \cong \overline{\bQ}_p( \chi_{\wv_0, i, \mathfrak{r}} )$.
\end{lemma}
\begin{proof}
We remark first that $\mathrm{WD}(\rho) = \mathrm{WD}(\rho)^{F-ss}$ since the image under $\rho$ of any Frobenius lift is regular semisimple. The representations $\mathrm{WD}(\rho)$ and $\rho$ have the same invariant subspaces. The result therefore follows from Corollary \ref{cor_invariant_WD_filtrations}.
\end{proof}
\begin{lemma}\label{lem_existence_of_p-adic_character}
Let $v_0 \in S - S_p$ and let $\alpha \in 1 + \varpi \cO$. Then, after possibly enlarging $E$, we can find a continuous character $\psi : G_{F, S_p} \to 1 + \varpi \cO$ with the following properties:
\begin{enumerate}
    \item $\psi \psi^c = 1$.
    \item $\psi(\Frob_{\wv_0}) = \alpha$.
\end{enumerate}
\end{lemma}
\begin{proof}
Let $L / F$ denote the maximal abelian pro-$p$ extension of $F$ which is unramified away from $S_p$, and let $\Delta = \Gal(L / F)$. We need to show that, after possibly enlarging $E$, there is a homomorphism $\Delta / (c+1) \to 1 + \varpi \cO$ sending $\Frob_{\wv_0}$ to $\alpha$. This will be the case if $\Frob_{\wv_0}$ does not have finite order in the finitely generated $\bZ_p$-module $\Delta / (c+1)$, or if $\Frob_{\wv_0} / \Frob_{\wv_0^c}$ does not have finite order in $\Delta$. By class field theory, it suffices to show that if $\beta \in \cO_F$ generates a (positive) power of the ideal corresponding to the place $\wv_0$, then $\beta/ \beta^c$ has infinite order in $\overline{\cO_F^\times} \backslash (\cO_F \otimes_\bZ \bZ_p)^\times$. If it does not have infinite order then there is $N \geq 1$ such that $(\beta/ \beta^c)^N \in \overline{\cO_F^\times}$; since $F$ is a CM field, there is $M \geq 1$ such that $\overline{\cO_F^\times}^M \subset \overline{\cO_{F^+}^\times}$, so after increasing $N$ we can assume that $(\beta / \beta^c)^N$ is fixed by $c$, therefore that $(\beta)^{2N} = (\beta^c)^{2N}$. This is a contradiction since these ideals are distinct ($v_0$ splits in $F$).
\end{proof}
We can now prove our main level-raising theorem. For the reader's convenience we first state our assumptions:
\begin{itemize}
    \item $F$ is a CM field of the form $F = F^+ F_0$, where $F_0$ is an imaginary quadratic field and $F / F^+$ is everywhere unramified.
    \item $p, q$ are primes which split in $F_0$.
    \item $v_0$ is a $q$-adic place of $F^+$ and $u_0$ is a $q$-adic place of $F_0$. We take $\wv_0$ to be the unique place of $F$ lying above both $v_0$ and $u_0$.
    \item $\iota : \overline{\bQ}_p \to \bC$ is an isomorphism.
\end{itemize}
\begin{theorem}\label{thm_level_raising_for_GL(n)}
With the above assumptions, fix a partition $n = n_1 + n_2$ and let $\pi_1, \pi_2$ be cuspidal, conjugate self-dual automorphic representations of $\GL_{n_1}(\bA_F)$, $\GL_{n_2}(\bA_F)$ such that $\pi = \pi_1 \boxplus \pi_2$ is regular algebraic and $\iota$-ordinary. Suppose that the following conditions are satisfied:
\begin{enumerate}
    \item If $w$ is a place of $F$ such that $\pi_w$ is ramified, then $w$ is split over $F^+$.
    \item The representations $\pi_1, \pi_2$ satisfy condition (1) of Theorem \ref{thm_endoscopic_forms_exist}.
    \item\label{ass_LR_genericity} There exists a finite place $w \nmid S$ of $F$, split over $F^+$, such that there is an isomorphism $\overline{r_\iota(\pi)}|_{G_{F_w}}^{ss} \cong \oplus_{i=1}^n \overline{\chi}_i$, where $\overline{\chi}_1, \dots, \overline{\chi}_n$ are unramified characters such that if $i \neq j$ then $\overline{\chi}_i / \overline{\chi}_j \neq \epsilon$. 
    \item There are unramified characters $\psi_1, \psi_2 : F_{\wv_0}^\times \to \bC^\times$ such that $\pi_{1, \wv_0} \cong \St_{n_1}(\psi_1)$, $\pi_{2, \wv_0} \cong \St_{n_2}(\psi_2)$, and the character $\iota^{-1}( \psi_1 \psi_2^{-1} | \cdot |^{-n/2} )$ takes values in $1 + \m_{\overline{\bZ}_p}$.
    \item If $w$ is a $q$-adic place of $F$ not lying above $v_0$, then $\pi_w$ is unramified.
\end{enumerate}
Then there exists a RACSDC $\iota$-ordinary automorphic representation $\Pi$ of $\GL_n(\bA_F)$ such that $\overline{r_\iota(\Pi)} \cong \overline{r_\iota(\pi)}$ and $\Pi_{\wv_0}$ is an unramified twist of $\St_n$. Moreover, if $w$ is a $q$-adic place of $F$ not lying above $v_0$, then $\Pi_w$ is unramified; and if $w$ is any place of $F$ such that $\Pi_w$ is ramified, then $w$ is split over $F^+$.
\end{theorem}
\begin{proof}
We can choose $S$, $\widetilde{S}$, and $K^p \in \cJ_S^p$ so that the following conditions are satisfied:
\begin{itemize}
    \item $S$ contains $S_p$, $v_0$, and each place below which $\pi$ is ramified. $S$ does not contain any $q$-adic places of $F^+$ apart from $v_0$.
    \item For all $v \in S - S_p$, $\pi_\wv^{\iota_\wv(K_v)} \neq 0$ and $\iota_{\wv_0}(K_{v_0}) = \Iw_{\wv_0}$.
\end{itemize}
Suppose for contradiction that the conclusion of the theorem does not hold. By Proposition \ref{prop_ordinary_endoscopic_descent}, we can find (after possibly enlarging $E$) a locally dominant algebraic character $\lambda_\pi \in \Hom_{cts}(\mathfrak{t}_p, \cO^\times)$ and a maximal ideal $\m \subset \T^S$ in the support of $S_{\lambda_\pi}(K^p, k) \cong S(K^p, k)$ such that for each place $w \nmid S$ of $F$ which is split over $F^+$, $P_w(X) \text{ mod }\m = \det(X - \overline{r_\iota(\pi)}(\Frob_w))$. Thus $\cS$, $\cH_\lambda$ are defined for any $\lambda \in \Hom_{cts}(\mathfrak{t}_p, \cO^\times)$, and Corollary \ref{cor_vanishing_of_derived_tensor_product} shows that $\cH_\lambda \otimes_\cO k \cong \cS$.

We  claim that 
\begin{equation}\label{eqn_vanishing_of_derived_tensor_product} R \Gamma_c ( \Omega_{F_{w_0}}^{n-1, ca}, k) \otimes^{\bL}_{k[\GL_n(F_{w_0})^0]} \cS = 0. 
\end{equation}
Theorem \ref{thm_torsion_in_derived_tensor_product} shows that the group
\[ H^i(R \Gamma_c ( \Omega_{F_{w_0}}^{n-1, ca}, k) \otimes^{\bL}_{k[\GL_n(F_{w_0})^0]} \cS) \]
is non-zero only if $i = n-1$. It follows that the group
\[ H^i(R \Gamma_c ( \Omega_{F_{w_0}}^{n-1, ca}, \cO) \otimes^{\bL}_{\cO[\GL_n(F_{w_0})^0]} \cH_{\lambda_\pi}) \]
is zero if $i \neq n-1$ and is a finite free $\cO$-module if $i = n-1$. However, if 
\[ H^{n-1}(R \Gamma_c ( \Omega_{F_{w_0}}^{n-1, ca}, E) \otimes^{\bL}_{E[\GL_n(F_{w_0})^0]} \cH_{\lambda_\pi}[1/p]) \neq 0\]
then Theorem \ref{thm_homology_of_G_0_elliptic_representations} shows that $\cH_{\lambda_\pi}[1/p]$ contains an unramified twist of the Steinberg representation of $\GL_n(F_{w_0})$, contradicting our assumption that there is no $\iota$-ordinary automorphic representation  $\Pi$ of $\GL_n(\bA_F)$ such that $\overline{r_\iota(\Pi)} \cong \overline{r_\iota(\pi)}$ and $\Pi_{w_0}$ is an unramified twist of $\St_n$. (Note that $\cH_{\lambda_\pi}[1/p]$ is a semi-simple $E[\GL_n(F_{w_0})]$-module, and all of its subquotients are generic, by the remark after Corollary \ref{cor_Gal_reps_for_G}.) This proves the claimed equality (\ref{eqn_vanishing_of_derived_tensor_product}). Applying Corollary \ref{cor_vanishing_of_derived_tensor_product}, we deduce the vanishing
\begin{equation}\label{eqn_explicit_vanishing_of_tensor_products} \cH_\lambda[1/p] \otimes^{\bL}_{E[\GL_n(F_{w_0})^0]} \pi_{i+1, 1, \dots, 1} = 0 
\end{equation}
for each $i = 0, \dots, n-1$ and for any $\lambda \in \Hom_{cts}(\mathfrak{t}_p, \cO^\times)$.

Using Lemma \ref{lem_existence_of_p-adic_character} and after possibly enlarging $\cO$, we can find a character $\psi : G_{F, S_p} \to \cO^\times$ with the following properties:
\begin{itemize}
    \item $\psi \psi^c = 1$.
    \item There is an equality
    \[ | \cdot |^{(1-n_1)/2} \psi_1 (\iota \psi \circ \Art_{F_{\wv_0}}) = | \cdot |^{(1 + n_2)/2} \psi_2 \]
    of characters $F_{\wv_0}^\times \to \bC^\times$.
\end{itemize}
According to Proposition \ref{prop_non-classical_twisted_endoscopic_descent} and Corollary \ref{cor_twisted_endoscopic_descent}, there is a character $\lambda_{\pi_\psi} \in \Hom_{cts}(\mathfrak{t}_p, \cO^\times)$ and a homomorphism $\bT^S(H_{\lambda_{\pi_\psi}}(K^p)) \to \cO$ of kernel $\mathfrak{p}$ with the following properties:
\begin{itemize}
    \item For each place $w \nmid S$ of $F$ which splits over $F^+$, the image of $P_w(X)$ in $\cO[X]$ equals $\det(X - ((r_\iota(\pi_1 | \cdot |^{(n_1 - n)/2}) \otimes \psi) \oplus r_\iota(\pi_2 | \cdot |^{(n_2 - n)/2})(\Frob_w))$.
    \item The only Iwahori-spherical subquotients of the admissible $E[\GL_n(F_{w_0})]$-module $\cH_{\lambda, \p}$ are isomorphic to one of $\pi_n(\iota^{-1} \psi_2 | \cdot |^{n_1/2}) = \St_n(\iota^{-1} \psi_2 | \cdot |^{n_1/2})$ and $\pi_{n_1, n_2}(\iota^{-1} \psi_2 | \cdot |^{n_1/2})$.
\end{itemize}
Since $\cH_{\lambda_{\pi_\psi}, \p}^{\Iw_{\wv_0}} = H_{\lambda_{\pi_\psi}}(K^p)_\p$ is non-zero, there must exist at least one such Iwahori-spherical subquotient of this representation. It follows that there exists either an injection $\pi_n(\iota^{-1} \psi_2 | \cdot |^{n_1/2}) \to \cH_{\lambda_{\pi_\psi}, \p}[1/p] $ or an injection $\pi_{n_1, n_2}(\iota^{-1} \psi_2 | \cdot |^{n_1/2}) \to \cH_{\lambda_{\pi_\psi}, \p}[1/p]$ of $E[\GL_n(F_{w_0})]$-modules. In either case, we will derive a contradiction.

Suppose first that there is an injection $\pi_n(\iota^{-1} \psi_2 | \cdot |^{n_1/2}) \to \cH_{\lambda_{\pi_\psi}, \p}[1/p]$. Theorem \ref{thm_homology_of_G_0_elliptic_representations} shows that 
\[ H^{i}( \pi_n(\iota^{-1} \psi_2 | \cdot |^{n_1/2}) \otimes^{\bL}_{E[\GL_n(F_{w_0})^0]} E) \]
is non-zero only if $i = -(n-1)$ (in which case it is a 1-dimensional $E$-vector space), while 
\[ H^{i}( \pi_{n_1, n_2}(\iota^{-1} \psi_2 | \cdot |^{n_1/2}) \otimes^{\bL}_{E[\GL_n(F_{w_0})^0]} E) \]
is non-zero only if $i = -(n-2)$ (in which case it is a 1-dimensional $E$-vector space). Thus $H^i(\cH_{{\lambda_{\pi_\psi}}, \p}[1/p]  \otimes^{\bL}_{E[\GL_n(F_{w_0})^0]} E)$ is zero in degrees $i < -(n-1)$ and  there is an injection
\begin{multline*}  H^{-(n-1)}( \pi_n(\iota^{-1} \psi_2 | \cdot |^{n_1/2}) \otimes^{\bL}_{E[\GL_n(F_{w_0})^0]} E) = E \\
\to H^{-(n-1)}( \cH_{\lambda_{\pi_\psi}, \p}[1/p]  \otimes^{\bL}_{E[\GL_n(F_{w_0})^0]} E) 
\end{multline*}
contradicting (\ref{eqn_explicit_vanishing_of_tensor_products}). 

Suppose instead that there is an injection $\pi_{n_1, n_2}(\iota^{-1} \psi_2 | \cdot |^{n_1/2}) \to \cH_{{\lambda_{\pi_\psi}}, \p}[1/p] $. Theorem \ref{thm_homology_of_G_0_elliptic_representations} shows that 
\[ H^i( \pi_n(\iota^{-1} \psi_2 | \cdot |^{n_1/2}) \otimes^{\bL}_{E[\GL_n(F_{w_0})^0]} \pi_{n_2+1, 1, \dots, 1}) \]
is non-zero only if $i = -(n_1-1)$ (in which case it is a 1-dimensional $E$-vector space), while 
\[ H^i( \pi_{n_1, n_2}(\iota^{-1} \psi_2 | \cdot |^{n_1/2}) \otimes^{\bL}_{E[\GL_n(F_{w_0})^0]} \pi_{n_2+1, 1, \dots, 1}) \]
is non-zero only if $i = -n_1$ (in which case it is a 1-dimensional $E$-vector space). Thus $H^i(\cH_{{\lambda_{\pi_\psi}}, \p}[1/p]  \otimes^{\bL}_{E[\GL_n(F_{w_0})^0]} \pi_{n_2+1, 1, \dots, 1})$ is zero in degrees $i < - n_1$ and  there is an injection
\begin{multline*}  H^{n_1}( \pi_{n_1, n_2}(\iota^{-1} \psi_2 | \cdot |^{n_1/2}) \otimes^{\bL}_{E[\GL_n(F_{w_0})^0]} \pi_{n_2+1, 1, \dots, 1}) = E \to \\H^{n_1}( \cH_{{\lambda_{\pi_\psi}}, \p}[1/p]  \otimes^{\bL}_{E[\GL_n(F_{w_0})^0]} \pi_{n_2+1, 1, \dots, 1}),
\end{multline*}
again contradicting (\ref{eqn_explicit_vanishing_of_tensor_products}). This completes the proof.
\end{proof}

\section{Application to symmetric power functoriality}

The following conjecture is \cite[Conjecture 3.2]{Clo14}.
\begin{conjecture}[$\mathrm{TP}_r$]
Let $r \geq 1$. Let $F$ be a CM field and let $(\pi_1, \psi_1)$, $(\pi_2, \psi_2)$ be RAECSDC automorphic representations of $\GL_2(\bA_F)$, $\GL_r(\bA_F)$, respectively. Fix a prime $p$ and isomorphism $\iota : \overline{\bQ}_p \to \bC$. If the representation $\rho = r_\iota(\pi_1) \otimes r_\iota(\pi_2)$ is irreducible and Hodge--Tate regular, then there exists a RAECSDC automorphic representation $(\pi_3, \psi_3)$ of $\GL_{2r}(\bA_F)$ such that $r_\iota(\pi_3) \cong \rho$.
\end{conjecture}
The following conjecture is \cite[Conjecture 3.3]{Clo14} in the most general case $\mathbf{K} = \emptyset$.
\begin{conjecture}[$\mathrm{SP}_n$]
Let $n \geq 1$. Let $F^+$ be a totally real field, and let $(\pi, \chi)$ be a RAESDC automorphic representation of $\GL_2(\bA_{F^+})$ without CM. Then there exists a RAESDC automorphic representation $(\Pi, \psi)$ of $\GL_{n}(\bA_{F^+})$ such that for any prime $p$ and isomorphism $\iota : \overline{\bQ}_p \to \bC$, $r_\iota(\Pi) \cong \Sym^{n-1} r_\iota(\pi)$.
\end{conjecture}
The main result we will prove here is the following (cf. \cite[Theorem 3.4]{Clo14}:
\begin{theorem}\label{thm_implication_of_SP_n}
Let $p \geq 5$ be a prime, and let $0 < r < p$ be an integer. Then the implication 
\[ \mathrm{SP}_{p-r} + \mathrm{SP}_{r} + \mathrm{TP}_r \Rightarrow \mathrm{SP}_{p+r} \]
holds. 
\end{theorem}
Before proceeding to the proof of Theorem \ref{thm_implication_of_SP_n} we give a corollary. 
\begin{corollary}\label{cor_unconditional_reg_alg_powers}
\begin{enumerate}
    \item $\mathrm{SP}_n$ holds for all $n = 1, \dots, 10$, and for all even integers $n = 12, \dots, 26$.
    \item Assume $\mathrm{TP}_r$ for all $r \geq 1$. Then $\mathrm{SP}_n$ holds for all $n \geq 1$.
\end{enumerate}
\end{corollary}
\begin{proof}
The proof is the same as the proofs of \cite[Corollary 3.5, Corollary 3.6]{Clo14}. We give the details again here. We know that $\mathrm{SP}_n$ is true for $n = 1, 2, 3, 4, 5$ and $\mathrm{TP}_r$ is true for $r = 1, 2, 3$ (see \cite[\S 3]{Clo14} for detailed references). Therefore we have the unconditional implications, valid for any prime number $p \geq 5$:
\[ \mathrm{SP}_{p-1} \Rightarrow \mathrm{SP}_{p+1}, \mathrm{SP}_{p-2} \Rightarrow \mathrm{SP}_{p+2}, \mathrm{SP}_{p-3} \Rightarrow \mathrm{SP}_{p+3}. \]
It is now elementary to check that this implies the first part of the corollary, using $p = 5, \dots, 23$. For the second, we use induction on $n$. Take an integer $n > 5$. Bertrand's postulate implies that there is a prime number $p$ such that $n < p < 2n$. Writing $n = p + r$, we have $0 < r < p$ and we deduce $\mathrm{SP}_{n} = \mathrm{SP}_{p+r}$ by Theorem \ref{thm_implication_of_SP_n} and induction.
\end{proof}
We now get everything in place to prove Theorem \ref{thm_implication_of_SP_n}, following the lines of \cite{Clo14}. 
\begin{lemma}\label{lem_tensor_adequacy}
Let $p \geq 5$ be a prime and let $\varphi \in G_{\bQ_p}$ be a lift of arithmetic Frobenius. There exists an integer $a_0 = a_0(p) \geq 3$ such that if $a \geq a_0$ and  $G \leq \GL_2(\overline{\bF}_p)$ is a finite subgroup which contains a conjugate of $\SL_2(\bF_{p^a})$, then $G$ has the following properties:
\begin{enumerate}
    \item For each $0 < r  < p$, the image of the homomorphism $\Sym^{r-1} : G \to \GL_r(\overline{\bF}_p)$ is adequate,  in the sense of \cite[Definition 2.3]{Tho12}. 
    \item For each $0 < r < p$, the image of the homomorphism ${}^\varphi \mathrm{Std} \otimes \Sym^{r-1} : G \to \GL_{2r}(\overline{\bF}_p)$ is adequate.
    \item If $H \leq G$ is a subgroup of index $[G : H] < 2p$, then $H$ contains a conjugate of $\SL_2(\bF_{p^a})$.
\end{enumerate}
\end{lemma}
\begin{proof}
    If $a \geq 3$ then the classification of finite subgroups of $\PGL_2(\overline{\bF}_p)$ shows that the projective image of $G$ is conjugate to $\PSL_2(\bF_{p^b})$ or $\PGL_2(\bF_{p^b})$ for some $b \geq a$, and therefore that there is (after replacing $G$ by a conjugate) a sandwich
    \[ \SL_2(\bF_{p^b}) \leq G \leq  k^\times \SL_2(\bF_{p^b}) \]
    for some finite extension $k / \bF_{p^b}$. The first two points therefore follow from \cite[Appendix, Theorem 2, Lemma 4]{Die15}. For the third, let us denote by $\overline{H}$, $\overline{G}$ the respective images of $H$, $G$ in $\PGL_2(\overline{\bF}_p)$. We have $[\overline{G} : \overline{H}] < 2p$, hence $[\PSL_2(\bF_{p^b}) : H \cap \PSL_2(\bF_{p^b})] < 2p$. The classification of finite subgroups of $\PSL_2(\bF_{p^b})$ shows that as $b \geq 3$, $\overline{H} \cap \PSL_2(\bF_{p^b})$ must equal $\PSL_2(\bF_{p^b})$ (because all the other subgroups of $\PGL_2(\bF_{p^b})$ have index greater than $2p$). In particular, $H$ must contain $\SL_2(\bF_{p^b})$.
\end{proof}
Here is an analogue of \cite[Theorem 4.2]{Clo14}.
\begin{proposition}\label{prop_SP_LR}
Let $F$ be an imaginary CM field and let $\pi$ be a RACSDC automorphic representation of $\GL_2(\bA_F)$ of weight 0. Let $p \geq 5$ be a prime and let $\iota : \overline{\bQ}_p \to \bC$ be an isomorphism. Fix an integer $0 < r <  p$, and suppose that the following conditions are satisfied:
\begin{enumerate}
    \item The extension $F / F^+$ is everywhere unramified and linearly disjoint from $F^+(\zeta_p) / F^+$. Moreover, $F$ contains an imaginary quadratic field $F_0$ in which $p$ splits.
    \item If $v$ is a place of $F$ such that $\pi_v$ is ramified, then $v$ is split over $F^+$. 
    \item $\pi$ is $\iota$-ordinary.
    \item There exists an integer $a \geq a_0(p)$ such that $\overline{r_\iota(\pi)}(G_{F})$ contains a conjugate of $\SL_2(\bF_{p^a})$.
    \item There exist RACSDC automorphic representations $\Pi_1, \Pi_2$ of $\GL_{r}(\bA_F)$, $\GL_{p-r}(\bA_F)$, respectively, such that $r_\iota(\Pi_1) \cong \Sym^{r-1} r_\iota(\pi)$ and $r_\iota(\Pi_2) \cong \Sym^{p-r-1} r_\iota(\pi)$. 
    \item There exists a place $w_0 \nmid p$ of $F$ such that $\pi_{w_0}$ is an unramified twist of the Steinberg representation and $q_{w_0} \equiv 1 \text{ mod }p$. The residue characteristic $q$ of $w_0$ splits in $F_0$, and $\pi_w$ is unramified for each $q$-adic place of $F$ such that $w \neq w_0, w_0^c$. Moreover $\overline{r_\iota(\pi)}|_{G_{F_{w_0}}}$ is trivial. 
\end{enumerate}
Suppose finally that Conjecture $\mathrm{TP}_r$ holds. Then there exists a RACSDC automorphic representation $\Pi$ of $\GL_{p+r}(\bA_F)$ with the the following properties:
\begin{enumerate}
    \item $\Pi$ is $\iota$-ordinary.
    \item There is an isomorphism $\overline{r_\iota(\Pi)} \cong (\Sym^{p+r-1} \overline{r_\iota(\pi)})^{ss}$.
    \item $\Pi_{w_0}$ is an unramified twist of the Steinberg representation.
\end{enumerate}
\end{proposition}
\begin{proof}
    Using \cite[Theorem 9.1, Theorem 10.2]{Tho12} (i.e. the Khare--Wintenberger method), we can find an $\iota$-ordinary RAECSDC automorphic representation $(\pi', |\cdot|^{1-p})$ of $\GL_2(\bA_F)$ satisfying the following  conditions:
\begin{itemize}
    \item There is an isomorphism $\overline{r_\iota(\pi')} \cong \overline{r_\iota(\pi)}$.
    \item For each finite place $v \nmid p$ of $F$, $r_\iota(\pi)|_{G_{F_v}} \sim r_\iota(\pi')|_{G_{F_v}}$ (the relation $\sim$ as defined in \cite[\S 1.3]{Bar14}). This implies in particular that $\pi'_{w_0}$ is an unramified twist of the Steinberg representation.
    \item For each embedding $\tau : F \to \overline{\bQ}_p$, $\mathrm{HT}_\tau(r_\iota(\pi')) = \{ 0, p \}$.
\end{itemize}
By Conjecture $\mathrm{TP}_r$, there exists a RAECSDC automorphic representation $(\Pi'_1, | \cdot |^{r-p})$ such that $r_\iota(\Pi'_1) \cong {}^{\varphi} r_\iota(\pi') \otimes \Sym^{r-1} r_\iota(\pi)$. By \cite[Lemma 3.7]{Clo14}, $r_\iota(\Pi'_1)$ is $\iota$-ordinary. Using the Khare--Wintenberger method again, we can find another RAECSDC automorphic representation $(\Pi''_1, |\cdot|^{r-p})$ satisfying the following conditions:
\begin{itemize}
\item There is an isomorphism $\overline{r_\iota(\Pi''_1)} \cong \overline{r_{\iota}(\Pi'_1)}$. (In particular, the image of $G_{F(\zeta_p)}$ under both residual representations is adequate, by Lemma \ref{lem_tensor_adequacy} and our assumption $a \geq a_0(p)$.)
\item $\Pi''_1$ is $\iota$-ordinary and of the same weight as $\Pi'_1$.
 \item For each finite place $v \nmid p w_0 w_0^c$ of $F$, $r_\iota(\Pi''_1)|_{G_{F_v}} \sim r_\iota(\Pi'_1)|_{G_{F_v}}$.
 \item $\Pi''_1$ is an unramified twist of the Steinberg representation.
\end{itemize}
We let $\sigma_1 = \Pi''_1 | \cdot |^{(p-r)/2}$. Then $\sigma_1$ is cuspidal and conjugate self-dual.  Let $\chi = \det r_\iota(\pi)$, and let $\sigma_2 = \Pi_2 \otimes \iota (\chi \epsilon)^r$. Then $\sigma_2$ is cuspidal and conjugate self-dual, $\sigma_{2, w_0}$ is an unramified twist of the Steinberg representation, and $\Sigma = \sigma_1 \boxplus \sigma_2$ is regular algebraic and $\iota$-ordinary of weight 0. 
    
    The conclusion of the proposition now follows from Theorem \ref{thm_level_raising_for_GL(n)} applied to $\Sigma$, provided we can check hypothesis (\ref{ass_LR_genericity}) of Theorem \ref{thm_level_raising_for_GL(n)}, i.e.\ the existence of $\gamma \in G_F$ such that $\overline{r_\iota(\Sigma)}(\gamma)$, or equivalently $(\Sym^{p+r-1} \overline{r_\iota(\pi)})(\gamma)$, has $p+r$ distinct eigenvalues, no two of which are in ratio $\epsilon(\gamma)$. Since $\overline{r_\iota(\pi)}(G_{F(\zeta_p)})$ contains a conjugate of $\SL_2(\bF_{p^a})$, we can choose $\gamma \in G_{F(\zeta_p)}$ with image a conjugate of $\diag(t, t^{-1})$ for any $t \in \bF_{p^a}$ such that $1, t^2, \dots, t^{2(p+r-1)}$ are all distinct. Such elements exist because of our assumption that $a \geq 3$.
\end{proof}
We next want to state a slightly strengthened version of the main result of \cite{All20}:
\begin{theorem}\label{thm_modified_ALT}
Let $F$ be a CM number field, let $n \geq 1$, let $p \geq 5$ be a prime not dividing $n$, and let $\rho : G_F \to \GL_n(\overline{\bQ}_p)$ be a continuous representation satisfying the following conditions:
\begin{enumerate}
    \item $\rho^{c} \cong \rho^\vee \epsilon^{1-n}$.
    \item $\rho$ is ramified at only finitely many places.
    \item There exists $\lambda \in (\bZ^n_+)^{\Hom(F, \overline{\bQ}_p)}$ such that $\rho$ is ordinary of weight $\lambda$.
    \item There is an isomorphism $\overline{\rho} \cong \oplus_{i=1}^r \overline{\rho}_i$, where each representation $\overline{\rho}_i$ is absolutely irreducible and satisfies $\overline{\rho}_i^c \cong \overline{\rho}_i^{\vee} \epsilon^{1-n}$ and $\overline{\rho}_i \not\cong \overline{\rho}_j$ if $i \neq j$.
    \item There is a place $w_0 \nmid p$ of $F$ such that $\rho|_{G_{F_{w_0}}}^{ss}$ is an unramified twist of $\oplus_{i=1}^n \epsilon^{1-i}$.
    \item There exists an $\iota$-ordinary RACSDC  automorphic representation $\Pi$ of $\GL_n(\bA_F)$ such that $\overline{r_\iota(\Pi)} \cong \overline{\rho}$ and $\Pi_{w_0}$ is an unramified twist of the Steinberg representation.
    \item There exists a place $w \nmid p$ of $F$ such that $\rho|_{G_{F_w}}$ is unramified and $H^0(F_w, \ad \overline{\rho}(1)) = 0$. We have $F \not\subset F^+(\zeta_p)$. For each $1 \leq i, j \leq r$, $\overline{\rho}_i|_{G_{F(\zeta_p)}}$ is absolutely irreducible and $\overline{\rho}_i|_{G_{F(\zeta_p)}} \not\cong \overline{\rho}_j|_{G_{F(\zeta_p)}}$ if $i \neq j$. Moreover, $\overline{\rho}$ is primitive and $\overline{\rho}(G_F)$ has no quotient of order $p$. 
\end{enumerate}
Then $\rho$ is automorphic: there exists an $\iota$-ordinary RACSDC automorphic representation $\pi$ of $\GL_n(\bA_F)$ such that $\rho \cong r_\iota(\pi)$.
\end{theorem}
\begin{proof}
    This theorem is identical in statement to \cite[Theorem 6.1]{All20}, except that the  condition ``$F(\zeta_p)$ not contained in $\overline{F}^{\ker \ad \overline{\rho}}$'' has been replaced by the condition ``there exists a place $w \nmid p$ of $F$ such that $\rho|_{G_{F_w}}$ is unramified and $H^0(F_w, \ad \overline{\rho}(1)) = 0$''. The former condition is used only to show the existence, by the Chebotarev density theorem, of a place $w \nmid p$ of $F$ such that $\rho|_{G_{F_{w}}}$ is unramified, $q_w \not\equiv 1 \text{ mod }p$, and $\overline{\rho}|_{G_{F_w}}$ is scalar. This implies that $H^0(F_w, \ad \overline{\rho}(1)) = 0$ and therefore that any lifting of $\overline{\rho}|_{G_{F_w}}$ is unramified. On the automorphic side, this means that we can choose a sufficiently small level subgroup at the place $w$ without introducing any new ramification. However, we can in fact take $w$ to be any place such that $\rho|_{G_{F_w}}$ is unramified and $H^0(F_w, \ad \overline{\rho}(1)) = 0$ and the rest of the argument remains unchanged. 
\end{proof}
We can use Theorem \ref{thm_modified_ALT} to get the following consequence of Proposition \ref{prop_SP_LR}. In contrast to the results of \cite{Clo14, Clo17}, there are no conditions here on the presence (or otherwise) of roots of unity in the base field $F^+$.
\begin{theorem}\label{thm_implication_for_SP_n_with_conditions}
Let $p \geq 5$ be a prime, let $0 < r < p$ be an integer, and assume Conjectures $\mathrm{SP}_{p-r}$, $\mathrm{SP}_r$, and $\mathrm{TP}_r$. Let $F^+$ be a totally real field and let $(\pi, \chi)$ be a RAESDC automorphic representation of $\GL_2(\bA_{F^+})$ satisfying the following conditions:
\begin{enumerate}
    \item There exists an isomorphism $\iota : \overline{\bQ}_p \to \bC$ such that $\pi$ is $\iota$-ordinary and $\overline{r_\iota(\pi)}(G_F)$ contains a conjugate of $\SL_2(\bF_{p^a})$, for some $a > a_0(p)$.
    \item $\pi$ has weight 0. 
    \item There exists a place $v_0 \nmid p$ such that $\pi_{v_0}$ is an unramified twist of the Steinberg representation. 
\end{enumerate}
Then the $(p+r-1)^\text{th}$ symmetric power of $\pi$ exists: there is an $\iota$-ordinary cuspidal RAESDC automorphic representation $\Pi$ of $\GL_{p+r}(\bA_{F^+})$ such that $\Sym^{p+r-1} r_\iota(\pi) \cong r_\iota(\Pi)$.
\end{theorem}
\begin{proof}
We are free, by soluble base change, to replace $F^+$ by any soluble totally real extension. We can therefore assume that $q_{v_0} \equiv 1 \text{ mod }p$ and that there exists a CM quadratic extension $F / F^+$ with the following properties:
\begin{itemize}
    \item The extension $F / F^+$ is everywhere unramified and linearly disjoint from the extension of $F^+(\zeta_p)$ cut out by $\overline{r_\iota(\pi)}|_{G_{F^+(\zeta_p)}}$. In particular, $F \not\subset F^+(\zeta_p)$ and $\overline{r_\iota(\pi)}(G_F) = \overline{r_\iota(\pi)}(G_{F^+})$.
    \item $F$ contains an imaginary quadratic field $F_0$ in which $p$ and the residue characteristic $q$ of the place $v_0$ split.
    \item For each $q$-adic place $w$ of $F$, $\overline{r_\iota(\pi)}|_{G_{F_w}}$ is trivial. 
    \item Each place $v$ of $F^+$ such that $\pi_v$ is ramified splits in $F$.
    \item There exists an everywhere unramified Hecke character $\psi : F^\times \backslash \bA_F^\times \to \bC^\times$ of type $A_0$ such that $\chi = \psi \circ \mathbf{N}_{F / F^+}$.
\end{itemize}
    Increasing $a$, we assume also that the projective image of $\overline{r_\iota(\pi)}$ is conjugate either to $\PSL_2(\bF_{p^a})$ or $\PGL_2(\bF_{p^a})$. Let $\pi_F$ denote the base change of $\pi$ to $F$, and let $\rho = r_\iota( \pi_F \otimes \psi^{-1})$, $r = \Sym^{p+r-1} \rho$.  The representation $\pi_F \otimes \psi^{-1}$ is RACSDC. Let $w_0$ be a place of $F$ lying above $v_0$. Using \cite[Theorem 9.1, Theorem 10.2]{Tho12} again we can find another RACSDC $\iota$-ordinary automorphic representation $\pi'$ of $\GL_2(\bA_F)$ such that $\overline{r_\iota(\pi')} \cong \overline{\rho}$, $\pi'_{w_0}$ is an unramified twist of the Steinberg representation, and such that for each $q$-adic place $w \neq w_0, w_0^c$ of $F$, $\pi'_w$ is unramified.
    
    We can then apply Theorem \ref{thm_modified_ALT} to deduce the automorphy of $\Sym^{p+r-1} \rho$, using Proposition \ref{prop_SP_LR} applied to $\pi'$ to verify the residual automorphy hypothesis, provided we can check the conditions on the residual representation $\overline{r}$. Before we carry out these checks, we note that this will conclude the proof: the automorphy of $\Sym^{p+r-1} r_\iota(\pi)$ follows by quadratic descent. 
    
    The representation $\overline{r}$ has two irreducible constituents $\overline{r}_1 = {}^\varphi \overline{\rho} \otimes \Sym^{r-1} \overline{\rho}, \overline{r}_2 = \det \overline{\rho}^r \otimes \Sym^{p-r-1} \overline{\rho}$ of distinct dimensions, which remain irreducible on restriction to $G_{F(\zeta_p)}$, so the conditions that require an argument to check are as follows:
    \begin{itemize}
        \item $\overline{r}$ is primitive.
        \item There exists a place $w \nmid p$ of $F$ such that $\overline{r}|_{G_{F_w}}$ is unramified and $H^0(F_w, \ad \overline{r}(1)) = 0$.
    \end{itemize}
    We treat these in turn. If $\overline{r}$ is not primitive, then there is an isomorphism $\overline{r} \cong \Ind_{G_L}^{G_F} \overline{\sigma}$ for some finite extension $L / F$ of degree $> 1$, hence an embedding $\overline{\sigma} \hookrightarrow \overline{r}|_{G_L}$. The extension of $F$ cut out by $\overline{r}$ has Galois group isomorphic to a subgroup of $\GL_2(\overline{\bF}_p)$ whose image in $\PGL_2(\overline{\bF}_p)$ is a conjugate of $\PSL_2(\bF_{p^a})$ or $\PGL_2(\bF_{p^a})$. Since we assume $a \geq a_0(p)$, and $[ L : F ] \leq \dim r < 2p$, it follows from Lemma \ref{lem_tensor_adequacy} that $\overline{r}(G_L)$ contains a conjugate of $\SL_2(\bF_{p^a})$, and therefore that the irreducible constituents of $\overline{r}$ remain irreducible on restriction to $G_L$. We find that $\overline{\sigma}$ is isomorphic to one of $\overline{r}_1|_{G_L}$ or $\overline{r}_2|_{G_L}$. In particular, $\overline{\sigma}$ extends to $G_F$ and we can take it outside of the induction to get an isomorphism
    \[ \overline{r} \cong \Ind_{G_L}^{G_F} \overline{\sigma} \cong \overline{r}_i \otimes \Ind_{G_L}^{G_F} \overline{\bF}_p \]
    for some $i = 1, 2$. Let $d = [L : F]$. Then for any $g \in G_F$, the eigenvalues of $g$ on $ \Ind_{G_L}^{G_F} \overline{\bF}_p$ are among the roots of unity of order at most $d$, and 1 is an eigenvalue. The above isomorphism implies that for any $g \in G_F$, $\overline{r}(g)$ has two eigenvalues whose ratio is a root of unity of order at most $d$. We will show that this leads to a contradiction. Indeed, since the image of $\overline{\rho}$ contains a conjugate of $\SL_2(\bF_{p^a})$, we can find an element in the image of $\overline{r}$ which has eigenvalues $t^{p+r-1}$, $t^{p+r-3}$, \dots, $t^{-(p+r-1)}$, there $t \in \bF_{p^a}^\times$ is an element of order $p^a-1$. The ratios of these eigenvalues are $t^2, \dots, t^{2(p+r-1)}$. We'll be done therefore it none of these ratios can have order at most $d$. This would imply that $t$ itself has order at most $2d(p+r-1) < 8 p^2$. As $p \geq 5$ and $a > 3$, we have $p^a - 1 > 8p^2$, so this is the desired contradiction. 
    
    Since $\Ind_{G_L}^{G_F} \overline{\bF}_p$ decomposes as a direct sum of characters and $\overline{r}_1, \overline{r}_2$ are irreducible and of distinct dimensions, this leads to a contradiction.
    
    We now show that there exists a place $w \nmid p$ of $F$ such that $\overline{r}|_{G_{F_w}}$ is unramified and $H^0(F_w, \ad \overline{r}(1)) = 0$. It suffices to find $\sigma \in G_F$ such that $\epsilon(\sigma) = -1$ and if $\alpha, \beta \in \overline{\bF}_p$ are the eigenvalues of $\overline{\rho}(\sigma)$, then $(\alpha / \beta)^i \neq -1$ for each $i = 1, \dots, p+r-1$. Since $F / F^+$ is linearly disjoint from the extension of $F^+(\zeta_p)$ cut out by $\overline{r_\iota(\pi)}|_{G_{F^+(\zeta_p)}}$, it is even enough to find $\sigma \in G_{F^+}$ such that $\epsilon(\sigma) = -1$ and $\overline{r_\iota(\pi)}$ satisfies the analogous condition on eigenvalues.
    
    Consider the homomorphism $\ad \overline{r_\iota(\pi)} \times \epsilon : G_{F^+} \to \PGL_2(\overline{\bF}_p) \times \bF_{p}^\times$. We can assume that the projective image of $\ad \overline{r_\iota(\pi)}$ is equal to one of $\PSL_2(\bF_{p^a})$ or $\PGL_2(\bF_{p^a})$. Since $\PSL_2(\bF_{p^a})$ is a simple group, the image of $\ad \overline{r_\iota(\pi)} \times \epsilon$ contains $\PSL_2(\bF_{p^a}) \times \{ 1 \}$. This image also contains $(g, -1)$, where $g \in \PGL_2(\bF_{p^a})$ has order 2 (the image of complex conjugation). The image of $\ad \overline{r_\iota(\pi)} \times \epsilon$ therefore contains all elements of the form $(gh, -1)$, where $h \in \PSL_2(\bF_{p^a})$. We see finally that the proof will be complete if we can show that for any $g \in \PGL_2(\bF_{p^a})$ of order 2, there exists $h \in \PSL_2(\bF_{p^a})$ such that the eigenvalues $\alpha, \beta \in \overline{\bF}_p$ of $gh$ (defined only up to scalar multiplication) satisfy $(\alpha / \beta)^i \neq -1$ for each $i = 1, \dots, p+r-1$. This is elementary (for example, consider elements $h$ which commute with $g$ and use the conditions $p \geq 5$, $a \geq 3$). 
    \end{proof}
Theorem \ref{thm_implication_of_SP_n} now follows on combining Theorem \ref{thm_implication_for_SP_n_with_conditions} and the following proposition.
\begin{proposition}
Let $F^+$ be a totally real number field and let $(\pi, \chi)$ be a RAESDC automorphic representation of $\GL_2(\bA_{F^+})$ without CM. Let $p \geq 5$ be a prime and let $0 < r < p$ be an integer. Then we can find a soluble totally real extension $E^+ / F^+$ and a RAESDC automorphic representation $(\pi', \chi')$ of $\GL_2(\bA_{E^+})$ without CM satisfying the following conditions:
\begin{enumerate}
    \item $\Sym^{p+r-1} r_\iota(\pi)$ is automorphic, associated to a RAESDC automorphic representation of $\GL_{p+r}(\bA_{F^+})$, if and only if $\Sym^{p+r-1} r_\iota(\pi')$ is automorphic.
    \item $\pi'$ is of weight 0. 
    \item For each isomorphism $\iota : \overline{\bQ}_p \to \bC$, $\pi'$ is $\iota$-ordinary and $\overline{r_\iota(\pi')}(G_{E^+})$ contains a conjugate of $\SL_2(\bF_{p^a})$ for some $a > a_0(p)$. 
    \item There exists a place $v_0 \nmid p$ of $E^+$ such that $\pi'_{v_0}$ is an unramified twist of the Steinberg representation.
\end{enumerate}
\end{proposition}
\begin{proof}
This is proved by repeating verbatim the proofs of \cite[Proposition 5.2]{Clo14} and \cite[Proposition 5.3]{Clo14}.
\end{proof}

\section{The mixed parity case}

Let $F^+$ be a totally real number field. If $F^+ \neq \bQ$, then not every cuspidal Hilbert modular eigenform of regular weight over $F^+$ admits an associated Galois representation. This is related to the fact that not every cuspidal automorphic representation $\pi$ of $\GL_2(\bA_{F^+})$ such that $\pi_\infty$ is essentially square-integrable admits a twist which is algebraic. In order to deduce Theorems \ref{introthm_unconditional_symmetric_powers} and \ref{introthm_conditional_symmetric_powers} of the introduction from the results of the previous section, we therefore need an additional argument. As a warm-up, we first review the description of those $\pi$ which are indeed algebraic. The arguments that follow will extend those given in \cite[\S 7]{Clo17}. 

The local Langlands correspondence $\rec_\bR$ for $\GL_2(\bR)$ gives a bijection between the set of infinitesimal equivalence clases of essentially square-integrable irreducible admissible representations of $\GL_2(\bR)$ and the set of conjugacy classes of continuous irreducible representations $W_\bR \to \GL_2(\bC)$. Each such representation of $W_\bR$ is isomorphic to a unique one of the form 
\[ \rho_{s, a} = \Ind_{\bC^\times}^{W_\bR} \chi_{s, a}\]
for $(s, a) \in \bC \times \bZ_{\geq 1}$, and where $\chi_{s, a}(r e^{i \theta}) = r^s e^{i a \theta}$. The corresponding representation of $\GL_2(\bR)$ is the unique essentially square-integrable subquotient of the unitary induction of the character 
\[ (t_1, t_2) \mapsto |t_1|^{(s+a)/2} \operatorname{sgn}(t_1)^{a+1} |t_2|^{(s-a)/2} \]
of the maximal torus $T_2(\bR) = \bR^\times \times \bR^\times$ of $\GL_2(\bR)$.  

Suppose given integers $k_v \in \bZ_{\geq 0}$ for each place $v | \infty$ of $F^+$ and $w \in \bR$. Cuspidal Hilbert modular forms of weights $((k_v+2)_{v | \infty}, w)$ may be lifted to cuspidal automorphic forms. In particular, \cite[Proposition 5.2.3]{Oht83} explains how a cuspidal Hilbert modular eigenform of weights $((k_v+2)_{v | \infty}, w)$ determines a cuspidal automorphic representation $\pi$ of $\GL_2(\bA_{F^+})$ such that for each place $v | \infty$ of $F^+$, $\rec_{F_v}(\pi_v) = \rho_{-s, k_v+1}$ (and conversely). Repeating the computations of \cite[pp. 91--92]{Clo90} we see that $\pi$ is regular algebraic if and only if $w$ is an integer and $k_v \equiv w \text{ mod }2$ for each place $v | \infty$; and that $\pi$ admits a character twist which is regular algebraic if and only if the parity of $k_v$ is independent of $v | \infty$.

Here then is the formulation of symmetric power functoriality for a more general class of automorphic representations of $\GL_2(\bA_{F^+})$ than the RAESDC ones considered in the previous section.
\begin{conjecture}[$\mathrm{SP}'_n$]
Let $n \geq 1$. Let $F^+$ be a totally real field, and let $\pi$ be a cuspidal automorphic representation of $\GL_2(\bA_{F^+})$ without CM, such that $\pi_\infty$ is essentially square-integrable. Then there exists a cuspidal automorphic representation $\Pi$ of $\GL_{n}(\bA_{F^+})$ such that for each place $v$ of $F^+$, we have
\[ \rec_{F^+_v}(\Pi_v) \cong \Sym^{n-1} \circ \rec_{F^+_v}(\pi_v). \]
\end{conjecture}
Our goal will be to prove the following result.
\begin{theorem}\label{thm_alg_implies_W_alg}
Conjecture $\mathrm{SP}_n$ implies Conjecture $\mathrm{SP}'_n$.
\end{theorem}
Together with Corollary \ref{cor_unconditional_reg_alg_powers}, this theorem implies Theorems \ref{introthm_unconditional_symmetric_powers} and \ref{introthm_conditional_symmetric_powers} of the introduction. The rest of this section is devoted to its proof. Let $\pi$ be a cuspidal automorphic representation of $\GL_2(\bA_{F^+})$ without CM such that $\pi_\infty$ is essentially square-integrable, and fix $n \geq 1$. Assuming Conjecture $\mathrm{SP}_n$, we will show $\Sym^{n-1} \pi$ exists in the sense of Conjecture $\mathrm{SP}'_n$. 

We first quote what is proved in \cite[\S 7]{Clo17}, assuming $\mathrm{SP}_n$. 
\begin{lemma}\label{lem_sym_power_over_F}
Let $F / F^+$ be a quadratic CM extension.  Let $\pi_F$ denote the base change of $\pi$ to $F$.
\begin{enumerate}
    \item There exists a continuous character $\psi : F^\times \backslash \bA_F^\times \to \bC^\times$ such that $\pi_F \otimes \psi^{-1}$ is RACSDC.
    \item There exists a RACSDC automorphic representation $\Pi_1$ of $\GL_n(\bA_F)$ such that for any isomorphism $\iota : \overline{\bQ}_p \to \bC$, we have $r_\iota(\Pi_1) \cong \Sym^{n-1} r_\iota(\pi_F \otimes \psi^{-1})$. 
\end{enumerate}
\end{lemma}
\begin{proposition}\label{prop_nearly_there}
There exists a cuspidal automorphic representation $\Pi$ of $\GL_n(\bA_{F^+})$ satisfying the following properties:
\begin{enumerate}
    \item For any finite place $v$ of $F^+$, $\rec_{F^+_v}(\Pi_v) \cong \Sym^{n-1} \rec_{F^+_v}(\pi_v)$.
    \item For each infinite place $v$ of $F^+$, $\rec_{F^+_v}(\Pi_v)|_{W_{F_v}} \cong \Sym^{n-1} \rec_{F^+_v}(\pi_v)|_{W_{F_v}}$.
\end{enumerate}
\end{proposition}
The existence of a $\Pi$ having the first property is part of the content of \cite[Theorem 7.1]{Clo17}; the proof of this in \emph{loc. cit.} appears incomplete, so we give a more detailed explanation here as well as a description of the lifting at the infinite place. 
\begin{proof}
We use an automorphic version of the patching argument of \cite[Proposition 4.3.1]{Bla93}. Let $F / F^+$ be a quadratic CM extension and let $\psi, \Pi_1$ be as in the statement of Lemma \ref{lem_sym_power_over_F}. Then we have 
\[ \rec_{F_w}( \Pi_{1, w}) \cong \Sym^{n-1} \rec_{F_w}( (\pi_F \otimes \psi^{-1})_w ) \]
for every place $w$ of $F$. Indeed, if $w$ is a finite place, then this follows from local-global compatibility for the representations $r_\iota(\Pi_1)$, $r_\iota(\pi_F \otimes \psi^{-1})$. If $w$ is an infinite place then $\rec_{F_w}(\Pi_{1, w})$ can be read off from the Hodge--Tate weights of $r_\iota(\Pi_1)$ (use the recipe of \cite[Theorem 2.2]{Clo14} -- here we are using the fact that $\Pi_{1, w}$ is essentially tempered), so the desired result holds also at the infinite places. Twisting by a character, we find that we also have
\[  \rec_{F_w}( (\Pi_{1} \otimes \psi^{n-1})_w) \cong \Sym^{n-1} \rec_{F_w}( \pi_{F, w} ) \]
for every place $w$ of $F$. In particular, $\Pi_{1} \otimes \psi^{n-1}$ is stable under the action of $\Gal(F / F^+)$ so by \cite[Ch. 3, Theorem 4.2]{Art89} there exists a cuspidal automorphic representation $\Pi(F)$ of $\GL_n(\bA_{F^+})$ with base change $\Pi_1 \otimes \psi^{n-1}$, determined up to twist by the quadratic Hecke character associated to the extension $F / F^+$. For any place $w$ of $F$ lying above a place $v$ of $F^+$, we have
\begin{equation}\label{eqn_Pi_F_w} \rec_{F_w}(\Pi(F)_{F, w} ) \cong \Sym^{n-1} \rec_{F^+_v}( \pi_{v} )|_{W_{F_w}}. \end{equation}
In particular, if $v$ splits in $F$ then we have
\begin{equation}\label{eqn_Pi_F_v} \rec_{F^+_v}(\Pi(F)_v) \cong \Sym^{n-1} \rec_{F^+_v}( \pi_{v} ). \end{equation}
We now choose an infinite sequence $F_0, F_1, F_2 \dots$ of CM quadratic extensions of $F^+$ with the following properties:
\begin{itemize}
    \item For each $0 \leq i < j < k$, $[F_i F_j F_k : F^+] = 8$.
    \item For each finite place $v$ of $F^+$, there exists $i > 0$ such that $v$ splits in $F_i$.
\end{itemize}
We can construct such a sequence inductively by enumerating the finite places $v_1, v_2, \dots$ of $F^+$, choosing $F_0$ arbitrarily, and in general choosing $F_n$ so that $v_n$ splits in $F_n$ and also so that for each quadratic subfield $M$ of $F_0 \dots F_{n-1} / F^+$, there exist a place $v$ of $F^+$ which is inert in $M$ and split in $F_n$. For any $i \geq 0$, we write $\epsilon_i$ for the quadratic Hecke character associated to the extension $F_i / F^+$.

For each $i \geq 0$ we are given a cuspidal automorphic representation $\Pi(F_i)$ of $\GL_n(\bA_{F^+})$. For any $i < j$, the base change representations $\Pi(F_i)_{F_i F_j}$ and $\Pi(F_j)_{F_i F_j}$ are cuspidal and isomorphic. Indeed, they are isomorphic by (\ref{eqn_Pi_F_w}). To show they are cuspidal, it's enough to show that $(\Pi(F_i)_{F_i} \otimes \psi^{1-n})_{F_i F_j}$ is cuspidal. However, $\Pi(F_i)_{F_i} \otimes \psi^{1-n}$ is RACSDC and $r_\iota(\Pi(F_i)_{F_i} \otimes \psi^{1-n})|_{G_{F_i F_j}}$ is irreducible (as the Zariski closure of its image contains a conjugate of the principal $\SL_2$, by \cite[Theorem 7.3]{Clo17} and \cite[Example 2.34]{New22}). It is therefore enough to know that if $\sigma$ is a RACSDC automorphic representation of $\GL_n(\bA_{F_i})$ and $r_\iota(\sigma)|_{G_{F_i F_j}}$ is irreducible, then $\sigma_{F_iF_j}$ is cuspidal. This is true, because otherwise \cite[Ch. 3, Theorem 4.2]{Art89} would imply that $\sigma$ is isomorphic to its twist by the quadratic Hecke character associated to the extension $F_i F_j / F_i$, and therefore that $r_\iota(\sigma)$ is induced from $G_{F_i F_j}$ and reducible on restriction to this subgroup, a contradiction. 

Applying \cite[Ch. 3, Theorem 4.2]{Art89} again, we see that $\Pi(F_i)_{F_j}$ is either isomorphic to $\Pi(F_j)_{F_j}$ or differs from it by quadratic twist by the Hecke character associated to the extension $F_i F_j / F_j$. We define $\Sigma_{ij} = \Pi(F_i)$ in the first instance and $\Sigma_{ij} = \Pi(F_i) \otimes \epsilon_i$ in the second. Then in either case we have $\Sigma_{ij, F_i} = \Pi(F_i)_{F_i}$ and $\Sigma_{ij, F_j} = \Pi(F_j)_{F_j}$. In particular, for any place $v$ of $F^+$ which splits in $F_j$, we have
\begin{equation}\label{eqn_local_component_at_split_place} \rec_{F^+_v}( \Sigma_{ij, v} ) \cong \Sym^{n-1} \rec_{F^+_v}(\pi_v). \end{equation}
Now consider an arbitrary $0 < i < j$. We have $\Sigma_{0i, F_0} = \Sigma_{0j, F_0}$, so there is $n_0 \in \{ 0, 1 \}$ such that $\Sigma_{0i} \cong \Sigma_{0j} \otimes \epsilon_0^{n_0}$. We have $\Sigma_{0i, F_i} = \Sigma_{ij, F_i}$, so there is $n_i \in \{ 0, 1 \}$ such that $\Sigma_{0i} \cong \Sigma_{ij} \otimes \epsilon_i^{n_i}$. We have $\Sigma_{0j, F_j} = \Sigma_{ij, F_j}$, so there is $n_j \in \{0, 1 \}$ such that $\Sigma_{0j} \cong \Sigma_{ij} \otimes \epsilon_j^{n_j}$. Putting these relations together, we have 
\[ \Sigma_{0i} \cong \Sigma_{0i} \otimes (\epsilon_0^{n_0} \epsilon_i^{n_i} \epsilon_j^{n_j}). \]
The base change of $\Sigma_{0i}$ to $F_i F_j F_k$ is cuspidal (by the same argument as before), so by \cite[Ch. 3, Theorem 4.2]{Art89} the character $\epsilon_0^{n_0} \epsilon_i^{n_i} \epsilon_j^{n_j}$ must be trivial. By hypothesis the map 
\[ \Gal(F_0F_iF_j / F^+) \to \Gal(F_0 / F^+) \times \Gal(F_i / F^+) \times \Gal(F_j / F^+) \]
is an isomorphism, so this is possible only if $n_0 = n_i = n_j = 0$ and the representations $\Sigma_{0i}$ (for any $i > 0$) are all isomorphic. We set $\Pi = \Sigma_{01} = \Sigma_{0i}$.

To complete the proof, we need to explain why $\Pi$ has the required properties. If $v$ is a finite place, we choose $F_i$ such that $v$ splits in $F_i$. Then (\ref{eqn_local_component_at_split_place}) shows that $\Pi_v$ has the expected form. If $v$ is an infinite place, then the desired property follows already from (\ref{eqn_Pi_F_w}). 
\end{proof}
Let $\Pi$ be the representation of Proposition \ref{prop_nearly_there}. To complete the proof of Theorem \ref{thm_alg_implies_W_alg}, we just need to check that $\rec_{F^+_v}(\Pi_v)$ has the expected form when $v$ is an infinite place of $F^+$. Let $v$ be an infinite place and identify $F_v = \bC$. Let $\rec_{F^+_v}(\pi_v) = \rho_{s, a}$ for some $s \in \bC, a \geq 1$. We have
\begin{equation}\label{eqn_pars_equal_restriction_to_C} \rec_{F^+_v}(\Pi_v)|_{\bC^\times} \cong \Sym^{n-1} \rec_{F^+_v}(\pi_v)|_{\bC^\times}. \end{equation}
Suppose first that $n = 2m$ is even. Then 
\[ \Sym^{n-1} \rec_{F^+_v}(\pi_v) = \oplus_{i=1}^m \rho_{(n-1)s, (2i-1)a}, \]
and (\ref{eqn_pars_equal_restriction_to_C}) implies that we must have $\rec_{F^+_v}(\Pi_v)\cong \Sym^{n-1} \rec_{F^+_v}(\pi_v)$ (because Frobenius reciprocity implies that any continuous representation of $W_\bR$ whose restriction to $\bC^\times$ is isomorphic to $\Sym^{n-1} \rec_{F^+_v}(\pi_v)|_{\bC^\times}$ is already isomorphic to $\Sym^{n-1} \rec_{F^+_v}(\pi_v)$ on the whole of $W_\bR$). 

Next suppose that $n = 2m + 1$ is odd. Then we compute
\[ \Sym^{n-1} \rec_{F^+_v}(\pi_v)= | \cdot |^{(n-1)s} \operatorname{sgn}^{am} \oplus \left( \oplus_{i=1}^m \rho_{(n-1)s, 2 i a} \right) \]
and (\ref{eqn_pars_equal_restriction_to_C}) implies that we have
\[ \rec_{F^+_v}(\Pi_v) = | \cdot |^{(n-1)s} \operatorname{sgn}^{\epsilon} \oplus \left( \oplus_{i=1}^m \rho_{(n-1)s, 2 i a} \right) \]
for some $\epsilon \in \{0, 1 \}$. We need to show that $(-1)^\epsilon = (-1)^{am}$. Let $\omega_\pi$, $\omega_{\Pi}$ denote the central characters of $\pi, \Pi$, respectively. Then $\omega_{\Pi} = \omega_\pi^{m(2m+1)}$, as can be checked at finite unramified places. We finally compute
\[ (-1)^{\epsilon+m} = \omega_{\Pi, v}(-1) = \omega_{\pi, v}(-1)^{m(2m+1)} = (-1)^{m(a+1)}. \]
This completes the proof of Theorem \ref{thm_alg_implies_W_alg}. \qed

	\bibliographystyle{alpha}
	\bibliography{LR}

	\end{document}